\documentclass[reqno,10pt,centertags,draft]{amsart}
\usepackage{amsmath,amsthm,amscd,amssymb,latexsym,upref,stmaryrd}
\usepackage{amsfonts,mathrsfs}
\newcommand{\bbC}{{\mathbb{C}}}

\newcommand{\bbN}{{\mathbb{N}}}

\newcommand{\bbR}{{\mathbb{R}}}

\newcommand{\cB}{{\mathcal B}}
\newcommand{\cC}{{\mathcal C}}
\newcommand{\cD}{{\mathcal D}}

\newcommand{\cH}{{\mathcal H}}

\newcommand{\cM}{{\mathcal M}}

\newcommand{\cS}{{\mathcal S}}

\newcommand{\cV}{{\mathcal V}}
\newcommand{\cW}{{\mathcal W}}
\newcommand{\cX}{{\mathcal X}}

\newcommand{\gQ}{\mathfrak Q}

\newcommand{\gq}{\mathfrak q}

\DeclareMathOperator{\supp}{supp}

\DeclareMathOperator{\dom}{dom}

\DeclareMathOperator{\tr}{tr}

\DeclareMathOperator*{\nlim}{n-lim}
\DeclareMathOperator*{\slim}{s-lim}
\DeclareMathOperator*{\wlim}{w-lim}

\renewcommand{\ln}{\text{\rm ln}}

\newcommand{\dott}{\,\cdot\,}
\newcommand{\abs}[1]{\lvert#1\rvert}

\newcommand{\norm}[1]{\lVert#1\rVert}

\newcommand{\loc}{\text{\rm{loc}}}

\newcommand{\beq}{\begin{equation}}
\newcommand{\enq}{\end{equation}}

\newcommand{\Om}{\Omega}
\newcommand{\dOm}{{\partial\Omega}}

\newcommand{\LOm}{L^2(\Om;d^nx)}
\newcommand{\LdOm}{L^2(\dOm;d^{n-1} \omega)}

\newcommand{\ga}{\gamma}

\newcommand{\sgn}{{\textrm{sgn}}}

\newcommand{\no}{\notag}
\newcommand{\lb}{\label}
\newcommand{\f}{\frac}

\newcommand{\ol}{\overline}

\newcommand{\wti}{\widetilde}
\newcommand{\Oh}{O}

\newcommand{\hatt}{\widehat}
\newcommand{\bi}{\bibitem}
\newcommand{\prodm}{\mu \otimes \mu}

\renewcommand{\le}{\leqslant}

\let\geq\geqslant
\let\leq\leqslant

\makeatletter
\def\theequation{\@arabic\c@equation}

\allowdisplaybreaks 
\numberwithin{equation}{section}

\newtheorem{theorem}{Theorem}[section]

\newtheorem{lemma}[theorem]{Lemma}
\newtheorem{corollary}[theorem]{Corollary}
\newtheorem{definition}[theorem]{Definition}
\newtheorem{hypothesis}[theorem]{Hypothesis}
\newtheorem{example}[theorem]{Example}
\theoremstyle{remark}
\newtheorem{remark}[theorem]{Remark}

\begin{document}

\title[Heat Kernel Bounds for Elliptic PDEs in Divergence Form]{Heat Kernel Bounds for Elliptic Partial Differential Operators in Divergence Form with Robin-Type Boundary Conditions}

\author[F.\ Gesztesy]{Fritz Gesztesy}
\address{Department of Mathematics,
University of Missouri, Columbia, MO 65211, USA}
%\email{\mailto{gesztesyf@missouri.edu}}
\email{gesztesyf@missouri.edu}
%\urladdr{\url{http://www.math.missouri.edu/personnel/faculty/gesztesyf.html}}
\urladdr{http://www.math.missouri.edu/personnel/faculty/gesztesyf.html}

\author[M.\ Mitrea]{Marius Mitrea}
\address{Department of Mathematics,
University of Missouri, Columbia, MO 65211, USA}
\email{mitream@missouri.edu}
%\urladdr{\href{http://www.math.missouri.edu/personnel/faculty/mitream.html}
%{http://www.math.missouri.edu/personnel/faculty/mitream.html}}
\urladdr{http://www.math.missouri.edu/personnel/faculty/mitream.html}

\author[R.\ Nichols]{Roger Nichols}
\address{Mathematics Department, The University of Tennessee at Chattanooga, 
415 EMCS Building, Dept. 6956, 615 McCallie Ave, Chattanooga, TN 37403, USA}
\email{Roger-Nichols@utc.edu}
\urladdr{http://rogeranichols.jimdo.com/}

%\dedicatory{}
\thanks{Work of M.\,M.\ was partially supported by the Simons Foundation Grant $\#$\,281566. }
\thanks{To appear in {\it Journal d'Analyse Math\'ematique.}}
\date{\today}

%%%%%%%%%%%%%%%%%%%%%%%%%%%%%%%%%%%%%%
%%%%%%%%%%%%%%%%%%%%%%%%%%%%%%%%%%%%%%
\begin{abstract} 
One of the principal topics of this paper concerns the realization of self-adjoint operators 
$L_{\Theta, \Om}$ in $L^2(\Om; d^n x)^m$, $m, n \in \bbN$, associated with divergence form 
elliptic partial differential expressions $L$ with (nonlocal) Robin-type boundary conditions in bounded 
Lipschitz domains $\Om \subset \bbR^n$. In particular, we develop the theory 
in the vector-valued case and hence focus on matrix-valued differential expressions $L$ which act as  
$$
Lu = - \biggl(\sum_{j,k=1}^n\partial_j\bigg(\sum_{\beta = 1}^m 
a^{\alpha,\beta}_{j,k}\partial_k u_\beta\bigg)
\bigg)_{1\leq\alpha\leq m},  \quad u=(u_1,\dots,u_m). 
$$
The (nonlocal) Robin-type boundary conditions are then of the form 
$$
\nu \cdot A D u + \Theta \big[u\big|_{\partial \Om}\big] = 0 \, \text{ on $\partial \Om$},
$$ 
where $\Theta$ represents an appropriate operator acting on Sobolev spaces associated with the 
boundary $\partial \Om$ of $\Om$, $\nu$ denotes the outward pointing normal unit vector on 
$\partial\Om$, and $Du:=\bigl(\partial_j u_\alpha\bigr)_{\substack{1\leq\alpha\leq m\\ 1\leq j\leq n}}$. 

Assuming $\Theta \geq 0$ in the scalar case $m=1$, we prove Gaussian heat kernel bounds for 
$L_{\Theta, \Om}$ by employing positivity preserving arguments for the associated semigroups and 
reducing the problem to the corresponding Gaussian heat kernel bounds for the case of Neumann 
boundary conditions on $\partial \Om$. We also discuss additional zero-order potential coefficients $V$ 
and hence operators corresponding to the form sum $L_{\Theta, \Om} + V$.
\end{abstract}
%%%%%%%%%%%%%%%%%%%%%%%%%%%%%%%%%%%%%%
%%%%%%%%%%%%%%%%%%%%%%%%%%%%%%%%%%%%%%

\maketitle

{\scriptsize{\tableofcontents}}

%%%%%%%%%%%%%%%%%%%%%%%%%%%%%%%%%%%%%%
%%%%%%%%%%%%%%%%%%%%%%%%%%%%%%%%%%%%%%
\section{Introduction}  \lb{s1}
%%%%%%%%%%%%%%%%%%%%%%%%%%%%%%%%%%%%%%
%%%%%%%%%%%%%%%%%%%%%%%%%%%%%%%%%%%%%%

In a nutshell, this paper is devoted to a new class of self-adjoint realizations $L_{\theta, \Om}$ in 
$L^2(\Om; d^n x)$ of elliptic partial differential expressions in divergence form, 
\begin{equation} 
L = - \sum_{j,k=1}^n\partial_ja_{j,k}\partial_k,   \lb{1.0} 
\end{equation} 
on bounded Lipschitz domains 
$\Om \subset \bbR^n$, $n\geq 2$, with Robin boundary conditions of the form 
$[\nu \cdot A \nabla u + \theta u]|_{\partial \Om} = 0$. (Here $\nu$ denotes the outward pointing 
normal unit vector and $\theta$ is a suitable function on the boundary $\partial \Om$ of $\Om$.) Following 
\cite{GM09}, we put particular emphasis on developing a theory of nonlocal Robin boundary conditions where the function $\theta$ on $\partial \Om$ is replaced by a suitable operator $\Theta$ acting in 
$L^2(\partial \Om; d^{n-1} \omega)$, with $d^{n-1} \omega$ representing the surface measure on 
$\partial \Om$. (More precisely, $\Theta$ acts in appropriate Sobolev spaces on the boundary of 
$\Om$, cf.\ Section \ref{s3}.) The resulting self-adjoint operator in $L^2(\Om; d^n x)$ 
is then denoted by $L_{\Theta, \Om}$ and we study its resolvent and semigroup, proving a Gaussian 
heat kernel bound and a corresponding bound for the Green's function of $L_{\Theta, \Om}$ on the 
basis of positivity preserving arguments. 

The corresponding case of Dirichlet boundary conditions, and similarly (although, to a somewhat lesser extent), the one of Neumann boundary conditions on $\partial \Om$ has been extensively studied in the literature. An authoritative  survey of this area until about 1989 was written by Davies \cite{Da89}, and 
newer developments since then were treated by Ouhabaz \cite{Ou05} in 2005. For the study of elliptic 
operators on Lie groups and the associated semigroup kernels we refer to the monograph by Robinson 
\cite{Ro91}. A thorough study of the heat equation and the heat kernel of the Laplacian on Riemannian 
manifolds has been undertaken by Grigor'yan \cite{Gr09}, and very recently, Dirichlet and Neumann 
heat kernels and associated two-sided bounds were developed for inner uniform domains by 
Gyrya and Saloff-Coste \cite{GS11}. 

The case of Robin-type boundary conditions on the other hand, received much less attention. The 
latter experienced a boost due to the seminal paper by Arendt and ter Elst \cite{AtE97} in 1997, and 
the past fifteen years have seen a number of interesting developments in this area.  

Before specializing to matters related to heat kernel bounds, we digress a bit and 
mention some of the literature devoted to Robin boundary conditions: A probabilistic approach to 
Robin-type (or third) boundary value problems for $L = \Delta + b \cdot \nabla + q$ was undertaken by Papanicolaou \cite{Pa90} for bounded $C^3$-domains $\Om$; the case of the Robin problem for the Laplacian on general domains, including fractals, was treated by Bass, Burdzy, and Chen \cite{BBC08}. 
The short-time asymptotics of the heat kernel for the Laplacian with Robin boundary conditions on a 
smooth domain was discussed by Zayed \cite{Za98}. Inequalities between Robin and Dirichlet 
eigenvalues for the Laplacian on bounded Lipschitz domains $\Om$ were derived by 
Filonov \cite{Fi05}. (The paper by Filonov motivated two of us to extend such results to the case of 
nonlocal Robin Laplacians which lead to their introduction in \cite{GM09}.) An isoperimetric inequality 
similar to the Faber--Krahn inequality for Robin Laplacians (and more generally, in the case of mixed Dirichlet and Robin boundary conditions on parts of the boundary) on bounded Lipschitz domains was 
established by Daners \cite{Da06}; in this isoperimetric context we also mention work by Kennedy 
\cite{Ke08}, \cite{Ke09}, \cite{Ke10}, \cite{Ke10a}. Uniqueness in the Faber--Krahn inequality for Robin Laplacians was derived by Daners and Kennedy \cite{DK07}, and for a recent alternative approach to the Faber--Krahn inequality for $p$-Laplacians with Robin boundary conditions we refer to Bucur and Daners \cite{BD10}. Domain perturbations for elliptic equations with Robin boundary conditions were studied by Dancer and Daners \cite{DD97}; domain 
monotonicity for the principal eigenvalue of Robin Laplacians was investigated by Giorgi and Smits 
\cite{GS05}, \cite{GS08}. Spectral stability of Robin Laplacians and Laplacians with mixed boundary conditions on bounded domains (Lipschitz and more general than that) defined in terms of quadratic forms 
were treated by Barbatis, Burenkov, Lamberti, and Lanza de Cristoforis in various collaborations 
\cite{BBL10}, \cite{BLL08}, \cite{BL08}. Eigenvalue asymptotics for Robin 
Laplacians was studied by Daners and Kennedy \cite{DK10}. Schatten--von Neumann properties 
of resolvent differences and the asymptotic behavior of singular values of nonlocal Robin Laplacians on bounded smooth domains were investigated by  Behrndt, Langer, Lobanov, Lotoreichik, and Popov \cite{BLLLP10}; in the case of local Robin boundary conditions these results were subsequently improved 
and extended to strongly elliptic symmetric partial differential operators by Grubb \cite{Gr11} (see 
also \cite{Gr11a}). Nodal 
line domains for Laplacians with (local) Robin boundary conditions on bounded Lipschitz domains 
were studied in \cite{Ke11}. The Robin boundary value problem for the two-dimensional Laplacian 
in domains with cusps was studied by Kamotski and Maz'ya \cite{KM11}. Hardy inequalities for Robin Laplacians were derived by Kova{\v r}{\' i}k and Laptev \cite{KL12}. The problem of minimizing the 
$n$th eigenvalue of the Robin Laplacian was recently investigated by Antunes, Freitas, and Kennedy 
\cite{AFK12}.  

Returning to Gaussian heat kernel bounds for uniformly elliptic partial differential operators, the 
literature associated with Dirichlet boundary conditions is so enormous that one would be hard-pressed 
to do it justice in the introduction to a manuscript of the present scale. Instead, we refer to \cite{Da89} and \cite{Ou05} which nicely survey the 
state of literature until about 1989 and 2005, respectively. Thus, we will now focus on studies in connection with Robin boundary conditions, the principal topic of this paper. We already mentioned the influential paper 
by Arendt and ter Elst \cite{AtE97} on Gaussian heat kernel bounds for $L_{\theta, \Om}$ on bounded Lipschitz domains $\Om$ with local Robin boundary conditions indexed by the function $\theta$ on 
$\partial \Om$. The authors study the case of real, not necessarily symmetric matrices 
$a_{j,k} \in L^{\infty}(\Om; d^n x)$, $1 \leq j,k \leq n$, satisfying a uniform ellipticity condition, assuming 
$0 \leq \theta \in L^{\infty}(\partial \Om; d^{n-1} \omega)$, and permitting lower-order (complex-valued) coefficients under some smoothness hypotheses. Subsequently, Daners \cite{Da00a}, in the case of  
real-valued coefficients, extended their results by removing the smoothness assumptions on the 
lower-order coefficients. Moreover, Ouhabaz \cite{Ou04} and \cite[Sect.\ 4.1, Ch.\ 6]{Ou05}, further extended these results by permitting complex-valued coefficients $a_{j,k}$ in $L$ which introduces 
a variety of challenges. One should also mention that the results developed by these authors also 
permit mixed boundary conditions in the sense that one may have Dirichlet boundary conditions at 
a part of the boundary and Neumann or Robin boundary conditions on the rest of the boundary.  
In addition, local Robin-type problems were studied by Daners \cite{Da00} on arbitrary bounded and open 
sets $\Om \subset \bbR^n$; he proved positivity improving of the underlying semigroups if $\Om$ is connected and $0 < \theta \in L^{\infty} (\partial \Om; d^{n-1}\omega)$. The case of arbitrary bounded 
and open sets $\Om \subset \bbR^n$ was also treated by Arendt and Warma \cite{AW03}, 
\cite{AW03a}, \cite{Wa02} (see also \cite{Wa06}), who developed  
a general theory with $\theta$ replaced by a (positive) measure on the boundary, and who proved  
Gaussian heat kernel bounds in this context. In particular, a characterization of generalized local Robin 
Laplacians whose semigroups satisfy domination by the Neumann semigroup (cf.\ \eqref{3.19} for details) 
was provided in \cite{AW03a}. Positivity preserving semigroups in the case of bounded Lipschitz 
domains $\Om$ and sign indefinite 
$\theta \in L^{\infty} (\partial \Om; d^{n-1}\omega)$ were established by Daners \cite{Da09} (and 
again mixed boundary conditions were permitted). For an illuminating survey on semigroups and 
heat kernel estimates we also refer to \cite{Ar04}. 

Next, we briefly recall some facts concerning positivity preserving integral operators. For 
simplicity, we describe the concrete case of linear operators in the Hilbert space $L^2(\Om; d^n x)$ 
($\Om \subset \bbR^n$ a bounded Lipschitz domain)  
and refer to Section \ref{s2} for the abstract setting: A bounded operator 
$A$ in $L^2(\Om; d^n x)$ is called {\it positivity preserving} if 
\begin{equation} 
0 \leq f \in L^2(\Om; d^n x) \backslash \{0\} \, \text{ implies } \, A f \geq 0 \, 
\text{ a.e.}      \lb{1.1}
\end{equation} 
This will be denoted by 
$A \succcurlyeq 0$ (resp., by $0 \preccurlyeq A$). More generally, $A \succcurlyeq B$ 
(resp., $B \preccurlyeq A$) then abbreviates that $A-B$ is positivity preserving. A key result in this context then concerns the fact that a bounded integral operator $A$ 
in $L^2(\Om; d^n x)$ is positivity preserving if and only if its integral kernel $A(\cdot,\cdot)$ is 
nonnegative a.e.\ on $\Om \times \Om$. In particular, if $A$ and $B$ are bounded integral operators in 
$L^2(\Om; d^n x)$, then 
\begin{equation}
0 \preccurlyeq B \preccurlyeq A \, \text{ if and only if } \, 0 \leq B(\cdot,\cdot) \leq A(\cdot,\cdot) 
\, \text{ a.e.\ on $\Om \times \Om$.}    \lb{1.2} 
\end{equation}
In particular, $0 \preccurlyeq B \preccurlyeq A$ yields a bound on the integral kernel of $B$ in terms of that of $A$. 

Our approach to proving a Gaussian heat kernel bound for the class of operators $L_{\Theta, \Om}$ 
then rests on the following strategy: Assuming $\Theta \geq 0$, we will prove that the semigroups 
$e^{- t L_{\Theta, \Om}}$, $t \geq 0$, are positivity preserving, and so are the differences 
$e^{- t L_{N, \Om}} - e^{- t L_{\Theta, \Om}}$, $t \geq 0$, where $L_{N, \Om}$ denotes the case 
of Neumann boundary conditions (i.e., the special case $\Theta = 0$). Thus, assuming 
$\Theta \geq 0$, we will prove 
\begin{equation}
0 \preccurlyeq e^{- t L_{\Theta, \Om}} \preccurlyeq e^{- t L_{N, \Om}}, \quad t\geq 0,   \lb{1.3}
\end{equation}
and hence obtain the Gaussian heat kernel bound for $L_{\Theta, \Om}$ by means of \eqref{1.2} and 
the corresponding known Gaussian heat kernel bounds for $L_{N, \Om}$. 

We briefly turn to a description of the contents of each section: Section \ref{s2} is devoted to an abstract 
discussion of positivity preserving (and improving) integral operators with special emphasis on resolvents and semigroups. Much of the material in this section revolves about an important early paper by 
Davies \cite{Da73} and the remarkable papers by Bratteli, Kishimoto, and Robinson \cite{BKR80} and Kishimoto and Robinson \cite{KR80}. 

Section \ref{s3} develops in detail the theory of nonlocal Robin-type operators $L_{\Theta, \Om}$ on 
bounded Lipschitz domains $\Om \subset \bbR^n$. In fact, we will go a step further and treat the 
$m \times m$ matrix-valued case, $m\in \bbN$, where $L$ acts like 
\begin{equation}\label{1.4}
Lu = - \biggl(\sum_{j,k=1}^n\partial_j\bigg(\sum_{\beta = 1}^m 
a^{\alpha,\beta}_{j,k}\partial_k u_\beta\bigg)
\bigg)_{1\leq\alpha\leq m},\quad u=(u_1,\dots,u_m). 
\end{equation}
While we will naturally follow the outline of an earlier treatment in the special scalar case $m=1$ of the Laplacian, $L = - \Delta$ in \cite{GM09}, we emphasize that the presence of the tensor coefficient 
$A =\bigl(a_{j,k}^{\alpha,\beta}\bigr)_{\substack{1\leq\alpha,\beta\leq m \\ 1\leq j,k\leq n}} 
\in L^{\infty}(\Om; d^nx)^{m\times m}$ requires a careful re-examination and extension 
of the earlier arguments in \cite{GM09}. We also note that our results are of interest in the special case 
of local Robin boundary conditions where $\Theta$ corresponds to the operator $M_{\theta}$ of multiplication by the function $\theta$ defined on $\partial \Om$ as we do not have to assume that 
$\theta \in L^\infty (\partial \Om; d^{n-1} \omega)$, but are able to permit appropriate $L^p$-conditions 
on $\theta$.   

In our final Section \ref{s4} we then derive the Gaussian heat kernel bounds for $L_{\Theta, \Om}$ employing the positivity preserving (in fact, positivity improving) arguments as outlined in connection 
with \eqref{1.3}. This section concludes with a series of remarks that enhances the main result on 
Gaussian heat kernel bounds, including the addition of a nonnegative potential term 
$0 \leq V \in L^1_{\loc} (\Om; d^n x)$ to $L_{\Theta, \Om}$ using the method of forms, and appropriate negative and small form potential perturbations. 

Appendix \ref{sA} summarizes the principal results of Sobolev spaces on Lipschitz domains 
$\Om \subset \bbR^n$ and on their boundaries $\partial \Om$, and Appendix \ref{sB} is devoted to 
the basics of sesquilinear forms and their associated operators. Both appendices provide the basics for 
Section \ref{s3} and are included to guarantee a certain degree of self-containment of this paper. 
Appendix \ref{sC} recalls some bounds for heat kernels and Green's functions; in particular, 
we provide a streamlined approach to Green's function bounds including the case of dimension $n=2$.  

Finally, we briefly summarize some of the notation used in this paper: 
Let $\cH$ be a separable complex Hilbert space, $(\cdot,\cdot)_{\cH}$ the scalar product in $\cH$
(linear in the second argument), and $I_{\cH}$ the identity operator in $\cH$.

Next, if $T$ is a linear operator mapping (a subspace of) a Hilbert space into another, then 
$\dom(T)$ and $\ker(T)$ denote the domain and kernel (i.e., null space) of $T$. 
The closure of a closable operator $S$ is denoted by $\ol S$. 
The spectrum, essential spectrum, discrete spectrum, and resolvent set 
of a closed linear operator in a Hilbert space will be denoted by $\sigma(\cdot)$, 
$\sigma_{\rm ess}(\cdot)$, $\sigma_{\rm d}(\cdot)$, and $\rho(\cdot)$, respectively. 

The convergence in the strong operator topology (i.e., pointwise limits) will be denoted by $\slim$. 
Similarly, limits in the weak (resp., norm) topology are abbreviated by $\wlim$ (resp., by $\nlim$). 

% The Banach space of bounded (resp., compact) linear operators on $L^2(M; d\mu)$ is
% denoted by $\cB(L^2(M; d\mu))$ (resp., $\cB_{\infty}(L^2(M; d\mu))$). 
%, the analogous notation $\cB(\cX_1,\cX_2)$,
% will be used for bounded operators between two Banach spaces 
% $\cX_1$ and $\cX_2$. Moreover, $\cX_1\hookrightarrow \cX_2$ 
% denotes the continuous embedding of $\cX_1$ into $\cX_2$.
%modified Fredholm determinants are abbreviated by ${\det}_{p,\cH}(I_{\cH} + \cdot)$, 
%$p\in\bbN$ (the subscript $p$ being omitted in the trace class case $p=1$).
%The form sum (resp.\ difference) of two self-adjoint operators $A$ and $W$ will 
%be denoted by $A +_q W$ (resp., $A -_q W = A +_q (-W)$). 

The Banach spaces of bounded and compact linear operators on a separable complex Hilbert space 
$\cH$ are denoted by $\cB(\cH)$ and $\cB_\infty(\cH)$, respectively; the corresponding 
$\ell^p$-based trace ideals will be denoted by $\cB_p (\cH)$, $p>0$. The trace of trace class operators in $\cH$ is denoted by ${\tr}_{\cH}(\cdot)$.  The analogous notation 
$\cB(\cX_1,\cX_2)$, $\cB_\infty (\cX_1,\cX_2)$, etc., will be used for bounded and compact 
operators between two Banach spaces $\cX_1$ and $\cX_2$. 
Moreover, $\cX_1\hookrightarrow \cX_2$ denotes the continuous embedding
of the Banach space $\cX_1$ into the Banach space $\cX_2$. 

Given a $\sigma$-finite measure space, $(M,\cM,\mu)$, the product measure on 
$M \times M$ will be denoted by $\prodm$. Without loss of generality, we also denote the completion of the product measure space $(M\times M, \cM \otimes \cM, \prodm)$ by the same symbol and always work with this completion in the following. For brevity, the identity operator in 
$L^2(M; d\mu)$ is denoted by $I_M$, the scalar product and norm in $L^2(M; d\mu)$ are abbreviated 
by $(\dott,\dott)_M$ and $\|\dott \|_M$, whenever the underlying measure is understood. For a linear 
subspace $\cD$ of $L^2(M; d\mu)$, the cone of nonnegative elements in $\cD$ is denoted by $\cD_+$. 
Lastly, $\chi_{{}_S}$ denotes the characteristic function of the set $S$.

%%%%%%%%%%%%%%%%%%%%%%%%%%%%%%%%%%%%%%
%%%%%%%%%%%%%%%%%%%%%%%%%%%%%%%%%%%%%%
\section{On Positivity Preserving/Improving Integral Operators}  \lb{s2}
%%%%%%%%%%%%%%%%%%%%%%%%%%%%%%%%%%%%%%
%%%%%%%%%%%%%%%%%%%%%%%%%%%%%%%%%%%%%%

We start by recalling some basic facts on positivity preserving/improving operators. 
Throughout this paper we will make the following assumption.

%%%%%%%%%%
\begin{hypothesis} \lb{h2.1}
Let $(M,\cM,\mu)$ denote a $\sigma$-finite, separable measure space associated with a nontrivial 
measure $($i.e., $0 < \mu(M) \leq \infty$$)$. 
\end{hypothesis}
%%%%%%%%%%

Assuming Hypothesis \ref{h2.1}, $L^2(M; d\mu)$ represents the associated complex, separable Hilbert space (cf.\ \cite[Sect.\ 1.5]{BS88} and \cite[p.\ 262--263]{Jo82} for additional facts in this context).

The set of nonnegative elements $0 \leq f \in L^2(M; d\mu)$ (i.e., $f(x) \geq 0$ $\mu$-a.e.) is a cone 
in $L^2(M; d\mu)$, closed in the norm and weak topologies. 

%%%%%%%%%%
\begin{definition} \lb{d2.2} 
Let $A$ be a linear operator $A$ in $L^2(M; d\mu)$. Then $A$ is called {\it positivity preserving} 
$($resp., {\it positivity improving}$)$ if
\begin{equation}
0 \neq f \in \dom(A), \, f \geq 0 \text{ $\mu$-a.e.\ implies } \, A f \geq 0 \, 
\text{ $($resp., $Af > 0$$)$ $\mu$-a.e.}
\end{equation}
The operator $A$ is called {\it indecomposable} if $\{0\}$ and $L^2(M; d\mu)$ are the only 
closed subspaces of $L^2(M; d\mu)$ left invariant by $A$. 
\end{definition}
%%%%%%%%%%

Positivity preserving (resp., improving) of $A$ will be denoted by 
\begin{equation}
A \succcurlyeq 0 \, \text{ (resp., $A \succ 0$).}
\end{equation}
(or by $0 \preccurlyeq A$ (resp., $0 \prec A$)). More generally, if 
$A, B \in \cB\big(L^2(M; d\mu)\big)$, then  
\begin{equation}
A \succcurlyeq B \, \text{ (resp., $A \succ B)$}
\end{equation}
(or $B \preccurlyeq A$ (resp., $B \prec A$)), by definition, imply that 
$A - B$ is positivity preserving (resp., positivity improving). We note parenthetically, 
that in large parts of the pertinent literature, ``positivity preserving operators'' are just called ``positive'', 
but due to the obvious conflict of notation with the concept of positive operators in a spectral theoretic context, we follow the mathematical physics community which prefers the notion of positivity preserving. 

Moreover, if $A$ has the property that whenever $0 \neq f \in \dom(A)$ then also $|f| \in \dom(A)$ (in particular, if $\dom(A) = L^2(M; d\mu)$) one infers that  
\begin{equation}
\text{$A \succcurlyeq 0$ if and only if $|A f| \leq A |f|$ $\mu$-a.e.}    \lb{2.4} 
\end{equation}
(upon decomposing $f= f_+ - f_-$, $f_{\pm} = (|f| \pm f)/2$).  
Similarly, one concludes that  
\begin{align}
\begin{split}
& \text{$A \succcurlyeq 0$ (resp., $A \succ 0$) if and only if $(f, A g)_M \geq 0$
(resp., $(f, A g)_M >0$)}     \lb{2.5} \\
& \quad \text{for all $0 \leq f \in L^2(M; d\mu) \backslash \{0\}$, 
$0 \leq g \in \dom(A) \backslash \{0\}$,}
\end{split}
\end{align} 
using the fact that 
\begin{align}
\begin{split}
& g \in L^2(M; d\mu) \backslash \{0\} \, \text{ is nonnegative (resp., strictly positive) if 
and only if } \\
& \quad  (f,g)_M \geq 0 \, \text{ (resp., $(f,g)_M > 0$) for all 
$0 \leq f \in L^2(M; d\mu) \backslash \{0\}$.}
\end{split} 
\end{align}
This, in turn, follows from $\sigma$-finiteness of $M$, that is, 
$M = \bigcup_{n\in\bbN} M_n$, where $M_n \in \cM$, $\mu(M_n) < \infty$, 
defining $S=\{x \in M \, | \, g(x) < 0\}$, $S_n = S \cap M_n$, $n\in\bbN$, and from 
the fact that if $\mu(S_{n_0}) \neq 0$ for some $n_0 \in \bbN$ (recalling that $\mu(M) > 0$), 
\begin{equation}
0 \leq (\text{ resp., $<$ }) \; (\chi_{{}_{S_{n_0}}},g)_M 
= \int_M \chi_{{}_{S_{n_0}}}(x) g(x) \, d\mu (x) \leq 0. 
\end{equation}

Turning our attention to integral operators in $L^2(M; d\mu)$ with associated integral kernels 
$A(\cdot,\cdot)$ on $M \times M$, we assume that
\begin{equation}
\text{$A(\cdot,\cdot): M \times M \to \bbC$ is $\prodm$-measurable,} 
\end{equation}  
and introduce the integral operator $A$ generated by the integral kernel $A(\cdot,\cdot)$ 
as follows: 
\begin{align}
& (A f)(x) = \int_M A(x,y) f(y) \, d\mu(y) \, \text{ for $\mu$-a.e.\ $x\in M$,}    \no \\
& f \in \dom(A) = \bigg\{g \in L^2(M; d\mu) \,\bigg| \, A(x,\cdot) g(\cdot) \, \text{is integrable 
over $M$ for $\mu$-a.e.\ $x \in M$,} \no \\
& \hspace*{5cm} \text{and} \, \int_M A(\cdot,y) g(y) \, d\mu(y) \in L^2(M; d\mu)\bigg\}.   
\lb{2.7}
\end{align}

As shown in \cite[Theorem\ 11.1]{Jo82}, 
\begin{equation}
A \in \cB\big(L^2(M; d\mu)\big) \, \text{ whenever } \, \dom(A) = L^2(M; d\mu)  
\end{equation}
(as $A$ turns out to be closed in this case). In addition, we also note that by \cite[Theorem\ 11.2]{Jo82}, 
$|A(\cdot,\cdot)|$ generates a bounded integral operator in $L^2(M; d\mu)$ if and only if 
\begin{equation}
\int_{M \times M} |f(x) A(x,y) g(y)| \, d(\prodm)(x,y) < \infty, \quad f, g \in L^2(M; d\mu). 
\end{equation} 

For additional results on integral operators we refer, for instance, to \cite[Sect.\ 9.5]{Ed95}, \cite{HS78a}, \cite{Ko65}, \cite[Ch.\ 6]{We80}.  

Next, we state the following result (special cases of which appear to be well-known, but we know of no published proof at this instant).

%%%%%%%%%%
\begin{theorem} \lb{t2.3}
Assume Hypothesis \ref{h2.1} and suppose $($without loss of generality\,$)$ that $M$ can be written 
as $M = \bigcup_{n\in\bbN} M_n$, where $M_n \in \cM$, $\mu(M_n) < \infty$, $n\in\bbN$. Moreover, suppose that $A$ is an integral operator in $L^2(M; d\mu)$ generated by the  integral kernel 
$A(\cdot,\cdot)$ on $M \times M$ satisfying 
\begin{equation}
A(\cdot,\cdot) |_{M_m \times M_n} \in 
L^1(M_m \times M_n, \cM_m \otimes \cM_n; d(\mu_m \otimes \mu_n)), \quad 
m, n \in \bbN,
\end{equation}
where 
\begin{align}
\cM_n &= M_n \cap \cM = \{M_n \cap S \, | \, S \in \cM\} 
= \{ T \subseteq M_n \, | \, T \in \cM\},   \quad n \in \bbN,   \\
\mu_n &= \mu|_{M_n}, \quad n \in \bbN. 
\end{align}
Then $A$ is positivity preserving if and only if 
\begin{equation}
A(\cdot,\cdot) \geq 0 \;\, \prodm \text{-a.e.\ on } M\times M.    \lb{2.13} 
\end{equation} 
\end{theorem}
%%%%%%%%%%
\begin{proof}
Clearly, a positivity preserving integral operator of the type \eqref{2.7} must have a real-valued integral kernel $A(\cdot,\cdot)$. Sufficiency of the condition \eqref{2.13} 
for positivity preserving of $A$ then follows directly from the representation of $A$ in 
\eqref{2.7}. 

In order to prove necessity of the condition \eqref{2.13} it suffices to proof that positivity preserving of $A$ 
implies that 
\begin{equation}
A(\cdot,\cdot) |_{M_m \times M_n} \geq 0 \;\, \mu_m \otimes \mu_n \text{-a.e.\ on } 
M_m\times M_n, \quad m, n \in \bbN.    \lb{2.14}
\end{equation} 
Next, for each $(m,n) \in \bbN^2$, we introduce the set 
\begin{equation}
S_{m,n} = \{(x,y) \in M_m \times M_n \, | \, A(x,y) < 0 \; \mu_m \otimes \mu_n \text{-a.e.}\}.
\end{equation}
Then $S_{m,n} \in \cM_m \otimes \cM_n$, $(m,n) \in \bbN^2$. 

By \cite[Lemma\ 11.1.2]{Jo82}, for each $\varepsilon > 0$, there exist disjoint measurable 
rectangles $S_j(\varepsilon) \times T_j(\varepsilon) \in \cM_m \otimes \cM_n$, 
with $\mu_m(S_j(\varepsilon))< \infty$, $\mu_n(T_j(\varepsilon)) < \infty$, 
$j = 1,\cdots,N(\varepsilon)$, $N(\varepsilon) \in \bbN$, such that with
\begin{equation}
R_{N(\varepsilon)} = \bigcup_{j=1}^{N(\varepsilon)} [S_j(\varepsilon) \times T_j(\varepsilon)], 
\end{equation}   
one infers that 
\begin{equation}
\bigg| \int_{S_{m,n}} A(x,y) \, d (\mu_m \otimes \mu_n)(x,y) 
- \int_{R_{N(\varepsilon)}} A(x,y) \, d (\mu_m \otimes \mu_n)(x,y) \bigg| \leq \varepsilon. 
\lb{2.17} 
\end{equation}
Next we claim that 
\begin{equation}
\int_{R_{N(\varepsilon)}} A(x,y) \, d (\mu_m \otimes \mu_n)(x,y) \geq 0.    \lb{2.18} 
\end{equation}
Indeed, since $A(\cdot,\cdot) |_{M_m \times M_n} \in 
L^1(M_m \times M_n, \cM_m \otimes \cM_n; d(\mu_m \otimes \mu_n))$, and 
$\chi_{{}_{S_j(\varepsilon)}} \in L^2(M_m; d \mu_m)$,  
$\chi_{{}_{T_j(\varepsilon)}} \in L^2(M_n; d \mu_n)$ as $\mu_m(S_j(\varepsilon))< \infty$, 
$\mu_n(T_j(\varepsilon)) < \infty$ by hypothesis, an application of Fubini's theorem yields 
\begin{align}
&\int_{S_j(\varepsilon) \times T_j(\varepsilon)} A(x,y) \, d (\mu_m \otimes \mu_n)(x,y)  \\
&\quad= \int_M \chi_{{}_{S_j(\varepsilon)}}(x) \bigg(\int_M A(x,y)\chi_{{}_{T_j(\varepsilon)}}(y) 
\, d\mu(y) \bigg) d\mu(x)\nonumber\\
&\quad =\big(\chi_{{}_{S_j(\varepsilon)}}, A\chi_{{}_{T_j(\varepsilon)}}\big)_{M} \geq 0, 
\quad j = 1,\cdots,N(\varepsilon),      \lb{2.20}
\end{align}
as $A$ is assumed to be positivity preserving. Thus, combining \eqref{2.17} and 
\eqref{2.20} one concludes that 
\begin{align}
& \int_{S_{m,n}} A(x,y) \, d (\mu_m \otimes \mu_n)(x,y)    \no \\
& \quad = \bigg(\int_{S_{m,n}} A(x,y) \, d (\mu_m \otimes \mu_n)(x,y) 
- \int_{R_{N(\varepsilon)}} A(x,y) \, d (\mu_m \otimes \mu_n)(x,y) \bigg)   \no \\ 
& \qquad + \int_{R_{N(\varepsilon)}} A(x,y) \, d (\mu_m \otimes \mu_n)(x,y) 
\in [- \varepsilon, \infty).    \lb{2.21} 
\end{align}
Letting $\varepsilon \downarrow 0$, \eqref{2.21} finally yields that 
\begin{equation}
\int_{S_{m,n}} A(x,y) \, d (\mu_m \otimes \mu_n)(x,y) \geq 0
\end{equation}
and hence 
\begin{equation}
(\mu_m \otimes \mu_n)(S_{m,n}) = 0.
\end{equation}
Consequently,
\begin{equation}
(\mu \otimes \mu) \bigg(\bigcup_{(m,n) \in \bbN^2}S_{m,n}\bigg) = 0, 
\end{equation}
implying \eqref{2.13}. 
\end{proof}
%%%%%%%%%%

%%%%%%%%%%
\begin{remark} \lb{r2.4}
In connection with the sets $M_n$, $n\in\bbN$, in Theorem \ref{t2.3}, which are used to formulate a substitute for the lack of local integrability of the integral kernel 
$A(\cdot,\cdot)$ (due to the absence of any topology imposed on $M$), we note that 
one can assume that the $M_n$ are mutually disjoint, $M_m \cap M_n = \emptyset$, 
$m \neq n$, or else, that they are nesting, $M_n \subseteq M_{n+1}$, $m, n \in \bbN$. 
In addition, we note that upon introducing 
$L_n = \bigcup_{j=1}^n M_j$, $n \in \bbN$, the Cartesian product of $M$ with itself 
takes on the simple form $M \times M = \bigcup_{n\in\bbN} \big[L_n \times L_n\big]$.    
\end{remark} 
%%%%%%%%%

%%%%%%%%%%
\begin{remark} \lb{r2.5}
It is clear that if the integral kernel for $A$ satisfies $A(\cdot,\cdot)>0$ $\prodm$-a.e., then 
$A$ is indecomposable (cf.\ \cite[Sect.\ 7.1]{Da80}), in fact, positivity improving, that is, $A \succ 0$. 
The converse, however, is clearly false as the following elementary example shows.
\end{remark} 
%%%%%%%%%

%%%%%%%%%
\begin{example} \lb{e2.6} 
In the Hilbert space $L^2([0,1]; d\mu_0)$ with 
\begin{equation}
\mu_0 (B) = \int_B dx + \chi_{{}_{B}} (1/2), \quad B \in \mathfrak{B}(\bbR),
\end{equation}
where $\mathfrak{B}(\bbR)$ denotes the Borel $\sigma$-algebra on $\bbR$, and 
$\chi_{{}_S} (\cdot)$ the characteristic function of the set $S$, consider the rank-one 
operator $A_0$ 
with associated integral kernel $A_0(\cdot,\cdot)$ given by  
\begin{equation}
A_0(x,y) = a_0(x) a_0(y), \quad a_0(t) = |t-(1/2)|, \quad x, y, t \in [0,1].  
\end{equation}
Then $A_0(\cdot,\cdot) \geq 0$ and $A_0  \succ 0$, but clearly $A_0(\cdot,\cdot)$ is 
{\textit{\textbf{not} strictly positive}} $\mu_0 \otimes \mu_0$-a.e.
\end{example}
%%%%%%%%%

Indeed, using the fact that $h \in L^2(M; d\mu)$ is strictly positive (i.e., $h > 0$ $\mu$-a.e.) if 
and only if $(f,h)_M > 0$ for all $0 \leq f \in L^2(M; d\mu) \backslash \{0\}$, one concludes, 
upon identifying $M=[0,1]$ and $d\mu$ with Lebesgue measure $dt$, that 
\begin{equation}
\int_{[0,1]} a_0(t) f(t) \, d\mu_0 (t) = \int_{[0,1]} a_0(t) f(t) \, dt > 0, \quad 0 \leq f \in L^2([0,1]; d\mu_0) \backslash \{0\},
\end{equation} 
implying
\begin{align}
(f, A_0 g)_{[0,1]} & = \int_{[0,1]} a_0(x) f(x) \, d\mu_0(x) \int_{[0,1]} a_0 (y) g(y) \, d\mu_0(y)    \no \\
& = \int_{[0,1]} a_0(x) f(x) \, dx \int_{[0,1]} a_0 (y) g(y) \, dy > 0,       \\
& \hspace*{2.23cm} 0 \leq f, g \in L^2([0,1]; d\mu_0) \backslash \{0\}.    \no 
\end{align} 

One observes that the subset of positivity preserving operators of $\cB(\cH)$,
\begin{equation}
\{A \in \cB(\cH) \,|\, A \succcurlyeq 0\}
\end{equation}
is a cone, closed under multiplication, taking adjoints, and under taking weak operator limits. 

As an immediate consequence of \eqref{2.5}, we note that if $A, B \in\cB(\cH)$ then
\begin{align}
& A \succcurlyeq 0 \, \text{ if and only if } \, A^* \succcurlyeq 0,    \no \\ 
& A \succcurlyeq B \succcurlyeq 0 \, \text{ if and only if } \, A^* \succcurlyeq B^* \succcurlyeq 0,      
\lb{2.21A} \\ 
& \, \text{if } \, A \succcurlyeq 0, \; B \succcurlyeq 0 \, \text{ then } \, A B \succcurlyeq 0.    \no 
\end{align}

We also recall the following (slight refinement of a) useful result discussed in \cite{Da73} on 
domination in connection with trace class and Hilbert-Schmidt operators and we add one more 
item with respect to compactness :

%%%%%%%%%
\begin{lemma}  \lb{l2.7}
Assume Hypothesis \ref{h2.1} and suppose that 
\begin{equation}
A, B \in \cB\big(L^2(M; d\mu)\big) \, \text{ and } \, A \succcurlyeq B \succcurlyeq 0.   \lb{2.21a}
\end{equation} 
$(i)$ The following norm bound holds, 
\begin{equation}
\|A\|_{\cB(L^2(M; d\mu))} \geq \|B\|_{\cB(L^2(M; d\mu))}.    \lb{2.22}
\end{equation} 
$(ii)$ Assume in addition to \eqref{2.21a} that $A \in \cB_{2n}\big(L^2(M; d\mu)\big)$ for some 
$n \in \bbN$. Then 
also $B \in \cB_{2n}\big(L^2(M; d\mu)\big)$ and 
\begin{equation}
\|A\|_{\cB_{2n}(L^2(M; d\mu))} \geq \|B\|_{\cB_{2n}(L^2(M; d\mu))}.    \lb{2.23} 
\end{equation} 
$(iii)$ Assume in addition to \eqref{2.21a} that $A \in \cB_{\infty}\big(L^2(M; d\mu)\big)$.\ Then 
also $B \in \cB_{\infty}\big(L^2(M; d\mu)\big)$. \\ 
$(iv)$ Assume in addition to \eqref{2.21a} that  $A \geq 0$, $B \geq 0$, and 
$A \in \cB_1\big(L^2(M; d\mu)\big)$. Then also $B \in \cB_1\big(L^2(M; d\mu)\big)$ and 
\begin{equation}
{\tr}_{L^2(M; d\mu)} (A) = \|A\|_{\cB_1(L^2(M; d\mu))} \geq \|B\|_{\cB_1(L^2(M; d\mu))} 
= {\tr}_{L^2(M; d\mu)} (B).    \lb{2.24}
\end{equation} 
\end{lemma}
%%%%%%%%%
\begin{proof}
$(i)$ We refer to \cite[Lemma\ 1.2]{Da73} for the proof of \eqref{2.22} (see also the 
paragraph following Theorem \ref{t2.6} for a simple argument). \\
$(ii)$ Repeatedly employing \eqref{2.21A} one infers that 
$A \succcurlyeq B \succcurlyeq 0$ implies $A^* A \succcurlyeq B^* B \succcurlyeq 0$ and hence 
\begin{equation}
(A^* A)^n \succcurlyeq (B^* B)^n \succcurlyeq 0. 
\end{equation}
Recalling that $T \in \cB_{2n}(\cH)$ if and only if $T^* T \in \cB_{n}(\cH)$, which in turn is equivalent 
to $(T^* T)^n \in \cB_1(\cH)$ and that
\begin{equation}
\|T\|_{\cB_{2n}(\cH)}^{2n} = \|T^* T\|_{\cB_{n}(\cH)}^{n} = \big\|(T^* T)^n\big\|_{\cB_{1}(\cH)} 
= {\tr}_{\cH} \big((T^* T)^n\big) , 
\end{equation}
one can apply item $(iv)$ (with the pair $A, B$ replaced by $(A^* A)^n, (B^* B)^n$) to conclude that 
$B \in \cB_{2n}\big(L^2(M; d\mu)\big)$ and 
\begin{align}
\begin{split} 
{\tr}_{L^2(M; d\mu)} \big((A^* A)^n\big) &= \big\|(A^* A)^n\big\|_{\cB_1(L^2(M; d\mu))}    \\
& \geq \big\|(B^* B)^n\big\|_{\cB_1(L^2(M; d\mu))}  = {\tr}_{L^2(M; d\mu)} \big((B^* B)^n\big).
\end{split}
\end{align} 
Item $(iii)$ is a special case of a result proven in \cite{DF79} and \cite{Pi79} (see also \cite{Le82}). \\
$(iv)$ Following \cite[Lemma\ 1.2]{Da73}, one constructs a filtering increasing set of orthogonal projections $P_L$ in $L^2(M; d\mu)$ 
associated with closed subspaces $L$ generated by finitely many characteristic functions such that $P_L$ strongly converges to $I_M$. Since the details of this construction are a bit involved, we now pause and present the precise argument: First, one recalls from Hypothesis \ref{h2.1} that $(M,\cM,\mu)$ is a 
$\sigma$-finite measure space. Consider $\cM_\ast = \{E\in\cM \,|\, 0<\mu(E)<\infty\}$ and,  
for every set $E\in\cM_\ast$, define $\widetilde{\chi}_E = \mu(E)^{-1/2}\chi_{{}_E}$, that is, 
the $L^2$-normalization of the characteristic function of $E$. Next, for every 
finite family $\{E_1,\dots,E_N\}$ of mutually disjoint sets in $\cM_\ast$ introduce 
\begin{equation}\label{MAR-A.0}
L(E_1,\dots,E_N) = \bigg\{\sum_{j=1}^N a_j\widetilde{\chi}_{E_j}\,\bigg|\,a_j\in{\mathbb{R}},
\,1\leq j\leq N\bigg\}, 
\end{equation} 
that is, the linear span of the $\widetilde{\chi}_{E_i}$'s. Going further, introduce
\begin{equation}\label{MAR-A.1}
{\mathfrak{L}}(M,\cM,\mu) = \{L(E_1,\dots,E_N) \,|\, N\in{\mathbb{N}}
\, \text{and} \, \{E_1,\dots,E_N\}\subseteq\cM_\ast \}, 
\end{equation} 
and equip this family of sets with the partial order induced by the inclusion. 
Finally, for each $L=L(E_1,\dots,E_N)\in{\mathfrak{L}}(M,\cM,\mu)$ denote by $P_L$
the operator of projection in $L^2(M;d\mu)$ onto the linear (closed) subspace 
$L$ of $L^2(M;d\mu)$. Hence, given that the collection 
$\{\widetilde{\chi}_{E_j}\}_{1\leq j\leq N}$ is an orthonormal 
basis for $L(E_1,\dots,E_N)$, for every function $f\in L^2(M;d\mu)$, one has   
\begin{equation}\label{MAR-A.2}
P_Lf=\sum_{j=1}^N\bigg(\int_{M}f\widetilde{\chi}_{E_j}\,d\mu\bigg)
\widetilde{\chi}_{E_j}
=\sum_{j=1}^N\bigg(\mu(E_j)^{-1}\int_{E_j}f\,d\mu\bigg)\chi_{{}_{E_j}}.
\end{equation} 
Moreover, for every $L_1,L_2\in{\mathfrak{L}}(M,\cM,\mu)$ such that 
$L_2\subseteq L_1$, the Pythagorean theorem implies
\begin{equation}\label{MAR-A.3}
\|f-P_{L_2}f\|^2_M = \|f-P_{L_1}f\|^2_M + \|P_{L_1}f-P_{L_2}f\|^2_M
\end{equation} 
for every $f\in L^2(M;d\mu)$. This implies that for every $f\in L^2(M;d\mu)$
the bound
\begin{equation}\label{MAR-A.4}
\|f-P_{L_1}f\|_M \leq \|f-P_{L_2}f\|_M,
\end{equation} 
holds whenever $L_1,L_2\in{\mathfrak{L}}(M,\cM,\mu)$ are such that $L_2\subseteq L_1$. 
At this stage, we make the claim that the filtering increasing family of projections
$\{P_L\}_{L\in{\mathfrak{L}}(M,\cM,\mu)}$ converges strongly in 
$L^2(M;d\mu)$ to the identity on $M$, that is, 
\begin{equation}\label{MAR-A.5}
\mbox{s-}\lim_{L}P_L=I_M \, \text{ in }\, L^2(M;d\mu).
\end{equation} 
To justify the above claim, fix an arbitrary function $f\in L^2(M;d\mu)$, 
along with an arbitrary number $\varepsilon>0$. Also, for each 
$\lambda\in(0,\infty)$ set $f_\lambda:=f\chi_{{}_{E_\lambda}}$ where 
$E_\lambda:= \{x\in M \,|\, |f(x)|\leq\lambda\}$. Hence, $E_\lambda\in\cM$ 
for every $\lambda\in(0,\infty)$. Since the sequence 
$\bigl\{f_\lambda\bigr\}_{\lambda>0}$ converges
to $f$ in $L^2(M;d\mu)$ as $\lambda\to\infty$ (by Lebesgue's Dominated Convergence 
Theorem), one can find $\lambda_0\in(0,\infty)$ with the property that  
\begin{equation}\label{MAR-A.6}
\|f-f_{\lambda_0}\|_M \leq \varepsilon/8.
\end{equation} 
Next, given that the measure space
$(M,\cM,\mu)$ is $\sigma$-finite, there exists a family $\{M_k\}_{k\in{\mathbb{N}}}$ 
of mutually disjoint sets in $\cM_\ast$ with the property that 
$M=\bigcup\limits_{k\in{\mathbb{N}}}M_k$. Since the sequence 
$\{f_{\lambda_0}\chi_{{}_{\cup_{k=1}^n M_k}}\}_{n\in{\mathbb{N}}}$ converges
to $f_{\lambda_0}$ in $L^2(M;d\mu)$ as $n\to\infty$ (again, by Lebesgue's 
Dominated Convergence Theorem), one can find $n_0\in{\mathbb{N}}$ such that 
\begin{equation}\label{MAR-A.7}
\|f_{\lambda_0}-f_{\lambda_0}\chi_{{}_{\cup_{k=1}^{n_0} M_k}}\|_M \leq \varepsilon/8.
\end{equation} 
To proceed, abbreviate $M_0 = \chi_{{}_{\cup_{k=1}^{n_0} M_k}}$ and 
$f_0 = f_{\lambda_0}\chi_{{}_{M_0}}$. Also, pick an arbitrary integer $n$ such 
that $n>\lambda_0$ (whose actual value is to be specified later). 
In particular, $|f_0(x)|<n$ for every $x\in M$, and $f_0(x)=0$ for every 
$x\in M \backslash M_0$. Also, from \eqref{MAR-A.6} and \eqref{MAR-A.7} one infers that 
\begin{equation}\label{MAR-A.7BIS}
\|f-f_0\|_M \leq \varepsilon/4.
\end{equation} 
For each integer $k\in [1-n2^n,n2^n]$, we now define 
\begin{equation}\label{MAR-A.8}
E_{n,k} = f_0^{-1}\bigl(\bigl[\tfrac{k-1}{2^n},\tfrac{k}{2^n}\bigr)\bigr)\cap M_0.
\end{equation} 
By design, for every $k\in [1-n2^n,n2^n]\cap{\mathbb{Z}}$ we then obtain
\begin{align}\label{MAR-A.9}
& E_{n,k}\in\cM,\quad E_{n,k}\subseteq M_0,      \\
&\text{and } \, \tfrac{k-1}{2^n}\leq f_0(x)<\tfrac{k}{2^n}\,\,\mbox{ for each }\,\,x\in E_{n,k}.
\label{MAR-A.10}
\end{align} 
Furthermore, 
\begin{equation}\label{MAR-A.11}
\text{the family $\{E_{n,k}\}_{1-n2^n\leq k \leq n2^n}$ is a 
disjoint partition of $M_0$}.
\end{equation} 
Hence, whenever $k\in [1-n2^n,n2^n]\cap{\mathbb{Z}}$ is such that
$\mu(E_{n,k})>0$, we have $E_{n,k}\in\cM_\ast$ and  
\begin{equation}\label{MAR-A.12}
\tfrac{k-1}{2^n}\leq\mu(E_{n,k})^{-1}\int_{E_{n,k}}f_0\,d\mu\leq\tfrac{k}{2^n}.
\end{equation} 
In light of \eqref{MAR-A.10}, this shows that 
whenever $k\in [1-n2^n,n2^n]\cap{\mathbb{Z}}$ is such that
$\mu(E_{n,k})>0$ we have 
\begin{equation}\label{MAR-A.13}
\bigg|f_0(x)-\mu(E_{n,k})^{-1}\int_{E_{n,k}}f_0\,d\mu\bigg|\leq 2^{-n}
\, \text{ for each }\, x\in E_{n,k}.
\end{equation} 
Next, thin out the family $\{E_{n,k}\}_{1-n2^n\leq k \leq n2^n}$ by 
throwing away the measure zero sets, and relabel the remaining ones as 
$\{E_1,\dots,E_N\}$. In addition, consider $L = L(E_1,\dots,E_N)\in{\mathfrak{L}}(M,\cM,\mu)$. 
Because of \eqref{MAR-A.11} it follows that $\{E_1,\dots,E_N\}$ is, up to a set of 
measure zero, a disjoint partition of $M_0$. Based on this, the fact that 
$f_0$ vanishes outside $M_0$ (as pointed out earlier), as well as  
\eqref{MAR-A.13} and \eqref{MAR-A.2}, we may then estimate  
\begin{align}\label{MAR-A.14}
& |f_0(x)-(P_L f_0)(x)|
=\bigg|\sum_{j=1}^Nf_0(x)\chi_{{}_{E_j}}(x)-\bigg(\mu(E_j)^{-1}\int_{E_j}f_0\,d\mu\bigg)
\chi_{{}_{E_j}}(x)\bigg|       \no \\
& \quad 
\leq\sum_{j=1}^N\bigg|\bigg(f_0(x)-\mu(E_j)^{-1}\int_{E_j}f_0\,d\mu\bigg)
\chi_{{}_{E_j}}(x)\bigg|\leq 2^{-n},
\end{align} 
for $\mu$-a.e. $x\in M$. Since both $f_0$ and $P_L f_0$ vanish identically 
outside $M_0$, we deduce from \eqref{MAR-A.14} that  
\begin{equation}\label{MAR-A.14a}
\|f_0-P_L f_0\|_M \leq 2^{-n}\mu(M_0)^{1/2}.
\end{equation} 
Consequently,  
\begin{align}\label{MAR-A.15}
\|f-P_L f\|_M & \leq  
\|f-f_0\|_M + \|f_0-P_L f_0\|_M    \no \\
& \quad + \|P_L(f_0-f)\|_M \no \\
&\leq  (\varepsilon/2) + 2^{-n}\mu(M_0)^{1/2},
\end{align} 
by \eqref{MAR-A.7BIS} (used twice) and \eqref{MAR-A.14a}. 
Hence, if the integer $n\in(\lambda_0,\infty)$ is chosen large enough so that
$2^{-n}\mu(M_0)^{1/2}\leq\varepsilon/2$ to begin with (recall that 
$\mu(M_0)<\infty$), it follows from \eqref{MAR-A.15} that 
$\|f-P_L f\|_M \leq \varepsilon$. In turn, 
from this and \eqref{MAR-A.4} we conclude that 
\begin{equation}\label{MAR-A.4a} 
\|f-P_{L'}f\|_M \leq\varepsilon
\end{equation} 
for every $L'\in{\mathfrak{L}}(M,\cM,\mu)$ such that $L\subseteq L'$. 
This finishes the proof of \eqref{MAR-A.5}.

Next, one notes that the strong convergence in \eqref{MAR-A.5} implies  
\begin{align}
{\tr}_{L^2(M; d\mu)} (A) = \sup_{L} {\tr}_{L^2(M; d\mu)} (P_L A P_L) \geq 
\sup_{L} {\tr}_{L^2(M; d\mu)} (P_L B P_L).     \lb{2.24a} 
\end{align}
Here we used $A \geq 0$, the monotonicity of ${\tr}_{L^2(M; d\mu)} (P_L A P_L)$ with respect to $L$, 
and (a special case of) the following Lemma \ref{l2.8} (with $S_n = S = A$, $R_n = T_n^* = P_L$) in 
the first equality in \eqref{2.24a} and 
\begin{equation}
\sum_{j=1}^J (\chi_{{}_{M_j}}, A \chi_{{}_{M_j}})_M \geq 
\sum_{j=1}^J (\chi_{{}_{M_j}}, B \chi_{{}_{M_j}})_M
\end{equation}
as a consequence of $A \succcurlyeq B \succcurlyeq 0$ in the last inequality in \eqref{2.24a}. 
The non-commutative Fatou lemma, \cite[Theorem\ 2.7\,(d)]{Si05}, and \eqref{2.24a} then imply  
$B \in \cB_1(L^2(M; d\mu))$, and 
\begin{equation} 
\sup_{J\in\bbN} {\tr}_{L^2(M; d\mu)} (P_J B P_J) 
\geq \|B\|_{\cB_1(L^2(M; d\mu))} = {\tr}_{L^2(M; d\mu)} (B),    \lb{2.24b}
\end{equation} 
 where we used $B \geq 0$ in the last equality in \eqref{2.24b}, proving \eqref{2.24}.    
 \end{proof}
%%%%%%%%%

 %%%%%%%%%
\begin{lemma}\lb{l2.8}
Let $p\in[1,\infty)$ and assume that $R,R_n,T,T_n\in\cB(\cH)$, 
$n\in\bbN$, satisfy
$\slim_{n\to\infty}R_n = R$  and $\slim_{n\to \infty}T_n = T$ and that
$S,S_n\in\cB_p(\cH)$, $n\in\bbN$, satisfy 
$\lim_{n\to\infty}\|S_n-S\|_{\cB_p(\cH)}=0$.
Then $\lim_{n\to\infty}\|R_n S_n T_n^\ast - R S T^\ast\|_{\cB_p(\cH)}=0$.
\end{lemma}
%%%%%%%%%
This follows, for instance, from \cite[Theorem 1]{Gr73}, \cite[p.\ 28--29]{Si05}, or
\cite[Lemma 6.1.3]{Ya92} with a minor additional effort (taking adjoints, etc.). 

We emphasize that items $(i)$--$(iii)$ in Lemma \ref{l2.7} are not optimal. Indeed,
one easily verifies that in analogy to \eqref{2.4}, under the assumptions that 
$A, B \in \cB\big(L^2(M; d\mu)\big)$ and $A \succcurlyeq 0$, $B \succcurlyeq 0$, one concludes that 
\begin{equation}
\text{$0 \preccurlyeq B \preccurlyeq A$ if and only if $|B f| \leq A |f|$, $f \in L^2(M; d\mu)$.} 
\lb{2.24c}
\end{equation}
On the other hand, the condition $|B f| \leq A |f|$, $f \in L^2(M; d\mu)$, in \eqref{2.24c} remains 
meaningful in the more general situation where $B$ is no longer assumed to satisfy $B \succcurlyeq 0$. 
In fact, this leads to the following notion of pointwise domination:

%%%%%%%%%%
\begin{definition} \lb{d2.7}
Assume $A, B \in \cB\big(L^2(M; d\mu)\big)$ and $A \succcurlyeq 0$. Then $A$ is said to 
{\it pointwise dominate} $B$ if 
\begin{equation}
|B f| \leq A |f|, \quad f \in L^2(M; d\mu).    \lb{2.24d}
\end{equation}
\end{definition}
%%%%%%%%%%

One then has the following improvement of Lemma \ref{l2.7}\,$(ii)$,\,$(iii)$ due to 
Dodds and Fremlin \cite{DF79}, Pitt \cite{Pi79}, and Simon \cite[Theorem\ 2.13]{Si05}:

%%%%%%%%%%
\begin{theorem} \lb{t2.6}
Assume $A, B \in \cB\big(L^2(M; d\mu)\big)$, that $A \succcurlyeq 0$, and that $A$ pointwise 
dominates $B$. Then the subsequent assertions hold: \\
$(i)$ The following norm bound is valid, 
\begin{equation}
\|A\|_{\cB(L^2(M; d\mu))} \geq \|B\|_{\cB(L^2(M; d\mu))}.    \lb{2.24da}
\end{equation} 
$(ii)$ Suppose in addition that $A \in \cB_{2n}\big(L^2(M; d\mu)\big)$ for some $n \in \bbN$. Then 
also $B \in \cB_{2n}\big(L^2(M; d\mu)\big)$ and 
\begin{equation}
\|A\|_{\cB_{2n}(L^2(M; d\mu))} \geq \|B\|_{\cB_{2n}(L^2(M; d\mu))}.    \lb{2.24e} 
\end{equation} 
$(iii)$ Suppose in addition that $A \in \cB_{\infty}\big(L^2(M; d\mu)\big)$. Then 
also $B \in \cB_{\infty}\big(L^2(M; d\mu)\big)$. 
\end{theorem} 
%%%%%%%%%%

\noindent 
Since obviously, 
$\|Bf\|_M \leq \big\|A|f|\big\|_M \leq \|A\|_{\cB(L^2(M;d\mu))}\|f\|_M$, 
$f\in L^2(M;d\mu)$, this settles item $(i)$ in Theorem \ref{t2.6} (and hence also item $(i)$ in 
Lemma \ref{l2.7}). Item $(ii)$ is proved in \cite[Theorem\ 2.13]{Si05}, and $(iii)$ is due to \cite{DF79} and \cite{Pi79} (see also 
\cite{Le82}). The case of integral operators $A$ 
and $B$ was first treated in \cite[p.~94]{ZKKMRS75}. We emphasize that 
Theorem \ref{t2.6}\,$(ii)$ is wrong with $2n$, $n\in\bbN$, replaced by any $p < 2$ (in particular, it fails 
in the trace class case $p=1$), it is also not true for $p \in (2,\infty) \backslash \{2(n+1)\}_{n \in \bbN}$, 
due to counterexamples by Peller \cite{Pe82}, employing Hankel operators (see also the discussion in 
\cite[p.\  24, 128--129]{Si05}).  

\medskip

Next, we turn to positivity preserving contraction semigroups. First, we recall the following well-known result:

%%%%%%%%%
\begin{theorem} [\cite{RS78}, p.\ 204, 209] \lb{t2.7}
Suppose that $H$ is a semibounded, self-adjoint operator in $L^2(M; d\mu)$ with 
$\lambda_0 = \inf(\sigma(H))$. Then the following assertions $(i)$--$(iii)$ are 
equivalent: \\[1mm]
$(i)$ \, $e^{-t H}\succcurlyeq 0$ for all $t \geq 0$. \\[1mm] 
$(ii)$ $(H - \lambda I_M)^{-1} \succcurlyeq 0$ for all $\lambda < \lambda_0$. \\[1mm]
$(iii)$ $f\in \dom\big(|H|^{1/2}\big)$ implies $|f| \in \dom\big(|H|^{1/2}\big)$ and \\[1mm]
\hspace*{7mm} $\big\| (H - \lambda_0 I_M)^{1/2} |f|\big\|_M \leq \big\| (H - \lambda_0 I_M)^{1/2} f \big\|_M$.
\end{theorem}
%%%%%%%%%

The equivalence of items $(i)$ and $(ii)$ follows from the relations
\begin{align}
& (H - \lambda I_{M})^{-1} = \int_0^{\infty} dt \, e^{t \lambda} e ^{- t H}, 
\quad \lambda < \lambda_0,    \lb{2.25} \\
& \, e^{-t H} = e^{- t \lambda_0}
\slim_{n \to \infty} \big[I_M + (t/n) (H - \lambda_0 I_M)\big]^{-n}, \quad t \geq 0.  \lb{2.26}
\end{align} 

The next result recalls basic semigroup domination facts:
 
%%%%%%%%%
\begin{theorem} \lb{t2.8}
Suppose that $H_j$, $j=1,2$, are semibounded, self-adjoint operators in 
$L^2(M; d\mu)$ with $\lambda_0 = \inf(\sigma(H_1), \sigma(H_2))$. Moreover, assume 
that $H_j$ generate positivity preserving semigroups, $\exp(-t H_j) \succcurlyeq 0$, 
$j=1,2$. Then the following assertions $(i)$--$(iii)$ are equivalent: \\[1mm]
$(i)$ \,\, $\exp(-t H_1) \succcurlyeq \exp(-t H_2) \succcurlyeq 0$ for all $t \geq 0$. \\[1mm]
$(ii)$ \, $(H_1 - \lambda I_M)^{-1} \succcurlyeq 
(H_2 - \lambda I_M)^{-1} \succcurlyeq 0$ for 
all $\lambda < \lambda_0$. \\[1mm]
$(iii)$ For all $f_j \in \dom(H_j)_+$, $j=1,2$, $(f_1,H_2 f_2)_M \geq (H_1 f_1, f_2)_M$. 
\\[1mm]
Suppose in addition that the form domains of $H_1$ and $H_2$ coincide, 
$\dom\big(|H_1|^{1/2}\big) = \dom\big(|H_2|^{1/2}\big)$. Then items 
$(i)$--$(iii)$ are further equivalent to \\[1mm]
$(iv)$ For all $f_j \in \dom\big(|H_1|^{1/2}\big)_+ = \dom\big(|H_2|^{1/2}\big)_+$, 
$j=1,2$, \\[1mm] 
$\big((H_2-\lambda_0 I_M)^{1/2} f_1,(H_2-\lambda_0 I_M)^{1/2} f_2\big)_M \geq 
\big((H_1-\lambda_0 I_M)^{1/2} f_1, (H_1-\lambda_0 I_M)^{1/2} f_2\big)_M$. 
\end{theorem} 
%%%%%%%%%
\begin{proof}
The equivalence of items $(i)$ and $(ii)$ again follows from the relations
\eqref{2.25}, \eqref{2.26}. 

In order to prove that item $(i)$ implies item $(iii)$ suppose that $f_j \in \dom(H_j)_+$, $j=1,2$, then item $(i)$ implies 
$\big(f_1, e^{- t H_1} f_2\big)_M \geq \big(f_1, e^{- t H_2} f_2\big)_M$ and hence 
\begin{equation}
\big(f_1, t^{-1}\big[I_M - e^{- t H_2}\big] f_2\big)_M \geq 
\big(f_1, t^{-1}\big[I_M - e^{- t H_1}\big] f_2\big)_M.     \lb{2.27}
\end{equation}
Letting $t \downarrow 0$ in \eqref{2.27} yields item $(iii)$ as $e^{-t H_j}$, $j=1,2$, are strongly differentiable with respect to $t \geq 0$. \\
To show that item $(iii)$ implies item $(ii)$ one can argue as follows: Let 
$f_j \in L^2(M; d\mu)_+$, $j=1,2$, then by Theorem \ref{t2.7}, 
$(H_j - \lambda I_M)^{-1} f_j \in \dom(H_j)_+$, $\lambda < \lambda_0$, $j=1,2$, and 
hence by item $(iii)$,
\begin{equation}
\big((H_1 - \lambda I_M)^{-1} f_1, H_2 (H_2 - \lambda I_M)^{-1} f_2\big)_M \geq 
\big(H_1 (H_1 - \lambda I_M)^{-1} f_1, (H_2 - \lambda I_M)^{-1} f_2\big)_M.   \lb{2.28}
\end{equation}
Using $H(H - \lambda I_M)^{-1} = I_M + \lambda (H - \lambda I_M)^{-1}$ on either 
side of \eqref{2.28} yields  
\begin{equation}
\big(f_1, (H_1 - \lambda I_M)^{-1} f_2\big)_M \geq 
\big(f_1, (H_2 - \lambda I_M)^{-1} f_2\big)_M.  
\end{equation}
and hence yields item $(ii)$. \\
Next suppose that $0 \leq f_j \in \dom\big(|H_2|^{1/2}\big)$, $j=1,2$, then item $(i)$ once again implies 
\begin{equation}
\big(f_1, t^{-1}\big[I_M - e^{- t (H_2-\lambda_0 I_M)}\big] f_2\big)_M \geq 
\big(f_1, t^{-1}\big[I_M - e^{- t (H_1-\lambda_0 I_M)}\big] f_2\big)_M.      \lb{2.30}
\end{equation} 
Letting $t\downarrow 0$ in \eqref{2.30} then implies the sesquilinear form 
version of $(iii)$, that is, 
\begin{equation} 
\big((H_2-\lambda_0 I_M)^{1/2} f_1, (H_2-\lambda_0 I_M)^{1/2} f_2\big)_M \geq 
\big((H_1-\lambda_0 I_M)^{1/2} f_1, (H_1-\lambda_0 I_M)^{1/2} f_2\big)_M, 
\end{equation} 
and hence item $(iv)$. \\
Finally, let $f_j \in \dom(H_j)_+$, then $f_j \in \dom\big(|H_2|^{1/2}\big)_+$, $j=1,2$, and 
item $(iv)$ implies 
\begin{align}
\begin{split} 
& (f_1, (H_2-\lambda_0 I_M) f_2)_M = \big((H_2-\lambda_0 I_M)^{1/2} f_1, (H_2-\lambda_0 I_M)^{1/2} f_2\big)_M   \\
& \quad \geq 
\big((H_1-\lambda_0 I_M)^{1/2} f_1, (H_1-\lambda_0 I_M)^{1/2} f_2\big)_M 
= \big((H_1-\lambda_0 I_M) f_1, f_2\big)_M,
\end{split} 
\end{align}
implying item $(iii)$.
\end{proof}
%%%%%%%%%

In the proof of Theorem \ref{t2.8} we used the known fact that for any semibounded, 
self-adjoint operator $H$ in a complex separable Hilbert space $\cH$ the domain and 
form domain of $H$ can be characterized in terms of its semigroup $e^{-t H}$, 
$t\geq 0$, by
\begin{equation}
\dom(H) = \Big\{g \in \cH \, \Big| \, \lim_{t \downarrow 0} 
t^{-1} \big\| \big[I_{\cH} - e^{-t H}\big] g \big\|_{\cH} \text{ exists finitely}\Big\},   \lb{2.33}
\end{equation}
in particular,
\begin{equation}
 \lim_{t \downarrow 0} t^{-1} \big[I_{\cH} - e^{-t H}\big] f = H f, \quad f \in \dom(H). 
 \lb{2.34}
\end{equation}
In addition, 
\begin{equation}
\dom\big(|H|^{1/2}\big) = \Big\{g \in \cH \, \Big| \, \lim_{t \downarrow 0} 
t^{-1} \big(g, \big[I_{\cH} - e^{-t H}\big] g \big)_{\cH} \text{ exists finitely}\Big\},   \lb{2.35}
\end{equation}  
and hence, 
\begin{equation}
\lim_{t \downarrow 0} 
t^{-1} \big(f, \big[I_{\cH} - e^{-t H}\big] f \big)_{\cH} 
= \big(|H|^{1/2}f, \sgn (H) |H|^{1/2} f\big)_{\cH}, \quad f \in \dom\big(|H|^{1/2}\big), 
\lb{2.36}
\end{equation}
and by polarization,
\begin{equation}
\lim_{t \downarrow 0} 
t^{-1} \big(f, \big[I_{\cH} - e^{-t H}\big] g \big)_{\cH} 
= \big(|H|^{1/2}f, \sgn (H) |H|^{1/2} g\big)_{\cH}, \quad f, g \in \dom\big(|H|^{1/2}\big).  
\lb{2.37} 
\end{equation}
The results \eqref{2.33}--\eqref{2.36} follow from a combination of the spectral theorem 
for self-adjoint operators and Lebesgue's Dominated Convergence Theorem. (In fact, 
\eqref{2.33}--\eqref{2.37} also extend to unitary groups $e^{- i t H}$, but in this case the proof of the analogs of \eqref{2.35}, \eqref{2.36} is more subtle and involves certain 
results on characteristic functions and moments of probability measures, see, e.g., 
\cite[Satz\ 2.1]{GR71}.) 

Theorem \ref{t2.8}\,$(i)$--$(iii)$ is due to \cite{BKR80} and for convenience of the reader 
we quickly reproduced the arguments in \cite{BKR80}. The elementary quadratic form 
part $(iv)$ in Theorem \ref{t2.8} is a special case of a general result in connection with m-accretive 
generators of semigroups and their associated forms due to Ouhabaz \cite{Ou96}, 
\cite[Theorem\ 2.24]{Ou05}, based on the concept of ideals of subspaces of $L^2(M; d\mu)$. 
In this context we also refer to \cite[Sect.\ 4]{MVV05} and \cite{SV96}. For 
completeness, we offered the elementary proof of part $(iv)$ based on the observation \eqref{2.37}, 
which also differs from the arguments used in \cite{SV96} in the special case of symmetric forms. 
Part $(iv)$ of Theorem \ref{t2.8} will subsequently be applied in the context of elliptic partial 
differential operators in divergence form in Section \ref{s4}. 

Next we turn to additional results on positivity improving semigroups. We recall that if no nontrivial, 
closed subspace of $L^2(M; d\mu)$ is left invariant by $T$ and every bounded operator of 
multiplication on $L^2(M; d\mu)$, then this is indicated by the statement that 
$\{T\} \cup L^{\infty}(M; d\mu)$ {\it acts irreducibly on} $L^2(M; d\mu)$. 

\medskip 

%%%%%%%%%
\begin{theorem}  [\cite{RS78}, Thm.\ XIII.44; \cite{KR80}]~Assume \lb{t2.9} that $H$ is a 
semibounded, self-adjoint operator in $L^2(M; d\mu)$ with 
$\lambda_0 = \inf(\sigma(H))$. Then the following assertions $(i)$--$(viii)$ are 
equivalent: \\[1mm]
$(i)$ \;\, $e^{-t H}\succ 0$ for all $t > 0$. \\[1mm] 
$(ii)$ \, $e^{-t H}\succ 0$ for some $t > 0$. \\[1mm] 
$(iii)$ $\big\{e^{-tH}\big\}_{t>0} \cup L^{\infty}(M; d\mu)$ acts irreducibly on 
$L^2(M; d\mu)$. \\[1mm]
$(iv)$ \hspace*{.3mm} $e^{-tH} \cup L^{\infty}(M; d\mu)$ acts irreducibly on 
$L^2(M; d\mu)$ for some $t>0$. \\[1mm]
$(v)$ \; $(H - \lambda I_M)^{-1} \succ 0$ for all $\lambda < \lambda_0$. \\[1mm]
$(vi)$ \hspace*{.2mm} $(H - \lambda I_M)^{-1} \succ 0$ for some 
$\lambda < \lambda_0$. \\[1mm]
$(vii)$ $\big\{(H - \lambda I_M)^{-1}\big\}_{\lambda < \lambda_0} \cup 
L^{\infty}(M; d\mu)$ acts irreducibly on $L^2(M; d\mu)$. \\[1mm]
$(viii)$ $(H - \lambda I_M)^{-1} \cup 
L^{\infty}(M; d\mu)$ acts irreducibly on $L^2(M; d\mu)$ for some $\lambda < \lambda_0$. 
\end{theorem}
%%%%%%%%%

The following remarkable consequence of Theorem \ref{t2.9} will be useful later on.

%%%%%%%%%
\begin{corollary} [\cite{KR80}] \lb{c2.10}
Suppose that $H_j \geq 0$, $j=1,2$, are semibounded, self-adjoint operators in $L^2(M; d\mu)$,  
and assume that $H_j$ generate positivity preserving semigroups, $\exp(-t H_j) \succcurlyeq 0$, 
$j=1,2$, satisfying 
\begin{equation}
e^{-t H_1} \succcurlyeq e^{-t H_2} \succcurlyeq 0 \, \text{ for all } \, t > 0. 
\end{equation}
If either $e^{-t H_1} \succ 0$, or $e^{-t H_2} \succ 0$ for some $t>0$, then  
\begin{equation}
\text{either } \, e^{-t H_1} \succ e^{-t H_2}, \, \text{ or else, } \,  
e^{-t H_1} = e^{-t H_2} \, \text{ for all } \, t > 0.
\end{equation}
\end{corollary}
%%%%%%%%%

We also remark that these notions of positivity preserving (resp., improving) naturally extend to a two-Hilbert space setting in which one deals with a second Hilbert space $L^2(Y; d\nu)$ with $Y \subset X$ and 
$\wti \mu = \mu|_{Y}$, see, for instance, 
\cite{BKR80}, \cite{KR80}. This is also frequently done in connection with (nondensely defined) quadratic forms (cf., e.g., \cite[p.\ 61--62]{Da89}).

For subsequent purpose in Section \ref{s4} we also recall the following facts on sub-Markovian operators:

%%%%%%%%%%
\begin{definition} \lb{d2.15}
An operator $A \in \cB\big(L^2(M; d\mu)\big)$ is called 
{\it $L^\infty$-contractive} if
\begin{equation}
\|A f\|_{L^\infty(M; d\mu)} \leq \|f\|_{L^\infty(M; d\mu)}, \quad 
f \in L^2(M; d\mu)\big) \cap L^\infty(M; d\mu)\big).
\end{equation}
Moreover, $A$ is called {\it sub-Markovian} if it is positivity preserving, 
$A \succcurlyeq 0$, and $L^\infty$-contractive. \\ 
Finally, a semigroup $\{T(t)\}_{t \geq 0} \subset \cB\big(L^2(M;d\mu)\big)$ is called a contraction 
semigroup on $L^p(M;d\mu)$ for some $p \geq 1$, if for all $t \geq 0$, $T(t)$ extends to a contraction on $L^p(M;d\mu)$. 
\end{definition}
%%%%%%%%%%

One verifies that $A \in \cB\big(L^2(M; d\mu)\big)$ is  
$L^\infty$-contractive if and only if it leaves the closed, convex set 
$\cC = \big\{f \in L^2(M; d\mu) \, \big| \, |f| \leq 1 \, \text{$\mu$-a.e.} \big\}$ invariant, that is, 
$A \cC \subseteq \cC$ (cf.\ also \cite{Ou92}). Moreover, one readily verifies the following equivalent 
definition of $A \in \cB\big(L^2(M; d\mu)\big)$ being sub-Markovian, namely, 
\begin{align}
\begin{split} 
& A \, \text{ is sub-Markovian if and only if } \, f \in L^2(M; d\mu)\big), \; 0 \leq f \leq 1 \\ 
& \quad \text{implies } \, 0 \leq A f \leq 1.    \lb{2.79} 
\end{split}
\end{align}
In fact, \eqref{2.79} is used in \cite[Sect.\ 1.4]{FOT11} as the basic definition for a bounded operator to 
be {\it Markovian} (an alternative name for sub-Markovian). 

In addition, one notes that if $A \in \cB\big(L^2(M; d\mu)\big)$ pointwise 
dominates the operator $B \in \cB\big(L^2(M; d\mu)\big)$, and if $A$ is an $L^\infty$-contraction, 
then so is $B$, that is, if 
\begin{align}
\begin{split} 
& |B f| \leq A |f|, \; f \in L^2(M; d\mu) \, \text{ and $A$ is an 
$L^\infty$-contraction,}    \lb{2.80} \\
& \quad \text{then $B$ is an $L^\infty$-contraction,}
\end{split} 
\end{align} 
as the following elementary observation shows:
\begin{align} 
\begin{split}
& \|B f \|_{L^\infty(M;d\mu)} = \| |B f| \|_{L^\infty(M;d\mu)} \leq 
\| A |f| \|_{L^\infty(M;d\mu)} \leq \| |f| \|_{L^\infty(M;d\mu)}   \lb{2.81} \\
& \quad  = \|f \|_{L^\infty(M;d\mu)},   \quad 
f \in L^2(M; d\mu)\big) \cap L^\infty(M; d\mu)\big).    
\end{split} 
\end{align}

\medskip

Historically, the circle of ideas described in this section originated with Perron \cite{Pe07} and 
Frobenius \cite{Fr08}--\cite{Fr12} in the context of matrices, and was first extended to the 
infinite-dimensional case of compact integral operators by Jentzsch \cite{Je12} (see also 
\cite[p.\ 183]{Jo82}). The issue, apparently, was revived within the mathematical physics community 
by Glimm and Jaffe \cite{GJ70} around 1970 in the context of 
nondegenerate ground states (i.e., a simple eigenvalue at the bottom of the spectrum) of 
quantum (especially, quantum field theoretic) Hamiltonians. In this connection we refer, for instance, to 
\cite{AB80}, \cite{AB06}, \cite{Ar84}, \cite{Ar84a}, \cite{Ar87}, \cite{AB92}, \cite{AS57}, \cite{BKR80}, 
\cite{CM09}, \cite{Da73}, \cite[Ch.\ 7]{Da80}, \cite[Ch.\ 13]{Da07}, \cite{Fa72}, \cite[Sects.\ 8, 10]{Fa75}, 
\cite{FS75}, \cite{Fr09a}, \cite{Ge84}, \cite[Sect.\ 3.3]{GJ81}, \cite{Go77}, \cite{Gr72}, \cite{HS78}, 
\cite{HSU77}, \cite{KR80}, \cite{KR81}, \cite{Kr64}, \cite{KR50}, \cite{MVV05}, \cite{MO86}, \cite{Ou92}, 
\cite{Ou96}, \cite[Chs.\ 2, 3]{Ou05}, \cite{Pe81}, \cite{RL88}, \cite[Sect.\ XIII.12]{RS78}, \cite{Ro01}, 
\cite{Sc85}, \cite{Si73}, \cite{Si77a}, \cite{Si79}, \cite[Sect.\ 10.5]{We80}, \cite[Sect.\ 3.3]{Ya10} and 
the references cited therein for the basics and some of the applications of this subject.

%%%%%%%%%%%%%%%%%%%%%%%%%%%%%%%%%%%%%%
%%%%%%%%%%%%%%%%%%%%%%%%%%%%%%%%%%%%%%
\section{Robin-Type Boundary Conditions for Matrix-Valued  
Divergence Form Elliptic Partial Differential Operators}  \lb{s3}
%%%%%%%%%%%%%%%%%%%%%%%%%%%%%%%%%%%%%%
%%%%%%%%%%%%%%%%%%%%%%%%%%%%%%%%%%%%%%

In this section we develop the basics for divergence form elliptic partial differential 
operators with (nonlocal) Robin-type boundary conditions in bounded Lipschitz domains. In fact, we 
will go a step further in this section and develop the theory in the vector-valued case as this is certainly of interest in its own right. Thus, we will focus on $m\times m$, $m\in\bbN$, matrix-valued differential expressions $L$ which act as 
\begin{equation}\label{L-def1}
Lu = - \biggl(\sum_{j,k=1}^n\partial_j\bigg(\sum_{\beta = 1}^m a^{\alpha,\beta}_{j,k}\partial_k u_\beta\bigg)
\bigg)_{1\leq\alpha\leq m},\quad u=(u_1,\dots,u_m). 
\end{equation}

For basic facts on Sobolev spaces and Dirichlet and Neumann trace operators, as well as the choice of notation used below, we refer to Appendix \ref{sA}. For the basics on sesquilinear forms and operators associated with them we refer to Appendix \ref{sB}. Moreover, for the definition of Lipschitz domains 
$\Om$ and the associated Sobolev spaces on $\Om$, $H^s(\Om)$, and its boundary $\dOm$, 
$H^r(\dOm)$, we refer to Appendix \ref{sA}.  

In the remainder of this section we make the following assumption:

%%%%%%%%%%%%%%%%%%%%%%%%%%%%%%%%%%%%%%%
\begin{hypothesis} \lb{h3.1}
Let $n\in\bbN$, $n\geq 2$, and assume that $\Om\subset{\bbR}^n$ is
a bounded Lipschitz domain.
\end{hypothesis}
%%%%%%%%%%%%%%%%%%%%%%%%%%%%%%%%%%%%%%%

For simplicity of notation we will denote the identity operators in $\LOm$ and
$\LdOm$ by $I_{\Om}$ and $I_{\dOm}$, respectively. Also, in the sequel, the
sesquilinear form 
\begin{equation}
\langle \dott, \dott \rangle_{s}^{(m)}={}_{H^{s}(\dOm)^m}\langle\dott,\dott
\rangle_{H^{-s}(\dOm)^m}\colon H^{s}(\dOm)^m\times H^{-s}(\dOm)^m
\to \bbC, \quad s\in [0,1],
\end{equation} 
(antilinear in the first, linear in the second factor), will denote the duality
pairing between $H^s(\dOm)$ and 
\begin{equation}\lb{3.2}
H^{-s}(\dOm)^m=\big(H^s(\dOm)^m\big)^*,\quad s\in [0,1],
\end{equation} 
such that 
\begin{align}\lb{3.3}
\begin{split} 
& \langle f,g\rangle_s^{(m)}=\int_{\dOm} d^{n-1}\omega(\xi)\,\ol{f(\xi)}\cdot g(\xi),  
\\
& f\in H^s(\dOm)^m,\,g\in L^2(\dOm;d^{n-1}\omega)^m\hookrightarrow 
H^{-s}(\dOm)^m, \; s\in [0,1],
\end{split} 
\end{align}
where $d^{n-1}\omega$ stands for the surface measure on $\dOm$.

Below we introduce the class of multipliers on generic Banach spaces
(though we are primarily concerned with Sobolev spaces; for an authoritarian 
treatment of this topic the interested reader is referred to \cite{MS09}).

%%%%%%%%%%
\begin{definition}\label{Pd-44}
Let $(X,\|\cdot\|_{X})$, $(Y,\|\cdot\|_{Y})$ be two given Banach spaces
of distributions in a domain $\Omega\subseteq\bbR^n$ with the 
property that $C^\infty(\overline{\Omega})$ is dense in both of these spaces.
In this context, define ${\mathcal{M}}(X\rightarrow Y)$, the space  
multipliers from $X$ to $Y$, as follows: If $M_a$ denotes the operator of pointwise
multiplication on $C^\infty(\overline{\Omega})$ by the function $a$, 
then $a\in{\mathcal{M}}(X\rightarrow Y)$ indicates that $M_a$ may be extended 
to an operator in $\cB(X,Y)$. Whenever $a\in{\mathcal{M}}(X\rightarrow Y)$, we set 
\begin{equation}\label{hgac}
\|a\|_{{\mathcal{M}}(X\rightarrow Y)}:=\|M_a\|_{\cB(X,Y)}.
\end{equation} 
Finally, we abbreviate ${\mathcal{M}}(X):={\mathcal{M}}(X\to X)$.
\end{definition}
%%%%%%%%%%

Next, we wish to describe a weak version of the normal trace operator associated
with an $m\times m$, $m\in{\mathbb{N}}$, second-order system in divergence form 
\eqref{L-def1}, considered in a domain $\Omega\subseteq{\mathbb{R}}^n$. 
To set the stage, assume Hypothesis \ref{h2.1} and suppose that some 
$s\in(0,1)$ has been fixed with the property that the tensor coefficient 
\begin{equation}\label{igc-Ts.55}
A:=\bigl(a_{j,k}^{\alpha,\beta}\bigr)
_{\substack{1\leq\alpha,\beta\leq m \\ 1\leq j,k\leq n}}
\end{equation}
satisfies 
\begin{equation}\label{A-M1}
A\in{\mathcal{M}}(H^{s-1/2}(\Om)^m).
\end{equation}
One observes that the inclusion
\begin{equation}\label{inc-1}
\iota:H^{s_0}(\Omega)\hookrightarrow \bigl(H^r(\Omega)\bigr)^*,\quad
s_0>-1/2, \; r>1/2,
\end{equation} 
is well-defined and bounded. We then introduce the weak Neumann trace operator 
\begin{equation}\label{2.8}
\wti\ga_N\colon\big\{u\in H^{s+1/2}(\Om)^m\,\big|\,Lu\in H^{s_0}(\Om)^m\big\}
\to H^{s-1}(\dOm)^m,\quad s_0>-1/2,
\end{equation} 
as follows: Given $u\in H^{s+1/2}(\Om)^m$ with $Lu\in H^{s_0}(\Om)^m$
for some $s_0>-1/2$ and $s\in(0,1)$, we set (with $\iota$ as in
\eqref{inc-1} for $r:=3/2-s>1/2$) 
\begin{align}\label{2.9}
\langle\phi,\wti\ga_N u \rangle_{1-s}^{(m)} &= 
{}_{H^{1/2-s}(\Om)^m}\langle D\Phi, ADu\rangle_{(H^{1/2-s}(\Om)^m)^*}
\nonumber\\  
& \quad  - {}_{H^{3/2-s}(\Om)^m}\langle\Phi,\iota(Lu)\rangle_{(H^{3/2-s}(\Om)^m)^*},
\end{align} 
for all $\phi\in H^{1-s}(\dOm)^m$ and $\Phi\in H^{3/2-s}(\Om)^m$ such that
$\ga_D\Phi=\phi$. Above, we used 
\begin{equation}\label{ihah}
Du:=\bigl(\partial_j u_\alpha\bigr)_{\substack{1\leq\alpha\leq m\\ 1\leq j\leq n}}
\end{equation} 
to denote the Jacobian matrix of the vector-valued function 
$u=(u_\alpha)_{1\leq \alpha\leq m}$, and we employed the convention that, 
in general, 
\begin{equation}\label{A-def}
A\zeta:=\Bigl(\sum_{k=1}^n\sum_{\beta=1}^m a^{\alpha,\beta}_{j,k}\zeta^\beta_k
\Bigr)_{\substack{1\leq\alpha\leq m\\ 1\leq j\leq n}} \, \text{ for all } \, 
\zeta = \big(\zeta^\beta_k\big)_{\substack{1 \leq \beta \leq m\\ 1 \leq k \leq n}}
\in{\mathbb{C}}^{\, n\times m}. 
\end{equation} 
In addition, we will later employ the convention 
\begin{equation}\label{A-defAA}
\bigl\langle\eta,A\zeta\bigr\rangle:=\sum_{j,k=1}^n\sum_{\alpha,\beta=1}^m
a^{\alpha,\beta}_{j,k}\zeta^\beta_k\eta^\alpha_j,   
\end{equation} 
for all $\eta = \big(\eta^\alpha_j \big)_{\substack{1 \leq \alpha \leq m\\ 1 \leq j \leq n}}
\in{\mathbb{C}}^{\, n\times m}$, 
$\zeta = \big(\zeta^\beta_k\big)_{\substack{1 \leq \beta \leq m\\ 1 \leq k \leq n}}
\in{\mathbb{C}}^{\, n\times m}$. 

Returning to the mainstream discussion, we note that the first pairing on 
the right-hand side of \eqref{2.9} is meaningful granted \eqref{A-M1}, since 
\begin{equation}\label{2.9JJ}
\bigl(H^{1/2-s}(\Om)^m\bigr)^*=H^{s-1/2}(\Om)^m.
\end{equation} 
Moreover, the definition \eqref{2.9} is independent of the particular
extension $\Phi$ of $\phi$. Indeed, if $\Phi'\in H^{3/2-s}(\Om)^m$ 
is another vector-valued function such that $\ga_D\Phi'=\phi$
then $\Psi:=\Phi-\Phi'\in H^{3/2-s}(\Om)^m$ satisfies $\ga_D\Psi=0$.
Hence $\Psi\in H^{3/2-s}_0(\Om)^m$ and, as such, there exists a sequence
$\{\psi_j\}_{j\in{\mathbb{N}}}\subseteq C^\infty_0(\Omega)^m$ with the
property that  
\begin{equation}\label{Mar-a.1}
\psi_j\rightarrow\Psi\,\,\mbox{ in $H^{3/2-s}(\Om)^m$, as }\,\,j\to\infty.
\end{equation} 
Consequently, given that $D\psi_j\rightarrow D\Psi$ in $H^{1/2-s}(\Om)^m$ 
as $j\to\infty$, we may compute 
\begin{align}\label{Mar-a.2}
& {}_{H^{3/2-s}(\Om)^m}\langle\Psi,\iota(Lu)\rangle_{(H^{3/2-s}(\Om)^m)^*}
\nonumber\\
& \quad 
=\lim_{j\to\infty}
{}_{H^{3/2-s}(\Om)^m}\langle\psi_j,\iota(Lu)\rangle_{(H^{3/2-s}(\Om)^m)^*}
\nonumber\\
& \quad 
=\lim_{j\to\infty}
{}_{{\mathcal{D}}(\Om)^m}\langle\psi_j,Lu\rangle_{({\mathcal{D}}(\Om)^m)'}
\nonumber\\ 
& \quad 
= \lim_{j\to\infty}
{}_{{\mathcal{D}}(\Om)^m}\langle D\psi_j,ADu\rangle_{({\mathcal{D}}(\Om)^m)'}
\nonumber\\
& \quad 
= \lim_{j\to\infty}
{}_{H^{1/2-s}(\Om)^m}\langle D\psi_j,ADu\rangle_{(H^{1/2-s}(\Om)^m)^*}
\nonumber\\
& \quad 
= {}_{H^{1/2-s}(\Om)^m}\langle D\Psi, ADu\rangle_{(H^{1/2-s}(\Om)^m)^*},
\end{align} 
with $\cD(\Om)$ denoting the space of test functions (i.e., 
$C_0^\infty(\Omega)$) equipped with the usual inductive limit topology
(so that, in particular, $\cD^\prime(\Om)=\bigr(C_0^\infty(\Omega)\bigr)^\prime$ 
is the space of distributions in $\Om$). Now \eqref{Mar-a.2} gives that
${}_{H^{1/2-s}(\Om)^m}\langle D\Psi, ADu\rangle_{(H^{1/2-s}(\Om)^m)^*}
- {}_{H^{3/2-s}(\Om)^m}\langle\Psi,\iota(Lu)\rangle_{(H^{3/2-s}(\Om)^m)^*}=0$
which ultimately shows that 
\begin{align}\label{Mar-a.4}
& {}_{H^{1/2-s}(\Om)^m}\langle D\Phi, ADu\rangle_{(H^{1/2-s}(\Om)^m)^*}
- {}_{H^{3/2-s}(\Om)^m}\langle\Phi,\iota(Lu)\rangle_{(H^{3/2-s}(\Om)^m)^*}
\\
& \quad 
={}_{H^{1/2-s}(\Om)^m}\langle D\Phi', ADu\rangle_{(H^{1/2-s}(\Om)^m)^*}
- {}_{H^{3/2-s}(\Om)^m}\langle\Phi',\iota(Lu)\rangle_{(H^{3/2-s}(\Om)^m)^*}.
\nonumber
\end{align} 
From this, the desired conclusion (pertaining to the unambiguity of 
defining $\wti\ga_N u$ as in \eqref{2.9}) follows. 

One recalls that $\gamma_D$ has a linear right-inverse, that is, 
there exists a linear and bounded operator 
\begin{equation}\label{Mar-a.5}
{\mathcal{E}}:H^s(\dOm)\to H^{s+1/2}(\Om),\quad 0<s<1,
\end{equation} 
which is universal (in the sense that it does not depend on $s\in(0,1)$)
and satisfies 
\begin{equation}\label{Mar-a.6}
\gamma_D({\mathcal{E}}\phi)=\phi, \quad \phi\in H^s(\dOm),\quad 0<s<1.
\end{equation} 
Given an arbitrary $\phi\in H^s(\dOm)$ and choosing 
$\Phi:={\mathcal{E}}\phi\in H^{s+1/2}(\Om)$ in \eqref{2.7} then allows us to estimate 
\begin{equation}\label{Mar-a.7}
\|\wti\ga_N u\|_{H^{s-1}(\dOm)^m}\leq C\bigl(
\|A\|_{{\mathcal{M}}(H^{s-1/2}(\Om)^m)}\|u\|_{H^{s+1/2}(\Om)^m}
+\|Lu\|_{H^{s_0}(\Om)^m}\bigr)
\end{equation} 
for every $u$ in the domain of $\wti\ga_N$. This proves that the 
operator $\wti\ga_N$ in \eqref{2.9} is well-defined, linear, and bounded. 

It is instructive to point out that, in the case when $s=1/2$, condition 
\eqref{A-M1} reduces precisely to  
\begin{equation}\label{A-M2}
A\in L^{\infty}(\Om; d^nx)^{m\times m}.
\end{equation} 
In particular,  
\begin{equation}\label{2.8BBB}
\wti\ga_N\colon\big\{u\in H^{1}(\Om)^m\,\big|\,Lu\in H^{s_0}(\Om)^m\big\}
\to H^{-1/2}(\dOm)^m,\quad s_0>-1/2,
\end{equation} 
is well-defined, linear, and bounded, whenever \eqref{A-M2} holds.

%%%%%%%%%%%%%%%%%%%%%%%%%%%%%
\begin{hypothesis} \lb{h3.2}
Assume Hypothesis \ref{h3.1}, suppose that $\delta>0$ is a given number,
consider $m\in{\mathbb{N}}$, and assume 
that $\Theta\in\cB\big(H^{1/2}(\dOm)^m,H^{-1/2}(\dOm)^m\big)$
is a self-adjoint operator $($in the sense discussed in \eqref{B.5}$)$ which can be written as 
\begin{equation}\label{Filo-1}
\Theta=\Theta^{(1)}+\Theta^{(2)}+\Theta^{(3)},
\end{equation} 
where $\Theta^{(j)}$, $j=1,2,3$, have the following properties: There exists 
a closed sesquilinear form $\gq_{\dOm}^{(0)}$ in $\LdOm^m$,
with domain $H^{1/2}(\dOm)^m\times H^{1/2}(\dOm)^m$, which is bounded from below
by $c_{\dOm}\in\bbR$ $($hence, $\gq_{\dOm}^{(0)}$ is symmetric\,$)$ 
such that if $\Theta_{\dOm}^{(0)} \geq c_{\dOm}I_{\dOm}$ denotes
the self-adjoint operator in $\LdOm^m$ uniquely associated with $\gq_{\dOm}^{(0)}$
$($cf. \eqref{B.25}$)$, then $\Theta^{(1)}=\wti\Theta_{\dOm}^{(0)}$, the extension
of $\Theta_{\dOm}^{(0)}$ to an operator in $\cB\big(H^{1/2}(\dOm)^m,H^{-1/2}(\dOm)^m\big)$
$($as discussed in \eqref{B.24a} and \eqref{B.28a}$)$. In addition,  
\begin{equation}\label{Filo-2}
\Theta^{(2)}\in\cB_{\infty}\big(H^{1/2}(\dOm)^m,H^{-1/2}(\dOm)^m\big),
\end{equation} 
whereas $\Theta^{(3)}\in\cB\big(H^{1/2}(\dOm),H^{-1/2}(\dOm)\big)$ satisfies  
\begin{equation}\label{Filo-3}
\big\|\Theta^{(3)}\big\|_{\cB(H^{1/2}(\dOm)^m,(H^{-1/2})^m(\dOm))}<\delta.
\end{equation}
\end{hypothesis}
%%%%%%%%%%%%%%%%%%%%%%%%%%%%%%

We record the following useful result involving the Dirichlet trace operator $\ga_D$.

%%%%%%%%%%%%%%%%%%%%
\begin{lemma}\lb{l3.3}
Assume Hypothesis \ref{h3.1} and fix $m\in{\mathbb{N}}$. Then for every $\varepsilon>0$
there exists a $\beta(\varepsilon)>0$
$($with $\beta(\varepsilon)\underset{\varepsilon\downarrow 0}{=}
\Oh(1/\varepsilon)$$)$ such that 
\begin{equation} \lb{3.7}
\|\gamma_D u\|_{\LdOm^m}^2 \le
\varepsilon \|Du\|_{\LOm^m}^2 + \beta(\varepsilon) \|u\|_{\LOm^m}^2, 
\quad u\in H^1(\Om)^m.
\end{equation}
\end{lemma}
%%%%%%%%%%%%%%%%%%%%
\begin{proof}
The case when $m=1$ (corresponding to the situation when $u$ is scalar-valued
and hence $Du=\nabla u$) has been treated in \cite{GM09}. The vector-valued
case then follows from this by summing up such scalar estimates. 
\end{proof}
%%%%%%%%%%%%%%%%%%%%

Lemma \ref{l3.3} is a key ingredient in proving the $H^1(\Om)$-coercivity of the sesquilinear form 
$\gQ_{\Theta,\Om,L}(\dott,\dott)$ in Theorem \ref{t3.4} below.  Before stating it, we recall our 
conventions \eqref{A-def} and \eqref{A-defAA} and introduce the following hypothesis.

%%%%%%%%%%%%%%%%%%%%%%%%%%%%%%%%%%%%%%%
\begin{hypothesis} \lb{h.MM}
Assume Hypothesis \ref{h3.1} and suppose that $L$ is an $m\times m$,  
$m\in{\mathbb{N}}$, second-order system in divergence form expressed 
as in \eqref{L-def1} for a tensor coefficient as in \eqref{igc-Ts.55}. 
In addition, assume that \eqref{A-M2} holds, that is, 
$A\in L^{\infty}(\Om; d^nx)^{m\times m}$, and that the tensor coefficient 
$A$ satisfies the strong Legendre ellipticity condition 
$($for some $a_0>0$$)$ 
\begin{equation}\label{strongellipticity}
{\rm Re}\,\bigl\langle\overline{\zeta},A\zeta\bigr\rangle
={\rm Re}\,\bigg(\sum_{j,k=1}^n\sum_{\alpha,\beta=1}^m
a_{j,k}^{\alpha,\beta}\zeta_k^\beta\overline{\zeta_j^\alpha}\,\bigg) 
\geq a_0 |\zeta|^2,
\end{equation} 
for every $\zeta= \big(\zeta_j^\alpha\big)_{\substack{1\leq\alpha\leq m \\ 1\leq j \leq n}} 
\in{\mathbb{C}}^{n\times m}$. Finally, suppose $A$ satisfies the symmetry condition 
\begin{equation}\label{symm-22}
a^{\alpha,\beta}_{j,k}=\ol{a^{\beta,\alpha}_{k,j}},\quad
\alpha,\beta\in\{1,\dots,m\}, \; j,k\in\{1,\dots,n\}.
\end{equation}
\end{hypothesis}
%%%%%%%%%%%%%%%%%%%%%%%%%%%%%%%%%%%%%%%

The following result extends work from \cite{GM09} carried out for the special case of the Laplacian 
$-\Delta$. 

%%%%%%%%%%%%%%%%%%%%
\begin{theorem}\lb{t3.4}
Assume Hypothesis \ref{h3.2}, where the number $\delta>0$ is taken
to be sufficiently small relative to the Lipschitz character of $\Om$, more precisely, 
suppose that $0<\delta \leq \f{1}{6} \|\gamma_D\|^{-2}_{\cB(H^1(\Om),H^{1/2}(\dOm))}$.
In addition, assume Hypothesis \ref{h.MM} and consider the sesquilinear 
form $\gQ_{\Theta,\Om}(\dott,\dott)$ defined on $H^1(\Om)^m\times H^1(\Om)^m$ by 
\begin{align} \lb{3.8}
\begin{split} 
\gQ_{\Theta,\Om} (u,v):=\int_{\Om} d^nx\,
\bigl\langle\overline{(Du)(x)},A(x)(Dv)(x)\bigr\rangle
+\big\langle\gamma_D u,\Theta\gamma_D v \big\rangle_{1/2}^{(m)},&   \\
u,v\in H^1(\Om)^m.&  
\end{split} 
\end{align} 
Then there exists $\kappa>1$ with the property that the form 
\begin{equation} \label{FI-5}
\gQ_{\Theta,\Om,\kappa}(u,v):=\gQ_{\Theta,\Om}(u,v)+\kappa\,(u,v)_{L^2(\Om;d^nx)^m},
\quad u,v\in H^1(\Om)^m,
\end{equation} 
is $H^1(\Om)^m$-coercive. As a consequence, the form $\gQ_{\Theta,\Om}(\dott,\dott)$ 
in \eqref{3.8} is symmetric, $H^1(\Om)^m$-bounded, bounded from below, and closed 
in $L^2(\Om;d^n x)^m$.
\end{theorem}
%%%%%%%%%%%%%%%%%%%%

In the proof of Theorem \ref{t3.4}, the following abstract functional analytic result
is going to be useful.

%%%%%%%%%%%%%%%%%%%%
\begin{lemma}[\cite{GM09}]\lb{l.WW}
Let $\cV$ be a reflexive Banach space, $\cW$ a Banach space, assume 
that $K\in\cB_{\infty}(\cV,\cV^*)$, and that $T\in\cB(\cV,\cW)$ is one-to-one. 
Then for every $\varepsilon>0$ there exists $C_{\varepsilon}>0$ such that 
\begin{equation}\lb{K-u1}
\big|{}_{\cV}\langle u,Ku\rangle_{\cV^*}\big|\leq\varepsilon\|u\|^2_{\cV}
+C_{\varepsilon}\|Tu\|^2_{\cW},\quad u\in\cV.
\end{equation}
\end{lemma}
%%%%%%%%%%%%%%%%%%%%%

We are now ready to present the 

\vskip 0.10in
\noindent{\it Proof of Theorem \ref{t3.4}}.
We shall show that $\kappa>0$ can be chosen large enough so that 
\begin{align}\lb{2.22Y.1} 
&\frac16\|u\|^2_{H^1(\Om)^m}\leq\frac{1}{3}\int_{\Om}d^nx\,|(Du)(x)|^2
+\frac{\kappa}{3}\int_{\Om}d^nx\,|u(x)|^2+
\big\langle\gamma_D u,\Theta^{(j)}\gamma_D u \big\rangle_{1/2}^{(m)},    \no \\
& \hspace*{7cm}  u\in H^1(\Om)^m, \; j=1,2,3, 
\end{align} 
where $\Theta^{(j)}$, $j=1,2,3$, are as introduced in Hypothesis \ref{h3.2}.
Summing up these three inequalities then proves that the form \eqref{FI-5} is
indeed $H^1(\Om)^m$-coercive. To this end, we assume first $j=1$ and recall
that there exists $c_{\Theta_0}\in\bbR$ such that 
\begin{equation}\label{PdPd-1}
\big\langle\ga_D u,\Theta^{(1)}\,\ga_D u\big\rangle_{1/2}^{(m)}
\geq c_{\Theta_0}\|\ga_D u\|_{L^2(\dOm;d^{n-1}\omega)^m}^2,\quad u\in H^1(\Om)^m.
\end{equation} 
Thus, in this case, it suffices to show that
\begin{align}\lb{PdPd-2}
& \max\,\{-c_{\Theta_0}\,,\,0\}\,\|\ga_D u\|_{L^2(\dOm;d^{n-1}\omega)^m}^2
+\frac{1}{6}\|u\|^2_{H^1(\Om)}
\no \\
&\quad 
\leq\frac{1}{3}\int_{\Om}d^nx\,|(Du)(x)|^2
+\frac{\kappa}{3}\int_{\Om}d^nx\,|u(x)|^2,
\quad u\in H^1(\Om)^m,
\end{align} 
or, equivalently, that   
\begin{align}\lb{PdPd-3}
& \max\,\{-c_{\Theta_0}\,,\,0\}\,\|\ga_D u\|_{L^2(\dOm;d^{n-1}\omega)^m}^2
\no \\
& \quad\leq\frac{1}{6}\int_{\Om}d^nx\,|(Du)(x)|^2
+\frac{2\kappa-1}{6}\int_{\Om}d^nx\,|u(x)|^2,\quad u\in H^1(\Om)^m,
\end{align} 
with the usual convention, 
\begin{equation}
\|u\|^2_{H^1(\Om)^m}=\|Du\|^2_{_{L^2(\Om;d^n x)^m}} 
+ \|u\|^2_{_{L^2(\Om;d^n x)^m}}, \quad u\in H^1(\Om)^m.
\end{equation} 
Now, the fact that there exists $\kappa>0$ for which \eqref{PdPd-3} holds
follows directly from Lemma \ref{l3.3}.

Next, one observes that in the case where $j=2,3$, the estimate \eqref{2.22Y.1} is
implied by
\begin{equation}\lb{2.22Y.2}
\big|\big\langle\gamma_D u,\Theta^{(j)}\gamma_D u\big\rangle_{1/2}^{(m)}\big|
\leq\frac{1}{6}\int_{\Om}d^nx\,|(Du)(x)|^2
+\frac{2\kappa-1}{6}\int_{\Om}d^nx\,|u(x)|^2,\quad u\in H^1(\Om)^m,
\end{equation} 
or, equivalently, by  
\begin{equation}\lb{2.22Y.3}
\big|\big\langle\gamma_D u,\Theta^{(j)}\gamma_D u\big\rangle_{1/2}^{(m)}\big|
\leq\frac{1}{6}\|u\|^2_{H^1(\Om)^m}
+\frac{\kappa-1}{3}\|u\|^2_{L^2(\Om;d^nx)^m},\quad u\in H^1(\Om)^m.
\end{equation} 
When $j=2$, in which case 
$\Theta^{(2)} \in\cB_{\infty}\big(H^1(\Om)^m,\bigl(H^1(\Om)^m\bigr)^*\big)$,
we invoke Lemma \ref{l.WW} with  
\begin{equation}\lb{K-u5}
\cV:=H^1(\Om)^m,\quad \cW:=L^2(\Om,d^nx)^m,
\end{equation} 
and, with $\gamma_D\in\cB\big(H^1(\Om)^m,H^{1/2}(\dOm)^m\big)$ denoting the 
Dirichlet trace (acting componentwise), 
\begin{equation}\lb{PdPd-4}
K:=\gamma_D^*\Theta_2\gamma_D
\in\cB_{\infty}\big(H^1(\Om)^m,\bigl(H^1(\Om)^m\bigr)^*\big),
\quad T:=\iota:H^1(\Om)^m\hookrightarrow L^2(\Om,d^nx)^m,
\end{equation} 
the inclusion operator. Then, with $\varepsilon=1/6$ and
$\kappa:=3C_{1/6}+1$, the estimate \eqref{K-u1} yields \eqref{2.22Y.3} for $j=2$.

Finally, consider \eqref{2.22Y.3} in the case where $j=3$ and note that 
by hypothesis, 
\begin{align}\lb{2.22Y.4}
& \big|\big\langle\gamma_D u,\Theta^{(3)}\gamma_D u\big\rangle_{1/2}^{(m)}\big|
\leq \big\|\Theta^{(3)}\big\|_{\cB(H^{1/2}(\dOm)^m,H^{-1/2}(\dOm)^m)}
\|\gamma_D u\|^2_{H^{1/2}(\dOm)^m}
\no \\
& \quad \leq \delta \|\gamma_D\|^2_{\cB(H^1(\Om)^m,H^{1/2}(\dOm)^m)}\|u\|^2_{H^1(\Om)^m}, 
\quad u\in H^1(\Om)^m.
\end{align} 
Thus \eqref{2.22Y.3} also holds for $j=3$ if
\begin{equation}\lb{2.22Y.5}
0<\delta \leq \f{1}{6} \|\gamma_D\|^{-2}_{\cB(H^1(\Om),H^{1/2}(\dOm))} \, 
\text{ and } \, \kappa>1.  
\end{equation} 
This completes the justification of the estimate \eqref{2.22Y.1}. In turn, this 
further implies, with the help of the strong Legendre ellipticity condition, that 
\begin{align}\lb{2.22Y.1a}
\frac16\|u\|^2_{H^1(\Om)^m}& \leq\frac{1}{3\lambda}
\int_{\Om} d^nx\,\bigl\langle\overline{(Du)(x)},A(x)(Du)(x)\bigr\rangle
\no \\
& \quad +\frac{\kappa}{3}\int_{\Om}d^nx\,|u(x)|^2+
\big\langle\gamma_D u,\Theta^{(j)}\gamma_D u \big\rangle_{1/2}^{(m)},  
\\
& \hspace*{2.96cm}  u\in H^1(\Om)^m, \; j=1,2,3,   \no 
\end{align} 
where $\lambda>0$ is as in \eqref{strongellipticity}. The estimate \eqref{2.22Y.1a}
establishes the claim about the sesquilinear form in \eqref{FI-5}.
Moreover, the symmetry of the sesquilinear form $\gQ_{\Theta,\Om,L}(\dott,\dott)$ 
from \eqref{3.8} is a direct consequence of \eqref{symm-22}. Finally, the 
remaining claims in the statement of Theorem \ref{t3.4} are implicit in what 
we have proved so far, and this finishes the proof of Theorem \ref{t3.4}.
\hfill$\square$
\vskip 0.10in

Next, we turn to a discussion of the realization of an $m\times m$ system $L$ as a 
self-adjoint operator on $L^2(\Om;d^nx)^m$ when equipped with certain nonlocal 
Robin boundary conditions in a bounded Lipschitz subdomain $\Omega$ of $\bbR^n$. 
Below, $\wti\gamma_N$ denotes the weak Neumann trace operator discussed in 
\eqref{2.8}--\eqref{2.8BBB}.

%%%%%%%%%%%%%%%%%%%%
\begin{theorem}\lb{t3.5BBB}
Assume Hypothesis \ref{h3.2}, where the number $\delta>0$ is taken
to be sufficiently small as in Theorem \ref{t3.4}. In addition, assume 
Hypothesis \ref{h.MM}. Then $L_{\Theta,\Om}$, the $L^2$-realization of $L$ 
equipped with a nonlocal Robin boundary condition in $L^2(\Om;d^nx)^m$, defined by 
\begin{align}
& L_{\Theta,\Om} = L,   \no \\
& \, \dom(L_{\Theta,\Om})=
\big\{u\in H^1(\Om)^m \,\big|\,Lu\in L^2(\Om;d^nx)^m,     \lb{3.10BBB} \\
& \hspace*{2.3cm}\big(\wti\gamma_N+\Theta\gamma_D\big)u=0 
\text{ in $H^{-1/2}(\dOm)^m$}\big\},   \no
\end{align} 
is self-adjoint and bounded from below. Moreover, 
\begin{equation}\lb{3.11BBB}
\dom\big(|L_{\Theta,\Om}|^{1/2}\big)=H^1(\Om)^m,
\end{equation} 
and $L_{\Theta,\Om}$, has purely discrete spectrum bounded from below, in particular, 
\begin{equation}\lb{3.12}
\sigma_{\rm ess}(L_{\Theta,\Om})=\emptyset.
\end{equation} 
Finally, $L_{\Theta,\Om}$ is the operator uniquely associated with the sesquilinear 
form $\gQ_{\Theta,\Om}$ in Theorem \ref{t3.4}. 
\end{theorem}
%%%%%%%%%%%%%%%%%%%%
\begin{proof}
Denote by $\gQ_{\Theta,\Om}(\dott,\dott)$ the sesquilinear form introduced 
in \eqref{3.8}. From Theorem \ref{t3.4}, we know that $\gQ_{\Theta,\Om}$ is symmetric,
$H^1(\Om)^m$-bounded, bounded from below, as well as densely defined and closed 
in $\LOm^m\times\LOm^m$. Thus, if as in \eqref{B.30}, we now introduce the operator
$L_{\Theta,\Om}$ in $L^2(\Om;d^n x)^m$ by 
\begin{align}\lb{2.31}
& \dom(L_{\Theta,\Om})=\bigg\{v\in H^1(\Om)^m\,\bigg|\, 
\text{there exists some $w_v\in L^2(\Om;d^n x)^m$ such that}
\nonumber\\[4pt]
& \,\,\,
\int_{\Om} d^nx\,\bigl\langle\overline{Dw},A(Dv)\bigr\rangle
+ \big\langle\gamma_D w,\Theta\gamma_D v\big\rangle_{1/2}^{(m)}
=\int_{\Om}d^n x\,\ol{w}\cdot w_v\text{ for all $w\in H^1(\Om)^m$}\bigg\},
\nonumber\\[4pt]
& L_{\Theta,\Om}u=w_u,\quad u\in\dom(L_{\Theta,\Om}),
\end{align} 
it follows from \eqref{B.19}--\eqref{B.31} (cf., in particular \eqref{B.25})
that $L_{\Theta,\Om}$ is self-adjoint and bounded from below in
$L^2(\Om;d^n x)^m$ and that \eqref{3.11BBB} holds. Next we recall that 
\begin{equation}\label{H-zer}
H_0^1(\Om)^m=\big\{u\in H^1(\Om)^m\,\big|\,\gamma_Du=0\mbox{ on }\partial\Omega\big\},
\end{equation} 
where, as usual, the Dirichlet trace operator $\gamma_D$ acts componentwise. 
Taking $v\in C_0^\infty(\Omega)^m\hookrightarrow H^1_0(\Om)^m
\hookrightarrow H^1(\Om)^m$, one concludes that if $u\in\dom(L_{\Theta,\Om})$ then  
\begin{equation}\lb{2.32}
\int_{\Om}d^nx\,{\ol v}\cdot w_u = \int_{\Om}d^nx\,{\ol v}\cdot Lu\,\
\text{ for all $v\in C_0^\infty(\Om)$, hence }\,w_u
= Lu\,\text{ in }\,\bigl(\cD(\Om)^m\bigr)^\prime.
\end{equation} 
Going further, suppose that $u\in\dom(L_{\Theta,\Om})$ and
$v\in H^1(\Om)^m$. We recall that $\gamma_D\colon H^1(\Om)^m\to H^{1/2}(\dOm)^m$
boundedly and compute 
\begin{align}
\int_{\Om} d^nx\,\bigl\langle\overline{Dv},A(Du)\bigr\rangle
& = \int_{\Om}d^n x\,{\ol v}\cdot Lu
+\langle\gamma_D v,\wti\gamma_N u\rangle_{1/2}^{(m)}
\nonumber\\[4pt]
& =\int_{\Om}d^n x\,{\ol v}\cdot w_u
+\big\langle\gamma_D v,\big(\wti\gamma_N+\Theta\gamma_D\big)u\big\rangle_{1/2}^{(m)}
-\big\langle\gamma_D v,\Theta\gamma_D u\big\rangle_{1/2}^{(m)}
\nonumber\\[4pt]
& =\int_{\Om} d^nx\,\bigl\langle\overline{Dv},A(Du)\bigr\rangle
+\big\langle\gamma_D v,\big(\wti\gamma_N+\Theta\gamma_D\big)u\big\rangle_{1/2}^{(m)},
\lb{2.33a}
\end{align} 
where we used the second line in \eqref{2.31}. Hence, 
\begin{equation}
\big\langle \gamma_D v, \big(\wti\gamma_N+\Theta\gamma_D\big) u
\big\rangle_{1/2}^{(m)}=0.
\lb{2.34a}
\end{equation} 
Since $v\in H^1(\Om)^m$ is arbitrary, and the map
$\gamma_D\colon H^1(\Om)^m\to H^{1/2}(\dOm)^m$ is actually onto,
one concludes that 
\begin{equation}
\big(\wti\gamma_N+\Theta\gamma_D\big)u=0\,\text{ in }\,H^{-1/2}(\dOm)^m.
\lb{2.35a}
\end{equation} 
Thus, 
\begin{align}\lb{2.36a}
\begin{split} 
& \dom(L_{\Theta,\Om}) \subseteq \big\{v\in H^1(\Om)^m \,\big|\,
Lv\in L^2(\Om; d^nx)^m,  \\
& \hspace*{2.1cm} \big(\wti\gamma_N+\Theta\gamma_D\big)v=0
\text{ in } H^{-1/2}(\dOm)^m\big\}.
\end{split} 
\end{align} 
Next, assume that $u\in \big\{v\in H^1(\Om)^m \,\big|\,Lv\in L^2(\Om;d^n x)^m,\,
\big(\wti\gamma_N+\Theta\gamma_D\big)v=0\big\}$, $w\in H^1(\Om)^m$,
and let $w_u= Lu\in L^2(\Om; d^n x)^m$. Then, 
\begin{align}
\int_{\Om} d^n x \, {\ol w}\cdot w_u
&= \int_{\Om} d^n x \, {\ol w}\cdot Lu   \no \\
&=\int_{\Om} d^nx\,\bigl\langle\overline{Dw},A(Du)\bigr\rangle
- \langle \gamma_D w, \wti \gamma_N u \rangle_{1/2}^{(m)} \no \\
&=\int_{\Om} d^nx\,\bigl\langle\overline{Dw},A(Du)\bigr\rangle
+ \big\langle \gamma_D w,\Theta \gamma_D u \big\rangle_{1/2}^{(m)}.
\lb{2.37a}
\end{align} 
Thus, applying \eqref{2.31}, one concludes that
$u \in \dom(-\Delta_{\Theta,\Om})$ and hence 
\begin{align}\lb{2.38}
\begin{split} 
& \dom(L_{\Theta,\Om}) \supseteq \big\{v\in H^1(\Om)^m\,\big|\,
\Delta v \in L^2(\Om; d^n x)^m\,   \\
& \hspace*{2.1cm} 
\big(\wti\gamma_N+\Theta\gamma_D\big) v=0
  \text{ in } H^{-1/2}(\dOm)^m\big\}.
  \end{split} 
\end{align} 
Finally, the last claim in the statement of Theorem \ref{t3.5BBB} follows
from the fact that $H^1(\Om)^m$ embeds compactly into $\LOm^m$ (cf., e.g.,
\cite[Theorem~V.4.17]{EE89}).
\end{proof}
%%%%%%%%%%%%%%%%%%%%

%%%%%%%%%%
\begin{remark} \lb{r3.8}
We emphasize the explicit form of the domain of the operator $L_{\Theta, \Om}$ displayed 
in \eqref{3.10BBB} in terms of boundary trace operators $\gamma_D$ and $\wti\gamma_N$. 
This particular feature of 
an explicit domain rather than an operator defined via the underlying sesquilinear form and 
the First Representation Theorem is one of the reasons for our choice of Lipschitz domains 
$\Om$; it also dictates our conditions on $\Theta$ in Hypothesis \ref{h3.2}. 
\end{remark}
%%%%%%%%%%

In the special case $\Theta=0$, that is, in the case of Neumann boundary conditions,
we will also use the notation
\begin{equation} \lb{3.13}
\gQ_{N,\Om}(\dott,\dott)= \gQ_{0,\Om}(\dott,\dott), \quad L_{N,\Om}:= L_{0,\Om}.
\end{equation}

When specialized to the case $m=1$ and $L= - \Delta$, Theorem \ref{t3.5BBB} yields 
a family of self-adjoint Laplace operators $-\Delta_{\Theta,\Om}$ in $L^2(\Om;d^nx)$
indexed by the boundary operator $\Theta$, which we shall refer to
as {\it nonlocal} Robin Laplacians. More specifically, we obtain the following result 
first proved in \cite{GM09}.

%%%%%%%%%%%%%%%%%%%%
\begin{corollary} \lb{t3.5}
Assume Hypothesis \ref{h3.2} $($with $m=1$$)$, where the number $\delta>0$ is taken
to be sufficiently small as in Theorem \ref{t3.4}.
Then $-\Delta_{\Theta,\Om}$, the nonlocal Robin Laplacian in $L^2(\Om;d^nx)$ 
defined by
\begin{align}
& -\Delta_{\Theta,\Om}=-\Delta,   \no \\
& \; \dom(-\Delta_{\Theta,\Om})=
\big\{u\in H^1(\Om) \,\big|\, \Delta u\in L^2(\Om;d^nx),     \lb{3.10} \\
& \hspace*{2.8cm}\big(\wti\gamma_N+\Theta \gamma_D\big)u=0 \text{ in $H^{-1/2}(\dOm)$}\big\}, \no
\end{align}
is self-adjoint and bounded from below. Moreover,
\begin{equation}\lb{3.11}
\dom\big(|-\Delta_{\Theta,\Om}|^{1/2}\big)=H^1(\Om),
\end{equation}
and $-\Delta_{\Theta,\Om}$, has purely discrete spectrum bounded from below,
in particular,
\begin{equation}\lb{3.12a}
\sigma_{\rm ess}(-\Delta_{\Theta,\Om})=\emptyset.
\end{equation}
Finally, $-\Delta_{\Theta,\Om}$ is the operator uniquely associated with the sesquilinear 
form 
\begin{equation}\lb{2.22a}
\gq_{\Theta,\Om}(u,v):=\int_{\Om} d^nx\,\ol{(\nabla u)(x)}\cdot(\nabla v)(x)
+\big\langle\gamma_D u,\Theta\gamma_D v \big\rangle_{1/2},
\quad u,v\in H^1(\Om).
\end{equation}
\end{corollary}
%%%%%%%%%%%%%%%%%%%%

In the special case $\Theta=0$, that is, in the case of the Neumann Laplacian,
we will also use the notation
\begin{equation} \lb{3.13a}
\gq_{N,\Om}(\dott,\dott)= \gq_{0,\Om}(\dott,\dott), \quad -\Delta_{N,\Om}:=-\Delta_{0,\Om}.
\end{equation}

Next, we briefly comment on the usual case of a local Robin boundary condition, 
that is, the case where $\Theta$ is the operator of multiplication $M_{\theta}$
by a function $\theta$ defined on $\dOm$: 

%%%%%%%%%%%%%%%%%%%
\begin{lemma} [\cite{GM09}] \lb{l3.6}
Assume Hypothesis \ref{h2.1} and suppose that $\Theta=M_\theta$, the operator
of multiplication with the function $\theta\in L^p(\dOm;d^{n-1}\omega)$, where
\begin{equation}\label{Fpp}
p=n-1 \,\mbox{ if } \, n>2,  \mbox{ and } \, p\in(1,\infty] \,\mbox{ if } \, n=2. 
\end{equation}
Then
\begin{equation}\label{FGN-13}
\Theta\in\cB_{\infty}\big(H^{1/2}(\dOm),H^{-1/2}(\dOm)\big)
\end{equation}
is a self-adjoint operator which satisfies
\begin{equation}\label{FGN-14}
\|\Theta\|_{\cB(H^{1/2}(\dOm),H^{-1/2}(\dOm))}
\leq C\|\theta\|_{L^p(\dOm;d^{n-1}\omega)},
\end{equation}
for some finite constant $C=C(\Om,n,p) > 0$. In particular, the present situation 
$\Theta=M_\theta$ subordinates to the case $\Theta^{(2)}$ described in \eqref{Filo-2}.
\end{lemma}
%%%%%%%%%%%%%%%%%%%%

It is worth noting that the conditions isolated in Hypothesis \ref{h3.2} permit one to go beyond the 
assumption $\theta\in L^{\infty}(\dOm;d^{n-1}\omega)$ one frequently finds in the literature 
in connection with local Robin boundary conditions and hence naturally lead to the $L^p$-conditions 
in Lemma \ref{l3.6}.   

In the case $\Theta=M_\theta$ described in Lemma \ref{l3.6}, the underlying 
sesquilinear form and operator will be denoted by $\gQ_{\theta, \Om}$ and $L_{\theta, \Om}$. 

The $L^2$-realization of $L$ equipped with a Dirichlet boundary condition, 
$L_{D,\Om}$, in $L^2(\Om;d^n x)^m$ formally corresponds to $\Theta = \infty$ 
and so we isolate it in the next result.
 
%%%%%%%%%%%%%%%%%%%%
\begin{theorem}\lb{t3.6BBB}
Assume Hypothesis \ref{h.MM}. Then $L_{D,\Om}$, the version of $L$ 
equipped with a Dirichlet boundary condition in $L^2(\Om;d^nx)^m$, defined by 
\begin{align}
& L_{D,\Om} = L,   \no \\
& \; \dom(L_{D,\Om}) =
\big\{u\in H^1(\Om)^m\,\big|\, Lu \in L^2(\Om;d^n x)^m, \,
\gamma_D u =0 \text{ in $H^{1/2}(\dOm)^m$}\big\}    \lb{3.14BBB} \\
& \hspace*{1.9cm} = \big\{u\in H_0^1(\Om)^m\,\big|\, Lu \in L^2(\Om;d^n x)^m\big\}, \no 
\end{align} 
is self-adjoint and strictly positive. Moreover, 
\begin{equation}
\dom\big((L_{D,\Om})^{1/2}\big) = H^1_0(\Om)^m,   \lb{3.15BBB}
\end{equation} 
and since $\Om$ is open and bounded, $L_{D,\Om}$ has
purely discrete spectrum contained in $(0,\infty)$, in particular, 
\begin{equation}
\sigma_{\rm ess}(L_{D,\Om})=\emptyset.    \lb{3.15aBBB} 
\end{equation} 
Finally, $L_{D,\Om}$ is the operator uniquely associated with the sesquilinear form  
\begin{equation} \lb{3.15bBBB}
\gQ_{D,\Om} (u,v):=\int_{\Om} d^nx\,\langle\ol{(Du)(x)},A(x)(Dv)(x)\rangle,
\quad u,v\in H_0^1(\Om)^m.
\end{equation}
\end{theorem}
%%%%%%%%%%%%%%%%%%%%

This is largely proved as in the case of Theorem \ref{t3.5BBB}, and hence we omit the argument.
Here we only wish to note that \eqref{3.15aBBB} follows from \eqref{3.15BBB} since 
$H^1_0(\Om)^m$ embeds compactly into $\LOm^m$; the latter fact actually 
holds for arbitrary open, bounded sets $\Om\subset\bbR^n$ (see, e.g., \cite[Theorem V.4.18]{EE89}).

By specializing Theorem \ref{t3.6BBB} to the situation when $m=1$ and $L= -\Delta$ yields 
the following corollary (cf.\ also \cite{GLMZ05}, \cite{GMZ07} and \cite{GM09}
for related results):

%%%%%%%%%%%%%%%%%%%%
\begin{corollary} \lb{t3.6}
Assume Hypothesis \ref{h3.1}. Then $-\Delta_{D,\Om}$, the Dirichlet Laplacian 
in $L^2(\Om;d^n x)$ defined by
\begin{align}
& -\Delta_{D,\Om} = -\Delta,   \no \\
& \; \dom(-\Delta_{D,\Om}) =
\big\{u\in H^1(\Om)\,\big|\, \Delta u \in L^2(\Om;d^n x), \,
\gamma_D u =0 \text{ in $H^{1/2}(\dOm)$}\big\}    \lb{3.14} \\
& \hspace*{2.13cm} = \big\{u\in H_0^1(\Om)\,\big|\, \Delta u \in L^2(\Om;d^n x)\big\},  \no 
\end{align}
is self-adjoint and strictly positive. Moreover,
\begin{equation}
\dom\big((-\Delta_{D,\Om})^{1/2}\big) = H^1_0(\Om),   \lb{3.15}
\end{equation}
and since $\Om$ is open and bounded, $-\Delta_{D,\Om}$ has
purely discrete spectrum contained in $(0,\infty)$, in particular,
\begin{equation}
\sigma_{\rm ess}(-\Delta_{D,\Om})=\emptyset.    \lb{3.15a} 
\end{equation}
Finally, $-\Delta_{D,\Om}$ is the operator uniquely associated with the sesquilinear form 
\begin{equation} \lb{3.15b}
\gq_{D,\Om} (u,v):=\int_{\Om} d^nx\,\ol{(\nabla u)(x)}\cdot(\nabla v)(x),
\quad u,v\in H_0^1(\Om).
\end{equation}
\end{corollary}
%%%%%%%%%%%%%%%%%%%%

%%%%%%%%%%%%%%%%%%%%%%%%%%%%%%%%%%%%%%
%%%%%%%%%%%%%%%%%%%%%%%%%%%%%%%%%%%%%%
\section{Gaussian Heat Kernel Bounds for Divergence Form Elliptic 
PDOs with Robin-Type Boundary Conditions}  \lb{s4}
%%%%%%%%%%%%%%%%%%%%%%%%%%%%%%%%%%%%%%
%%%%%%%%%%%%%%%%%%%%%%%%%%%%%%%%%%%%%%

In this section we apply the abstract results of Section \ref{s2} to the concrete cases involving elliptic 
partial differential operators in divergence form in $L^2(\Omega; d^n x)$ on bounded, connected Lipschitz 
domains $\Omega$ with Robin-type boundary conditions on $\partial\Omega$ studied in 
Section \ref{s3}. 

Throughout this section we consider the scalar case $m=1$ and assume Hypothesis \ref{h3.2} 
whenever a (nonlocal) Robin boundary condition is involved. Moreover, we introduce the following assumption:

%%%%%%%%%%%
\begin{hypothesis} \lb{h4.1} Let $n\in\bbN$, $n\geq 2$. \\
$(i)$ Assume that $\Om\subset{\bbR}^n$ is a bounded Lipschitz domain $($cf.\ Appendix \ref{sA}$)$. \\
$(ii)$ Suppose that the matrix $A(\cdot)$ is Lebesgue measurable and real symmetric a.e.\ on $\Om$. 
In addition, given $0 < a_0 < a_1 < \infty$, assume that $A$ satisfies the uniform ellipticity conditions
\begin{equation}
a_0 I_n \leq A(x) \leq a_1 I_n \, \text{ for a.e.\ $x \in \Om$.}
\end{equation} 
\end{hypothesis}
%%%%%%%%%%%

Here $I_n$ represents the identity matrix in $\bbC^n$.

Given the basic setup for divergence form elliptic partial differential operators $L_{\Theta,\Om}$, 
with nonocal Robin boundary conditions developed in Section \ref{s3}, we can now turn to (Gaussian) 
heat kernel and Green's function bounds for $L_{\Theta,\Om}$ (and hence for the special case 
$-\Delta_{\Theta,\Om}$), and subsequently also for the corresponding 
Schr\"odinger-type operators. 

We will use the following heat kernel notation (for $t > 0$, a.e.\ $x, y \in \Om$)
\begin{align}
\begin{split}
&K_{\Theta,\Om} (t,x,y) =  e^{- t L_{\Theta, \Omega}}(x,y),  \quad 
K_{N,\Om} (t,x,y) =  e^{- t L_{N, \Omega}}(x,y),      \\
&K_{D,\Om} (t,x,y) =  e^{- t L_{D, \Omega}}(x,y),       
\end{split}
\end{align}
and similarly for Green's functions (for $z \in\bbC\backslash \bbR$, a.e.\ $x, y \in \Om$), 
\begin{align}
&G_{\Theta,\Om} (z,x,y) = (L_{\Theta, \Omega} - z I_{\Om})^{-1}(x,y),    \quad 
G_{N,\Om} (z,x,y) = (L_{N, \Omega} - z I_{\Om})^{-1}(x,y),    \no \\
&G_{D,\Om} (z,x,y) = (L_{D, \Omega} - z I_{\Om})^{-1}(x,y), \quad x \neq y. 
\end{align}

Next, we need one more preparatory result which identifies $L_{D,\Om}$ as the strong resolvent limit 
of $L_{\theta,\Om}$ as $\theta\uparrow\infty$. For this purpose we recall an ordering $\lessdot$ on the set of nonnegative quadratic forms in a complex, separable Hilbert space $\cH$ as follows.  If $\mathfrak{t}_1$ and $\mathfrak{t}_2$ are two non-negative quadratic forms in $\cH$, then 
\begin{equation} 
\mathfrak{t}_2\lessdot \mathfrak{t}_1 \, \text{ if and only if } \, 
\dom(\mathfrak{t}_1)\subseteq \dom(\mathfrak{t}_2) \, \text{ and } \, 
\mathfrak{t}_2(u)\leq \mathfrak{t}_1(u), \quad u\in \dom(\mathfrak{t}_1).     \lb{1f}
\end{equation}

Monotonically increasing sequences of non-negative quadratic forms have limits.  More precisely, 
one has the following fact (cf.\ \cite[Lemma\ 5.2.13]{BR02}, \cite[Ch.\,VIII, Theorem\ 3.13a]{Ka80}, 
\cite{Si77a}, \cite{Si78}):

%%%%%%%%%%%%%%%%%%
\begin{lemma}\lb{l1f}
Suppose that $\{\mathfrak{t}_n\}_{n\in \bbN}$ denotes a monotonically increasing sequence $($with respect to the ordering $\lessdot$ for quadratic forms introduced in \eqref{1f}$)$ of non-negative, closed quadratic forms in a complex, separable Hilbert space $\cH$.  Then the following items $(i)$ and $(ii)$ hold.\\
$(i)$   $\mathfrak{t}$ defined by
\begin{align}
\begin{split} 
& \mathfrak{t}(u)= \sup_{n\in \bbN}\mathfrak{t}_n(u),     \lb{2f}\\
& \, u \in \dom(\mathfrak{t}) = \bigg\{u\in \cH\, \bigg| \, u\in\bigcap_{n\in \bbN}\dom(\mathfrak{t}_n), \; \sup_{n\in \bbN}\mathfrak{t}_n(u)<\infty \bigg\},   
\end{split} 
\end{align}
is a non-negative, closed quadratic form in $\cH$, and one has
\begin{equation}\lb{3f}
\lim_{n\rightarrow \infty} \mathfrak{t}_n(u,v)=\mathfrak{t}(u,v), \quad u,v\in \dom(\mathfrak{t}),
\end{equation}
where $\mathfrak{t}_n(\cdot,\cdot)$, $n\in \bbN$ $($resp., $\mathfrak{t}(\cdot,\cdot)$$)$ is the sesquilinear form defined from the quadratic form $\mathfrak{t}_n$, $n\in \bbN$ $($resp., $\mathfrak{t}$$)$ via polarization. \\
$(ii)$  Suppose in addition that $\mathfrak{t}$ $($and therefore each $\mathfrak{t}_n$, $n\in \bbN$$)$ is densely defined.  If $T_n$, $n\in \bbN$, $($resp., $T$$)$ denotes the unique densely defined, non-negative, self-adjoint operator associated to the sesquilinear form $\mathfrak{t}_n(\cdot,\cdot)$ $($resp., $\mathfrak{t}(\cdot,\cdot)$$)$ by the First Representation Theorem, then 
\begin{equation}\lb{4f}
\slim_{n\rightarrow \infty}\, (T_n-\lambda I_{\cH})^{-1}=(T-\lambda I_{\cH})^{-1}, \quad \lambda<0.
\end{equation}
In particular,
\begin{equation}\lb{5f}
\slim_{n\rightarrow \infty}\, e^{-tT_n}=e^{-tT}, \quad t\geq 0.
\end{equation}
\end{lemma}
%%%%%%%%%%%%%%%%%%

\noindent 
We note that item $(i)$ and \eqref{4f} are taken from \cite[Lemma 5.2.13]{BR02} and 
\cite[Ch.\,VIII, Theorem 3.13a]{Ka80}. By analyticity, \eqref{4f} extends to all 
$\lambda \in \bbC\backslash \bbR$ (cf., e.g., \cite[Theorem 9.15]{We80}),
\begin{equation}\lb{6f}
\slim_{n\rightarrow \infty}\, (T_n - z I_{\cH})^{-1}=(T - z I_{\cH})^{-1}, \quad z \in \bbC\backslash [0,\infty).
\end{equation}
Since $T_n$, $n\in \bbN$, and $T$ are non-negative, choosing $z=i \, (= \sqrt{-1})$ in \eqref{6f} and applying 
\cite[Theorem 9.18]{We80} yields \eqref{5f}.

The following result is mentioned in \cite{KR80} in the context of Dirichlet Laplacians and we thank 
Derek Robinson for helpful discussions concerning its proof:
 
%%%%%%%%%%%%%%%%%%
\begin{lemma}\lb{l2f}
Assume Hypothesis \ref{h3.2} $($with $m=1$$)$ and denote by $L_{\vartheta,\Om}$ the operator 
$L_{\Theta,\Om}$ in the special case where $\Theta$ denotes the operator of multiplication, 
$M_{\vartheta}$, in $L^2(\partial\Om; d^{n-1} \omega)$ with the positive constant $\vartheta >0$. Then 
\begin{equation}\lb{7f}
\slim_{\vartheta \uparrow \infty} e^{- t L_{\vartheta, \Omega}} = e^{- t L_{D,\Omega}}, \quad t\geq 0.
\end{equation}
\end{lemma}
%%%%%%%%%%%%%%%%%%
\begin{proof}
Let $\{\vartheta_n\}_{n=1}^{\infty}$ denote an arbitrary sequence of positive real numbers satisfying
\begin{equation}\lb{8f}
0<\vartheta_n<\vartheta_{n+1}, \; n\in\bbN, \, \text{ and } \, 
\lim_{n\rightarrow \infty}\vartheta_n = \infty.
\end{equation}
The inequalities in \eqref{8f} imply the sequence $\{\gQ_{\vartheta_n,\Omega}(\cdot,\cdot)\}_{n=1}^{\infty}$ 
of non-negative, closed quadratic forms (cf.\ \eqref{3.8} with $\Theta$ replaced by $M_{\vartheta_n}$) is monotonically increasing with respect to the quadratic form ordering $\lessdot$ introduced above in \eqref{1f}.  In this simple case, the quadratic forms have a common 
domain (viz., $H^1(\Omega)$); thus, monotonicity simply amounts to
\begin{equation}\lb{9f}
\gQ_{\vartheta_n,\Omega}(u,u)\leq \gQ_{\vartheta_{n+1},\Omega}(u,u), \quad 
u\in H^1(\Omega),\; n\in \bbN.
\end{equation}
By Lemma \ref{l1f}, 
\begin{align}
&\gQ_{\infty,\Omega}(u,u)= \sup_{n\in \bbN}\gQ_{\vartheta_n,\Omega}(u,u), \quad 
u\in \dom(\gQ_{\infty,\Omega}(\cdot,\cdot)),\no\\
&\dom(\gQ_{\infty,\Omega}(\cdot,\cdot))\lb{10f}\\
&\quad =\bigg\{u\in L^2(\Omega;d^nx)\, \bigg| \, 
u\in\bigcap_{n\in \bbN}\dom(\gQ_{\vartheta_n,\Omega}(\cdot,\cdot)), \; 
\sup_{n\in \bbN}\gQ_{\vartheta_n,\Omega}(u,u)<\infty \bigg\},\no
\end{align}
defines a non-negative, closed quadratic form in $L^2(\Omega; d^n x)$.  We claim the limiting quadratic form $\gQ_{\infty,\Omega}(\cdot,\cdot)$ defined by \eqref{10f} coincides with $\gQ_{D,\Omega}(\cdot,\cdot)$, the quadratic form of the Dirichlet Laplacian $L_{D,\Omega}$ in $L^2(\Omega;d^n x)$.  In order to prove this, one needs to verify the following two conditions,
\begin{align}
& \dom(\gQ_{\infty,\Omega}(\cdot,\cdot)) = H_0^1(\Omega),\lb{12f}\\
& \, \gQ_{\infty,\Omega}(u,u) = \gQ_{D,\Omega}(u,u), \quad u\in H_0^1(\Omega).\lb{13f}
\end{align}
To verify \eqref{12f} one observes that $u\in \dom(\gQ_{\infty,\Omega}(\cdot,\cdot))$ if and only if 
$u\in H^1(\Omega)$ and 
\begin{equation}\lb{14f}
\sup_{n\in \bbN}\bigg[\int_{\Omega} d^nx\, \langle(\nabla u)(x), A(x) (\nabla u)(x)\rangle 
+\vartheta_n\big\langle\gamma_D u,\gamma_D u \big\rangle_{1/2}\bigg]<\infty.
\end{equation}
Clearly, \eqref{14f} is fulfilled if and only if
\begin{equation}\lb{15f}
\big\langle\gamma_D u,\gamma_D u \big\rangle_{1/2}=0,
\end{equation}
that is, if and only if $\gamma_Du=0$ in $H^{1/2}(\Omega)$.  Therefore, 
$u\in \dom(\gQ_{\infty,\Omega}(\cdot,\cdot))$ if and only if $u\in H^1(\Omega)$ and 
$\gamma_Du=0$ in $H^{1/2}(\Omega)$, and \eqref{12f} follows.  Finally, one computes
\begin{align}\lb{16f}
& \gQ_{\infty,\Omega}(u,u)=\sup_{n\in \bbN}\gQ_{\vartheta_n,\Omega}(u,u)
= \int_{\Omega} d^nx\, \langle(\nabla u)(x), A(x) (\nabla u)(x)\rangle 
= \gQ_{D,\Omega}(u,u),    \no \\
& \hspace*{9.2cm} u\in H_0^1(\Omega),
\end{align}
implying \eqref{13f}.  By Lemma \ref{l1f}, and \eqref{5f} in particular, with 
$T_n=L_{\vartheta_n,\Omega}$, $n\in \bbN$, and $T= L_{D,\Omega}$,
\begin{equation}\lb{17f}
\slim_{n\rightarrow \infty} e^{- t L_{\vartheta_n, \Omega}} 
= e^{- t L_{D,\Omega}}, \quad t\geq 0.
\end{equation}
Since $\{\vartheta_n\}_{n=1}^{\infty}$ satisfying \eqref{8f} was arbitrary, \eqref{7f} follows.
\end{proof}
%%%%%%%%%%%%%%%%%%

Now we are ready to formulate the first principal result of this section. 

%%%%%%%%%%%%%%%%%%%%
\begin{theorem} \lb{t4.3}
Assume Hypothesis \ref{h4.1}, suppose that $\Theta_j$, $j=1,2$, satisfy the assumptions 
introduced in Hypothesis \ref{h3.2}, and denote by $L_{\Theta_j,\Om}$ the operators in 
\eqref{3.10BBB} uniquely associated with the sesquilinear forms  $\gQ_{\Theta_j,\Om} (\dott,\dott)$, 
$j=1,2$, defined on $H^1(\Om)\times H^1(\Om)$ according to \eqref{3.8}. Suppose, in addition, that 
\begin{equation}
\big\langle\gamma_D |u|,\Theta_j \gamma_D |u| \big\rangle_{1/2} \leq 
\big\langle\gamma_D u,\Theta_j \gamma_D u \big\rangle_{1/2},
\quad u \in H^1(\Om), \; j=1,2.    \lb{3.18} 
\end{equation}
Then, assuming $0 \leq \Theta_1 \leq \Theta_2$, one has the positivity preserving relations 
\begin{equation}
0 \preccurlyeq e^{- t L_{D,\Omega}}   
\preccurlyeq e^{- t L_{\Theta_2, \Omega}} 
\preccurlyeq e^{- t L_{\Theta_1, \Omega}}
\preccurlyeq e^{- t L_{N,\Omega}}, 
\quad t \geq 0,    \lb{3.19} 
\end{equation} 
or equivalently, 
\begin{align}
\begin{split} 
0 & \preccurlyeq (L_{D,\Omega} + \lambda I_{\Omega})^{-1}  
\preccurlyeq (L_{\Theta_2, \Omega} + \lambda I_{\Omega})^{-1} 
\preccurlyeq (L_{\Theta_1, \Omega} + \lambda I_{\Omega})^{-1}     \\ 
& \preccurlyeq (L_{N,\Omega} + \lambda I_{\Omega})^{-1}, \quad \lambda > 0.     \lb{3.20}
\end{split} 
\end{align}  
In addition, all semigroups appearing in \eqref{3.19} lie in the trace class, 
\begin{equation}
e^{- t L_{D,\Omega}}, \,  e^{-t L_{\Theta_j, \Omega}},  
\, e^{- t L_{N,\Omega}} \in \cB_1\big(L^2(\Om; d^n x)\big), \quad j=1,2, \; t>0. 
\end{equation}
In particular, one has the Gaussian heat kernel bounds $($for $t > 0$, a.e.\ $x, y \in \Om$$)$,
\begin{align}
\begin{split} 
0 &\leq K_{D,\Om} (t,x,y) \leq K_{\Theta_2,\Om} (t,x,y) \leq K_{\Theta_1,\Om} (t,x,y) \leq 
K_{N,\Om} (t,x,y)     \lb{3.21} \\ 
& \leq C_{\alpha, a_0, \Om} \max \big(t^{-n/2}, 1\big) \exp\big\{- |x-y|^2/ [4 (1+\gamma) a_1 t]\big\},  
\quad \gamma \in (0,1).   
\end{split} 
\end{align}
and Green's function bounds $($for $\lambda > 0$, a.e.\ $x, y \in \Om$$)$,
\begin{align}
\begin{split} 
0 &\leq G_{D,\Om} (\lambda,x,y) \leq G_{\Theta_2,\Om} (t,x,y) \leq G_{\Theta_1,\Om} (\lambda,x,y) 
\leq G_{N,\Om} (\lambda,x,y)     \lb{3.22} \\ 
& \leq \begin{cases} C_{a_0,a_1,\lambda,\Omega,n} |x - y|^{2-n}, & n \geq 3, \\
C_{a_0,a_1,\lambda,\Omega} \big|\ln\big(1 + |x - y|^{-1}\big)\big|, & n=2,  \end{cases} 
\quad x \neq y.    
\end{split} 
\end{align}
\end{theorem}
%%%%%%%%%%%%%%%%%%%%
\begin{proof}
By \cite[Theorems 1.3.5, 13.9]{Da89}, $L_{D,\Omega}$ and $L_{N,\Omega}$ satisfy the 
Beurling--Deny condition $(iii)$ of Theorem \ref{t2.7}, and hence by Theorem \ref{t2.7}\,$(i)$, 
$e^{- t L_{D, \Omega}}$ and $e^{- t L_{N, \Omega}}$, $t\geq 0$, are positivity preserving. In 
addition, both have nonnegative integral kernels (cf.\ also Theorem \ref{t2.3}) and are known to satisfy 
the bounds \eqref{3.16A} and \eqref{3.16a}. 

As the relations \eqref{3.19} and \eqref{3.20} are equivalent by Theorem \ref{t2.8}, it suffices to focus on \eqref{3.19}. But then, 
$0 \preccurlyeq e^{- t L_{\Theta_2, \Omega}} 
\preccurlyeq e^{- t L_{\Theta_1, \Omega}} \preccurlyeq 
e^{- t L_{N,\Omega}}$, $t \geq 0$, is 
immediate upon combining Theorem \ref{t2.8}\,$(iv)$, \eqref{3.8}, and \eqref{3.18}. Moreover, using 
Lemma \ref{l2f}, \eqref{3.19} follows. Thus, an application of \eqref{2.13} implies the inequalities 
\eqref{3.21} and \eqref{3.22}, except, the very last in either one of them. The inequality \eqref{3.16a} 
for $K_{N,\Om} (t,x,y)$ then completes the proof of \eqref{3.21}. Similarly, the 
inequality $G_{N,\Om} (\lambda,x,y) \leq C_{a_0,a_1,\lambda, \Om,n} |x-y|^{2-n}$, 
$\lambda > 0$, $x,y \in \Om$, $x \neq y$, can be found, for instance, in \cite[Lemma\ 3.2]{DM96} 
(see also \cite{CK11}) for $n\geq 3$. As we were not able to find the case $n=2$ in the literature, we 
provide a short proof in Appendix \ref{sC}. 

The inequalities \eqref{3.21} then follow from Theorem \ref{t2.3}, \eqref{3.19}, and \eqref{3.20}, 
as soon as we establish that semigroups of the type 
$e^{- t L_{\Theta, \Omega}}$ with $\Theta$ satisfying the conditions of $\Theta_j$ are, in fact, 
integral operators. This follows from combining Lemma \ref{l2.7}\,$(iii)$ and \eqref{3.19} as the 
Neumann heat kernel bound \eqref{3.16a} yields the trace class property 
$e^{- t L_{N, \Omega}} \in \cB_1\big(L^2(\Om; d^n x)\big)$, $t > 0$ (this trace class 
property  of course also applies to $e^{- t L_{D, \Omega}}$, $t > 0$). 

Finally, using \eqref{2.25} and \eqref{3.21}, the results in \cite[Sects.\ 2, 4]{AB94} show that 
also $(L_{\Theta, \Omega} + \lambda I_{\Omega})^{-1}$ are integral operators whose 
integral kernels satisfy \eqref{3.22}.  
\end{proof}
%%%%%%%%%%%%%%%%%%%%

Positivity improving and nondegeneracy of the ground state are considered next:

%%%%%%%%%%%%%
\begin{corollary} \lb{c4.4}
Under the hypotheses of Theorem \ref{t4.3}, and under the assumptions that $\Om$ is connected 
and $L_{\Theta_1,\Om} \neq L_{\Theta_2,\Om}$, the positivity preserving relations in \eqref{3.19} 
and \eqref{3.20} actually extend to positivity improving relations, that is, $\preccurlyeq$ in \eqref{3.19} 
and \eqref{3.20} can be replaced by $\prec$. In addition, the infimum of the spectrum 
of $L_{D,\Om}$, $L_{\Theta_j,\Om}$, or $L_{N,\Om}$ is a simple  eigenvalue and the associated eigenfunction can be chosen to be strictly positive a.e.\ in $\Omega$. 
\end{corollary}
%%%%%%%%%%%%%
\begin{proof}
Positivity improving rather than just positivity preserving of the semigroups (resp., resolvents) 
in \eqref{3.19} (resp., \eqref{3.20}) follows from Corollary \ref{c2.10} and the fact that 
$L_{D,\Om}^{-1}$ is positivity improving if $\Om$ is connected. The latter fact is implied, for 
instance, by the explicit lower bound for the Green's function 
$G_{D, \Om} (\cdot,\cdot) = L_{D,\Om}^{-1} (\cdot,\cdot)$ established in \cite{Da84}, 
\cite{Da87}, \cite{vdB90} for $n\geq 3$. Alternatively, one can invoke the $3 G$ theorem 
(proved, e.g., in \cite[Theorems\ 6.5, 6.15]{CZ95}, see also \cite{AL05}, \cite{CFZ88}) to 
obtain strict positivity of $G_{D,\Om}(\cdot,\cdot)$ on $\Om \times \Om$. In this context we also 
note that indecomposability of $e^{- t L_{D,\Om}}$ also follows from \cite[Theorem\ 3.3.5]{Da89} 
and positivity improving of $e^{- t L_{D,\Om}}$ and $e^{- t L_{N,\Om}}$ is proved in \cite{Ou04} and 
\cite[Theorem\ 4.5]{Ou05}. 

Nondegeneracy of the groundstate of $L_{D,\Om}$, $L_{\Theta_j,\Om}$, or $L_{N,\Om}$, and 
an associated strictly positive eigenfunction a.e.\ in $\Om$ then follows from 
\cite[Theorem\ X.III.44]{RS78}. In this connect one recalls that the spectra of 
 $L_{D,\Om}$, $L_{\Theta_j,\Om}$, and $L_{N,\Om}$ are purely discrete (cf.\ \eqref{3.12a} and \eqref{3.15aBBB}).
\end{proof}
%%%%%%%%%%%%%

For background literature on nondegenerate groundstates we refer, for instance, to 
\cite{Da73}, \cite[Ch.\ 7]{Da80}, \cite[Ch.\ 13]{Da07}, \cite[Sect.\ 10]{Fa75}, \cite{FS75}, \cite{Ge84}, 
\cite[Sect.\ 3.3]{GJ81}, \cite{Go77}, \cite[Sect.\ XIII.12]{RS78}, \cite[Sect.\ 10.5]{We80}.  

%%%%%%%%%%%%%
\begin{corollary} \lb{c4.6}
Under the hypotheses of Theorem \ref{t4.3}, all semigroups in 
\eqref{3.19} are sub-Markovian and hence extend to contraction semigroups on $L^\infty(\Om; d^nx)$. Moreover, all semigroups in 
\eqref{3.19} extend to strongly continuous semigroups on 
$L^p(\Om; d^nx)$, $p \in [1,\infty)$.
\end{corollary}
%%%%%%%%%%%%%
\begin{proof}
This is an immediate consequence of the interpolation and duality considerations discussed in 
\cite[p.\ 56--57]{Ou05}, upon noticing the following facts: $1)$ The generators of all semigroups in 
\eqref{3.19} are self-adjoint and nonnegative, and hence 
$L^2(\Om; d^nx)$-contractions. $2)$ Since $e^{- t L_{N,\Om}}$ extends to an 
$L^\infty(\Om; d^nx)$-contraction by \cite[Corollary\ 4.10]{Ou05}, the fact 
\eqref{2.80} then implies the $L^\infty(\Om; d^nx)$-contractivity of all remaining semigroups in 
\eqref{3.19}. 
\end{proof}
%%%%%%%%%%%%%

%%%%%%%%%%%%%
\begin{remark} \lb{r4.5}
One observes that condition \eqref{3.18} is automatically satisfied in the special case of local Robin 
boundary conditions considered in Lemma \ref{l3.6}. In this context of local Robin boundary 
conditions a fair number of references proving Gaussian heat kernel bounds for $L_{\theta, \Om}$ on 
bounded Lipschitz domains have been established in the literature as detailed in the paragraph 
preceding \eqref{1.1} in the introduction of this paper. The nonlocal Robin boundary conditions in 
terms of $\Theta$ as encoded in \eqref{3.8} originated in 
\cite{GM09} and to the best of our knowledge, the corresponding heat kernel and Green's function 
estimates for $L_{\Theta,\Om}$ in Theorem \ref{t4.3} are new.  
\end{remark}
%%%%%%%%%%%%%

%%%%%%%%%%%%%
\begin{remark} \lb{r3.9}
One can add a nonnegative potential $ 0 \leq V \in L^1_{\loc}(\Om; d^nx)$ to all operators in 
Theorem \ref{t4.3} by employing the following standard procedure: First, adding the sesquilinear 
form $\gQ_{V,\Om}$ defined by
\begin{equation}
\gQ_{V,\Om} (u,v) = \int_{\Om} d^n x \, V(x) \ol{u(x)} v(x), \quad u,v \in \dom\big(V^{1/2}\big),
\end{equation} 
to $\gQ_{\Theta, \Om}$ and $\gQ_{D, \Om}$ with domains $H^1(\Om) \cap \dom\big(V^{1/2}\big)$ and 
$H_0^1(\Om) \cap \dom\big(V^{1/2}\big)$, respectively, yields densely defined and closed forms 
bounded from below. The uniquely associated nonnegative self-adjoint operators associated with 
\begin{equation} 
\gQ_{\Theta, \Om}(u,v) + \gQ_{V,\Om}(u,v), \quad u, v \in H_0^1(\Om) \cap \dom\big(V^{1/2}\big), 
\end{equation}
and 
\begin{equation}
\gQ_{D, \Om}(u,v) + \gQ_{V,\Om}(u,v), \quad u, v \in H^1(\Om) \cap \dom\big(V^{1/2}\big), 
\end{equation}
in obvious notation, will be denoted by $H_{\Theta, \Om}$ (and $H_{N, \Om}$ if $\Theta = 0$) 
and $H_{D, \Om}$. One notes that $C^\infty(\ol \Om)$ (the restrictions of 
$C^\infty(\bbR^n)$-functions to $\Om$) and $C_0^\infty(\Om)$ are form cores for $H_{\Theta, \Om}$ 
and $H_{D, \Om}$, respectively. Temporarily replacing $V$ by the bounded approximants  
$V_{\varepsilon} = [1 - \exp(- \varepsilon V)]/\varepsilon$, $\varepsilon >0$ yields positivity preserving 
semigroups $e^{- t (L_{D,\Omega} + V_{\varepsilon})}$,    
$e^{- t (L_{\Theta, \Omega} + V_{\varepsilon})}$, 
$e^{- t (L_{N,\Omega} + V_{\varepsilon})}$, $t\geq 0$, employing the Trotter--Kato formula,  
\begin{equation} 
e^{-t(A+B)} = \slim_{m\to\infty} \big[e^{- t (A/m)} e^{-t (B/m)}\big]^m, 
\quad t \geq 0,
\end{equation} 
with $0 \leq A = A^*$ and 
$B=B^* \in \cB(\cH)$. Using the monotone convergence for forms (cf.\ \cite[Lemma\ 4]{BKR80} and 
the corresponding strong convergence of the semigroups 
$e^{- t (L_{D,\Omega} + V_{\varepsilon})}$,    
$e^{- t (L_{\Theta, \Omega} + V_{\varepsilon})}$, 
$e^{- t (L_{N,\Omega} + V_{\varepsilon})}$ to $e^{-t H_{D,\Omega}}$,    
$e^{- t H_{\Theta, \Omega}}$, $e^{- t H_{N,\Omega}}$, respectively, as 
$\varepsilon \downarrow 0$ (cf.\, e.g., \cite[Lemma\ 5.2.13]{BR02}), yields positivity preserving of 
the latter for all $t\geq 0$. (Alternatively, one can apply the truncation procedure for $V$ described, 
e.g., in \cite[Sects.\ 8, 9]{Fa75}.) An application of Lemma \ref{l2.7}\,$(iii)$ then again yields the trace 
class property of all semigroups involved. Thus, the analogs of \eqref{3.19}--\eqref{3.20} hold for 
$H_{D, \Om}$, $H_{\Theta, \Om}$, and $H_{N, \Om}$. Moreover (cf.\ \cite[Lemma\ 1.1]{Da73}, 
\cite[Sect.\ 5.B]{BKR80}), the Trotter--Kato formula applied to form perturbations (either in the form 
established in \cite{Ka78} or using again approximations via $V_{\varepsilon}$) also yields
\begin{equation}
e^{- t H_{N,\Omega}} \preccurlyeq e^{- t L_{N,\Om}}, \quad t \geq 0,
\end{equation} 
and hence by Theorem \ref{t2.3}, the analogs of the integral kernel estimates \eqref{3.21} and 
\eqref{3.22} hold. 

One can also add a nonpositive potential $0 \geq W \in L^1_{\loc} (\Om; d^n x)$ either by applying the perturbation method treated in \cite[Theorem\ XIII.45]{RS78} or by following the small form perturbation 
approach discussed in \cite[Sect.\ 10]{Fa75}. This will, in general, result in an additional multiplicative 
factor of the type $e^{c t}$ on the right-hand side of \eqref{3.21}. 

In this context of additive perturbations we also refer to \cite{AD06} and \cite{De08}, where Gaussian heat kernel bounds are obtained via H\"older inequalities and domination of semigroups. 
\end{remark}
%%%%%%%%%%%%%

%%%%%%%%%%%%%
\begin{remark} \lb{r4.9}
For a variety of interesting implications of the $L^2$-theory of Gaussian heat kernel bounds to $L^p$-spectral theory and analyticity of semigroups we refer to \cite[Ch.\ 7]{Ou05} and the extensive literature cited in the notes to this book chapter. Here we only mention the following facts (cf.\ also \cite{Ba98}, \cite{LP95}, 
\cite{LSV02}, \cite{Ou95}, \cite[p.\ 95--97]{Ou05}, \cite{Ou06}, \cite{SV02}): Introducing the sector 
$S(\phi) = \{z \in \bbC\backslash\{0\} \,|\, |\arg(z)| \leq \phi\}$, 
$\phi \in (0,\pi/2]$, all semigroups in \eqref{3.19} extend to holomorphic semigroups in $L^p(\Om; d^nx)$, $p \in (1,\infty)$, on the sector 
$S\big((\pi/2) - \arctan\big(|p-2|\big/ \big[2(p-1)^{1/2}\big]\big)\big)$. In addition, 
\begin{align}
& \big\| e^{- z L_{p,\Om}}\big\|_{\cB(L^p(\Om; d^nx))} \leq 1, \quad 
z \in S\big((\pi/2) - \arctan\big(|p-2|\big/ \big[2(p-1)^{1/2}\big]\big)\big),  \no \\
& \hspace*{9.1cm} p \in (1,\infty),  
\end{align}   
and
\begin{equation}
\sigma(L_{p.\Om}) \subseteq \big\{z\in\bbC \,\big| \, |\arg(z)| \leq 
\arctan \big(|p-2|\big/ \big[2(p-1)^{1/2}\big]\big)\big)\big\} \cup \{0\}, 
\quad p \in (1,\infty).  
\end{equation}
Here $L_{p,\Om}$ denotes any of the generators of (the extension of) the semigroups in \eqref{3.19} on $L^p(\Om; d^nx)$, $p \in (1,\infty)$.  
\end{remark}
%%%%%%%%%%%%%

Finally, we mention a canonical counter example concerning positivity preserving semigroups among  the set of nonnegative self-adjoint extensions of a strictly positive minimal operator:

%%%%%%%%%%%%%
\begin{remark} \lb{r4.10}
The literature on positivity preserving semigroups associated with the Friedrichs extension of differential operators is rather extensive as, typically, the Friedrichs extension corresponds to the case of 
Dirichlet boundary conditions. In the case where the associated minimal operator is nonnegative, it  
also possesses a second nonnegative distinguished self-adjoint extension, the Krein--von Neumann 
extension. The latter, in the context of Laplacians, was recently discussed in detail in \cite{AGMT10} and 
\cite{GM11} (see also \cite{AGMST10} and the numerous references to the literature in these three 
sources) in the special case where the underlying minimal operator is in fact bounded from below by 
$\varepsilon > 0$. Then, admittedly, on a formal level, the Krein--von Neumann extension in this special case appears to be defined in terms of a boundary condition that on the surface resembles a nonlocal 
Robin boundary condition, where $\Theta$ is expressed in terms of the operator-valued 
Dirichlet-to-Neumann map (cf.\ \cite[Subsect.\ 5.2]{AGMT10}, \cite[Sect.\ 13]{GM11}). However, the 
well-known degeneracy of the lowest point in the spectrum, namely zero, of the Krein--von Neumann extension corresponding to the Laplacian associated with a bounded, connected domain $\Om \subset \bbR^n$, 
$n \in\bbN$, and appropriate regularity of 
$\partial \Om$, stemming from the nontrivial nullspace of the adjoint of the underlying minimal operator,  guarantees that its semigroup is, in fact, never positivity preserving (let alone, positivity improving). For additional results in this direction, and the fact that Krein semigroups associated with $L$ and bounded domains $\Om$ are not Markovian, we refer to \cite[Sect.\ 3.3]{FOT11}. The boundary conditions 
leading to a positivity preserving semigroup in the case of the (generalized) one-dimensional Laplacian 
on a bounded interval were classified in \cite{Fe57} (see also \cite[p.\ 147]{FOT11}).  

On an abstract level, and by appealing once more to \cite[Theorem\ X.III.44]{RS78}, the 
Krein--von Neumann extension $A_K$ of a densely defined and strictly positive operator $A$ with 
deficiency indices $d \geq 2$ in some fixed complex, separable Hilbert space $\cH$ will always have 
the $d$-dimensional nullspace of $A^*$ as its nullspace (cf., e.g., \cite[Sect.\ 2]{AGMT10}, 
\cite[Sect.\ 2]{AGMST10}) and hence can never be positivity improving as long as the resolvent (or semigroup) of $A_K$ is known to be irreducible.
\end{remark}
%%%%%%%%%%%%%

%%%%%%%%%%%%%%%%%%%%%%%%%%%%%%%%%%%%%%
%%%%%%%%%%%%%%% appendices %%%%%%%%%%%%%%%%
\appendix
%%%%%%%%%%%% Appendix A %%%%%%%%%%%%%%
\section{Sobolev Spaces on Lipschitz Domains in a Nutshell} \lb{sA}
\renewcommand{\theequation}{A.\arabic{equation}}
\renewcommand{\thetheorem}{A.\arabic{theorem}}
\setcounter{theorem}{0} \setcounter{equation}{0}
%%%%%%%%%%%%%%%%%%%%%%%%%%%%%%%%%%%%%%
%%%%%%%%%%%%%%%%%%%%%%%%%%%%%%%%%%%%%%

Following \cite{GM09}, we recall some basic facts in connection with Sobolev spaces 
on Lipschitz domains and on their boundaries. 

We start by briefly considering open subsets $\Om\subset\bbR^n$, $n\in\bbN$. For an arbitrary 
$m\in\bbN\cup\{0\}$, we follow the customary
way of defining $L^2$-Sobolev spaces of order $\pm m$ in $\Om$ as
\begin{align}\label{hGi-1}
H^m(\Om) &:= \big\{u\in L^2(\Om;d^nx)\,\big|\,\partial^\alpha u\in L^2(\Om;d^nx)
\mbox{ for } 0\leq|\alpha|\leq m\big\}, \\
H^{-m}(\Om) &:=\bigg\{u\in\cD^{\prime}(\Om)\,\bigg|\,u=\sum_{|\alpha|\leq m}
\partial^\alpha u_{\alpha}, \mbox{ with }u_\alpha\in L^2(\Om;d^nx), \, 
0\leq|\alpha|\leq m\bigg\},
\label{hGi-2}
\end{align}
equipped with natural norms (cf., e.g., \cite[Ch.\ 3]{AF03}, \cite[Ch.\ 1]{Ma85}). Here 
$\cD^\prime(\Om)$ denotes the usual set of distributions on $\Omega\subseteq \bbR^n$. Then one sets
\begin{equation}\label{hGi-3}
H^m_0(\Om):=\,\mbox{the closure of $C^\infty_0(\Om)$ in $H^m(\Om)$}, 
\quad m\in\bbN \cup \{0\}.
\end{equation}
As is well-known, all three spaces above are Banach, reflexive and, in addition,
\begin{equation}\label{hGi-4}
\big(H^m_0(\Om)\big)^*=H^{-m}(\Om).
\end{equation}
Again, see, for instance, \cite[Ch.\ 3]{AF03}, \cite[Sect.\ 1.1.15]{Ma85}. Throughout this paper, 
we agree to use the {\it adjoint} (rather than the dual) space $X^*$ of
a Banach space $X$.

One recalls that an open, nonempty, bounded set $\Omega\subset\bbR^n$ 
is called a {\it bounded Lipschitz domain} if the following property holds: 
There exists an open covering $\{{\mathcal O}_j\}_{1\leq j\leq N}$
of the boundary $\partial\Omega$ of $\Om$ such that for every
$j\in\{1,\dots,N\}$, ${\mathcal O}_j\cap\Omega$ coincides with the portion
of ${\mathcal O}_j$ lying in the over-graph of a Lipschitz function
$\varphi_j:\bbR^{n-1}\to\bbR$ (considered in a new system of coordinates
obtained from the original one via a rigid motion). The number
$\max\,\{\|\nabla\varphi_j\|_{L^\infty(\bbR^{n-1};d^{n-1}x')} \,|\,1\leq j\leq N\}$
is said to represent the {\it Lipschitz character} of $\Omega$.

The classical theorem of Rademacher of almost everywhere differentiability 
of Lipschitz functions ensures that, for any  Lipschitz domain $\Omega$, the
surface measure $d^{n-1} \omega$ is well-defined on  $\partial\Omega$ and
that there exists an outward  pointing normal unit vector $\nu$ at
almost every point of $\partial\Omega$.

In the remainder of this appendix we shall assume that Hypothesis \ref{h3.1} holds, that is, 
we suppose that $\Om\subset{\bbR}^n$, $n\in\bbN$, $n\geq 2$, is a bounded Lipschitz domain.

As regards $L^2$-based Sobolev spaces of fractional order $s\in\bbR$,
in a bounded {\it Lipschitz domain} $\Om\subset\bbR^n$ we set
\begin{align}\label{HH-h1}
H^{s}(\bbR^n) &:=\bigg\{U\in \cS^\prime(\bbR^n)\,\bigg|\,
\norm{U}_{H^{s}(\bbR^n)}^2 = \int_{\bbR^n}d^n\xi\,
\big|\hatt U(\xi)\big|^2\big(1+\abs{\xi}^{2s}\big)<\infty \bigg\},
\\
H^{s}(\Om) &:=\big\{u\in \cD^\prime(\Om)\,\big|\,u=U|_\Om\text{ for some }
U\in H^{s}(\bbR^n)\big\}. 
\label{HH-h2}
\end{align}
Here $\cS^\prime(\bbR^n)$ is the space of tempered distributions on $\bbR^n$,
and $\hatt U$ denotes the Fourier transform of $U\in\cS^\prime(\bbR^n)$.
These definitions are consistent with \eqref{hGi-1}, \eqref{hGi-2}.
Moreover, so is
\begin{equation}\label{incl-xxx}
H^{s}_0(\Omega):= \big\{u\in H^{s}(\bbR^n)\,\big|\, \supp(u)\subseteq\ol{\Omega}\big\},
\quad s\in\bbR,
\end{equation}
equipped with the natural norm induced by $H^{s}(\bbR^n)$, in relation to
\eqref{hGi-3}. One also has 
\begin{equation}\label{incl-Ya}
\big(H^{s}_0(\Omega)\big)^*=H^{-s}(\Omega),\quad s\in\bbR 
\end{equation}
(cf., e.g., \cite{JK95}). For a bounded Lipschitz domain $\Omega\subset\bbR^n$ 
it is known that
\begin{equation}\label{dual-xxx}
\bigl(H^{s}(\Omega)\bigr)^*=H^{-s}(\Omega), \quad - 1/2 <s< 1/2.
\end{equation}
See \cite{Tr02} for this and other related properties. 

To discuss Sobolev spaces on the boundary of a Lipschitz domain, consider
first the case when $\Omega\subset\bbR^n$ is the domain lying above the graph
of a Lipschitz function $\varphi\colon\bbR^{n-1}\to\bbR$. In this setting,
we define the Sobolev space $H^s(\partial\Omega)$ for $0\leq s\leq 1$,
as the space of functions $f\in L^2(\partial\Omega;d^{n-1}\omega)$ with the
property that $f(x',\varphi(x'))$, as a function of $x'\in\bbR^{n-1}$,
belongs to $H^s(\bbR^{n-1})$. This definition is easily adapted to the case
when $\Omega$ is a Lipschitz domain whose boundary is compact,
by using a smooth partition of unity. Finally, for $-1\leq s\leq 0$, we set
\begin{equation}\label{A.6}
H^s(\dOm) = \big(H^{-s}(\dOm)\big)^*, \quad -1 \le s \le 0.
\end{equation}
 From the above characterization of $H^s(\partial\Omega)$ it follows that
any property of Sobolev spaces (of order $s\in[-1,1]$) defined in Euclidean
domains, which are invariant under multiplication by smooth, compactly
supported functions as well as compositions by bi-Lipschitz diffeomorphisms,
readily extends to the setting of $H^s(\partial\Omega)$ (via localization and
pull-back). As a concrete example, for each Lipschitz domain $\Omega$ 
with compact boundary, one has  
\begin{equation} \label{EQ1}
H^s(\partial\Omega)\hookrightarrow L^2(\partial\Omega;d^{n-1} \omega)
\, \text{ compactly if }\,0<s\leq 1.  
\end{equation}
For additional background 
information in this context we refer, for instance, to \cite{Au04}, 
\cite{Au06}, \cite[Chs.\ V, VI]{EE89}, \cite[Ch.\ 1]{Gr85}, 
\cite[Ch.\ 3]{Mc00}, \cite[Sect.\ I.4.2]{Wl87}.

Assuming Hypothesis \ref{h3.1}, we introduce the boundary trace
operator $\ga_D^0$ (the Dirichlet trace) by
\begin{equation}
\ga_D^0\colon C(\ol{\Om})\to C(\dOm), \quad \ga_D^0 u = u|_\dOm.   \label{2.5A}
\end{equation}
Then there exists a bounded linear operator $\gamma_D$
\begin{align}
\begin{split}
& \ga_D\colon H^{s}(\Om)\to H^{s-(1/2)}(\dOm) \hookrightarrow \LdOm,
\quad 1/2<s<3/2, \label{2.6A}  \\
& \ga_D\colon H^{3/2}(\Om)\to H^{1-\varepsilon}(\dOm) \hookrightarrow \LdOm,
\quad \varepsilon \in (0,1) 
\end{split}
\end{align}
(cf., e.g., \cite[Theorem 3.38]{Mc00}), whose action is compatible with that of 
$\ga_D^0$. That is, the two Dirichlet trace  operators coincide on the intersection 
of their domains. Moreover, we recall that
\begin{equation}\label{2.6B}
\ga_D\colon H^{s}(\Om)\to H^{s-(1/2)}(\dOm) \, \text{ is onto for $1/2<s<3/2$}.
\end{equation}

Next, retain Hypothesis \ref{h3.1} and assume that 
\begin{equation}\label{MAR-B.1}
A\in \cM\bigl(H^s(\Omega)\bigr), \quad 1/2<s<3/2.
\end{equation} 
We then introduce the operator
$\ga^A_N$ (the strong Neumann trace) by 
\begin{equation}\label{2.7A}
\ga^A_N=\nu\cdot\ga_D A\nabla \colon H^{s+1}(\Om)\to \LdOm, \quad 1/2<s<3/2, 
\end{equation} 
where $\nu$ denotes the outward pointing normal unit vector to
$\partial\Om$. It follows from \eqref{2.6A} that $\ga_N$ is also a
bounded operator. We seek to define the action of the Neumann trace
operator in other (related) settings. Specifically, 
introduce the weak Neumann trace operator 
\begin{equation}\label{2.8A}
\wti\ga^A_N\colon\big\{u\in H^1(\Om)\,\big|\,L_A u\in L^2(\Om)\big\}
\to H^{-1/2}(\dOm),
\end{equation} 
as follows: Given $u\in H^1(\Om)$ with $L_A u \in L^2(\Om)$ we set  
\begin{equation}\label{2.9A}
\bigl\langle\phi,\wti\ga^A_N u\bigr\rangle_{1/2}
=\mathfrak{q}_{A}(\Phi,u)+(\Phi,L_A u)_{L^2(\Om)},
\end{equation} 
for all $\phi\in H^{1/2}(\dOm)$ and $\Phi\in H^{1}(\Om)$ such that
$\ga_D\Phi=\phi$. We note that the definition \eqref{2.9A} is independent 
of the particular extension $\Phi$ of $\phi$, and that $\wti\ga^A_N$ is bounded.

%%%%%%%%%%%%%%%%%%%%%%%%%%%%%%%%%%%%%%
%%%%%%%%%%%%%%%%%%%%%%%%%%%%%%%%%%%%%%
\section{Sesquilinear Forms and Associated Operators} \lb{sB}
\renewcommand{\theequation}{B.\arabic{equation}}
\renewcommand{\thetheorem}{B.\arabic{theorem}}
\setcounter{theorem}{0} \setcounter{equation}{0}

%%%%%%%%%%%%%%%%%%%%%%%%%%%%%%%%%%%%%%
%%%%%%%%%%%%%%%%%%%%%%%%%%%%%%%%%%%%%%

In this section we describe a few basic facts on sesquilinear forms and
linear operators associated with them.
Let $\cH$ be a complex separable Hilbert space with scalar product
$(\dott,\dott)_{\cH}$ (antilinear in the first and linear in the second
argument), $\cV$ a reflexive Banach space continuously and densely embedded
into $\cH$. Then also $\cH$ embeds continuously and densely into $\cV^*$.
That is,
\begin{equation}
\cV  \hookrightarrow \cH  \hookrightarrow \cV^*.     \lb{B.1}
\end{equation}
Here the continuous embedding $\cH\hookrightarrow \cV^*$ is accomplished via
the identification
\begin{equation}
\cH \ni v \mapsto (\dott,v)_{\cH} \in \cV^*,     \lb{B.2}
\end{equation}
and we use the convention in this manuscript that if $X$ denotes a Banach space, 
$X^*$ denotes the {\it adjoint space} of continuous  conjugate linear functionals on $X$,
also known as the {\it conjugate dual} of $X$.

In particular, if the sesquilinear form
\begin{equation}
{}_{\cV}\langle \dott, \dott \rangle_{\cV^*} \colon \cV \times \cV^* \to \bbC
\end{equation}
denotes the duality pairing between $\cV$ and $\cV^*$, then
\begin{equation}
{}_{\cV}\langle u,v\rangle_{\cV^*} = (u,v)_{\cH}, \quad u\in\cV, \;
v\in\cH\hookrightarrow\cV^*,   \lb{B.3}
\end{equation}
that is, the $\cV, \cV^*$ pairing
${}_{\cV}\langle \dott,\dott \rangle_{\cV^*}$ is compatible with the
scalar product $(\dott,\dott)_{\cH}$ in $\cH$.

Let $T \in\cB(\cV,\cV^*)$. Since $\cV$ is reflexive, $(\cV^*)^* = \cV$, one has
\begin{equation}
T \colon \cV \to \cV^*, \quad  T^* \colon \cV \to \cV^*   \lb{B.4}
\end{equation}
and
\begin{equation}
{}_{\cV}\langle u, Tv \rangle_{\cV^*}
= {}_{\cV^*}\langle T^* u, v\rangle_{(\cV^*)^*}
= {}_{\cV^*}\langle T^* u, v \rangle_{\cV}
= \ol{{}_{\cV}\langle v, T^* u \rangle_{\cV^*}}.
\end{equation}
{\it Self-adjointness} of $T$ is then defined by $T=T^*$, that is,
\begin{equation}
{}_{\cV}\langle u,T v \rangle_{\cV^*}
= {}_{\cV^*}\langle T u, v \rangle_{\cV}
= \ol{{}_{\cV}\langle v, T u \rangle_{\cV^*}}, \quad u, v \in \cV,    \lb{B.5}
\end{equation}
{\it nonnegativity} of $T$ is defined by
\begin{equation}
{}_{\cV}\langle u, T u \rangle_{\cV^*} \geq 0, \quad u \in \cV,    \lb{B.6}
\end{equation}
and {\it boundedness from below of $T$ by $c_T \in\bbR$} is defined by
\begin{equation}
{}_{\cV}\langle u, T u \rangle_{\cV^*} \geq c_T \|u\|^2_{\cH},
\quad u \in \cV.
\lb{B.6a}
\end{equation}
(By \eqref{B.3}, this is equivalent to
${}_{\cV}\langle u, T u \rangle_{\cV^*} \geq c_T \,
{}_{\cV}\langle u, u \rangle_{\cV^*}$, $u \in \cV$.)

Next, let the sesquilinear form $a(\dott,\dott)\colon\cV \times \cV \to \bbC$
(antilinear in the first and linear in the second argument) be
{\it $\cV$-bounded}, that is, there exists a $c_a>0$ such that
\begin{equation}
|a(u,v)| \le c_a \|u\|_{\cV} \|v\|_{\cV},  \quad u, v \in \cV.
\end{equation}
Then $\wti A$ defined by
\begin{equation}
\wti A \colon \begin{cases} \cV \to \cV^*, \\
\, v \mapsto \wti A v = a(\dott,v), \end{cases}    \lb{B.7}
\end{equation}
satisfies
\begin{equation}
\wti A \in\cB(\cV,\cV^*) \, \text{ and } \,
{}_{\cV}\big\langle u, \wti A v \big\rangle_{\cV^*}
= a(u,v), \quad  u, v \in \cV.    \lb{B.8}
\end{equation}
Assuming further that $a(\dott,\dott)$ is {\it symmetric}, that is,
\begin{equation}
a(u,v) = \ol{a(v,u)},  \quad u,v\in \cV,    \lb{B.9}
\end{equation}
and that $a$ is {\it $\cV$-coercive}, that is, there exists a constant
$C_0>0$ such that
\begin{equation}
a(u,u)  \geq C_0 \|u\|^2_{\cV}, \quad u\in\cV,    \lb{B.10}
\end{equation}
respectively, then,
\begin{equation}
\wti A \colon \cV \to \cV^* \, \text{ is bounded, self-adjoint, and boundedly
invertible.}    \lb{B.11}
\end{equation}
Moreover, denoting by $A$ the part of $\wti A$ in $\cH$ defined by
\begin{align}
\dom(A) = \big\{u\in\cV \,|\, \wti A u \in \cH \big\} \subseteq \cH, \quad
A= \wti A\big|_{\dom(A)}\colon \dom(A) \to \cH,   \lb{B.12}
\end{align}
then $A$ is a (possibly unbounded) self-adjoint operator in $\cH$ satisfying
\begin{align}
& A \geq C_0 I_{\cH},   \lb{B.13}  \\
& \dom\big(A^{1/2}\big) = \cV.  \lb{B.14}
\end{align}
In particular,
\begin{equation}
A^{-1} \in\cB(\cH).   \lb{B.15}
\end{equation}
The facts \eqref{B.1}--\eqref{B.15} are a consequence of the Lax--Milgram
theorem and the second representation theorem for symmetric sesquilinear forms.
Details can be found, for instance, in \cite[Sects.\ VI.3, VII.1]{DL00},
\cite[Ch.\ IV]{EE89}, and \cite{Li62}.

Next, consider a symmetric form $b(\dott,\dott)\colon \cV\times\cV\to\bbC$
and assume that $b$ is {\it bounded from below by $c_b\in\bbR$}, that is,
\begin{equation}
b(u,u) \geq c_b \|u\|_{\cH}^2, \quad u\in\cV.  \lb{B.19}
\end{equation}
Introducing the scalar product
$(\dott,\dott)_{\cV_b}\colon \cV\times\cV\to\bbC$
(and the associated norm $\|\cdot\|_{\cV_b}$) by
\begin{equation}
(u,v)_{\cV_b} = b(u,v) + (1- c_b)(u,v)_{\cH}, \quad u,v\in\cV,  \lb{B.20}
\end{equation}
turns $\cV$ into a pre-Hilbert space $(\cV; (\dott,\dott)_{\cV_b})$,
which we denote by
$\cV_b$. The form $b$ is called {\it closed} in $\cH$ if $\cV_b$ is actually
complete, and hence a Hilbert space. The form $b$ is called {\it closable}
in $\cH$ if it has a closed extension. If $b$ is closed in $\cH$, then
\begin{equation}
|b(u,v) + (1- c_b)(u,v)_{\cH}| \le \|u\|_{\cV_b} \|v\|_{\cV_b},
\quad u,v\in \cV,
\lb{B.21}
\end{equation}
and
\begin{equation}
|b(u,u) + (1 - c_b)\|u\|_{\cH}^2| = \|u\|_{\cV_b}^2, \quad u \in \cV,
\lb{B.22}
\end{equation}
show that the form $b(\dott,\dott)+(1 - c_b)(\dott,\dott)_{\cH}$ is a
symmetric, $\cV$-bounded, and $\cV$-coercive sesquilinear form. Hence,
by \eqref{B.7} and \eqref{B.8}, there exists a linear map
\begin{equation}
\wti B_{c_b} \colon \begin{cases} \cV_b \to \cV_b^*, \\
\hspace*{.51cm}
v \mapsto \wti B_{c_b} v = b(\dott,v) +(1 - c_b)(\dott,v)_{\cH},
\end{cases}
\lb{B.23}
\end{equation}
with
\begin{equation}
\wti B_{c_b} \in\cB(\cV_b,\cV_b^*) \, \text{ and } \,
{}_{\cV_b}\big\langle u, \wti B_{c_b} v \big\rangle_{\cV_b^*}
= b(u,v)+(1 -c_b)(u,v)_{\cH}, \quad  u, v \in \cV.    \lb{B.24}
\end{equation}
Introducing the linear map
\begin{equation}
\wti B = \wti B_{c_b} + (c_b - 1)\wti I \colon \cV_b\to\cV_b^*,
\lb{B.24a}
\end{equation}
where $\wti I\colon \cV_b\hookrightarrow\cV_b^*$ denotes
the continuous inclusion (embedding) map of $\cV_b$ into $\cV_b^*$, one
obtains a self-adjoint operator $B$ in $\cH$ by restricting $\wti B$ to $\cH$,
\begin{align}
\dom(B) = \big\{u\in\cV \,\big|\, \wti B u \in \cH \big\} \subseteq \cH, \quad
B= \wti B\big|_{\dom(B)}\colon \dom(B) \to \cH,   \lb{B.25}
\end{align}
satisfying the following properties:
\begin{align}
& B \geq c_b I_{\cH},  \lb{B.26} \\
& \dom\big(|B|^{1/2}\big) = \dom\big((B - c_bI_{\cH})^{1/2}\big)
= \cV,  \lb{B.27} \\
& b(u,v) = \big(|B|^{1/2}u, U_B |B|^{1/2}v\big)_{\cH}    \lb{B.28b} \\
& \hspace*{.97cm}
= \big((B - c_bI_{\cH})^{1/2}u, (B - c_bI_{\cH})^{1/2}v\big)_{\cH}
+ c_b (u, v)_{\cH}
\lb{B.28} \\
& \hspace*{.97cm}
= {}_{\cV_b}\big\langle u, \wti B v \big\rangle_{\cV_b^*},
\quad u, v \in \cV, \lb{B.28a} \\
& b(u,v) = (u, Bv)_{\cH}, \quad  u\in \cV, \; v \in\dom(B),  \lb{B.29} \\
& \dom(B) = \{v\in\cV\,|\, \text{there exists an $f_v\in\cH$ such that}  \no \\
& \hspace*{3.05cm} b(w,v)=(w,f_v)_{\cH} \text{ for all $w\in\cV$}\},
\lb{B.30} \\
& Bu = f_u, \quad u\in\dom(B),  \no \\
& \dom(B) \text{ is dense in $\cH$ and in $\cV_b$}.  \lb{B.31}
\end{align}
Properties \eqref{B.30} and \eqref{B.31} uniquely determine $B$.
Here $U_B$ in \eqref{B.28} is the partial isometry in the polar
decomposition of $B$, that is,
\begin{equation}
B=U_B |B|, \quad  |B|=(B^*B)^{1/2} \geq 0.   \lb{B.32}
\end{equation}
The operator $B$ is called the {\it operator associated with the form $b$}.

The facts \eqref{B.19}--\eqref{B.32} comprise the second representation
theorem of sesquilinear forms (cf.\ \cite[Sect.\ IV.2]{EE89},
\cite[Sects.\ 1.2--1.5]{Fa75}, and \cite[Sect.\ VI.2.6]{Ka80}).

%%%%%%%%%%%%%%%%%%%%%%%%%%%%%%%%%%%%%%
%%%%%%%%%%%%%%%%%%%%%%%%%%%%%%%%%%%%%%
\section{On Heat Kernel and Green's Function Bounds} \lb{sC}
\renewcommand{\theequation}{C.\arabic{equation}}
\renewcommand{\thetheorem}{C.\arabic{theorem}}
\setcounter{theorem}{0} \setcounter{equation}{0}

%%%%%%%%%%%%%%%%%%%%%%%%%%%%%%%%%%%%%%
%%%%%%%%%%%%%%%%%%%%%%%%%%%%%%%%%%%%%%

In this appendix we briefly recall some bounds for heat kernels and Green's functions and 
prove the Green's function bounds as the latter in dimension $n=2$ are difficult to find in the 
literature.  

In the case of Dirichlet boundary conditions (for general open sets 
$\Om \subseteq \bbR^n$), $e^{- t (- \Delta_{D, \Omega})}$ is known to be {\it ultracontractive}, that is, a bounded map from 
$L^2(\Om; d^n x)$ into $L^\infty(\Om; d^n x)$, and to have a nonnegative integral kernel satisfying for 
all $(x,y) \in \Om \times \Om$, and all $t>0$, 
\begin{equation}
0 \leq e^{- t (- \Delta_{D, \Omega})}(x,y) \leq e^{- t H_0}(x,y) 
= (4 \pi t)^{-n/2} e^{- |x - y|^2/(4t)} \leq (4 \pi t)^{- n/2}    \lb{3.16}
\end{equation}
due to domain monotonicity, as discussed in 
\cite[Lemma\ 2.1.2, Example\ 2.1.8, Theorems\ 2.1.6 and 2.3.6]{Da89}. 
Here $H_0$ denotes the self-adjoint realization of $-\Delta$ in $L^2(\bbR^n; d^n x)$, 
\begin{equation}
H_0 = - \Delta, \quad \dom(H_0) = H^2(\bbR^n).  
\end{equation} 
In addition, assuming Hypothesis \ref{h4.1}, there exists a constant $c_{\gamma,a_0}>0$, such 
that for all $(x,y) \in \Om \times \Om$, and all $\gamma \in (0,1)$, $t > 0$, one has the bound 
\begin{equation}
0 \leq e^{- t L_{D, \Omega}}(x,y) \leq c_{\gamma,a_0} t^{- n/2} 
e^{- |x - y|^2/[4(1+\gamma) a_1t]}, 
\lb{3.16A} 
\end{equation}     
by \cite[Corollary\ 3.2.8]{Da89}. In particular, $e^{- t L_{D, \Omega}}$ is ultracontractive. 

Similarly, assuming again Hypothesis \ref{h4.1}, also 
$e^{- t L_{N, \Omega}}$ is known to be ultracontractive and to have a nonnegative integral 
kernel satisfying for some constant $C_{\gamma,a_0,\Om}>0$, 
all $(x,y) \in \Om \times \Om$, and all $\gamma \in (0,1)$, $t > 0$,  
\begin{equation}
0 \leq e^{- t L_{N, \Omega}}(x,y) \leq C_{\gamma, a_0, \Om} \max (t^{- n/2}, 1) 
e^{- |x - y|^2/[4(1+\gamma) a_1 t]},      \lb{3.16a} 
\end{equation}     
by \cite[Theorems\ 1.3.9, 2.4.4, and 3.2.9]{Da89} (see also the more 
recent \cite{CWZ94}, \cite[Theorem\ 4.4]{AtE97}, \cite{Da00a}, \cite{Ou04}, \cite[Ch.\ 6]{Ou05}). 
We note that the Lipschitz domain hypothesis is crucial in connection with the estimate 
\eqref{3.16a} as it has to be modified in less regular domains characterized by 
H\"older regularity of $\partial\Om$, as discussed in \cite[Proposition\ 3]{BD02} (and the 
references cited in this context).  

For completeness, we also mention the explicit (but rather crude) Green's function estimate in the 
context of the Dirichlet Laplacian and for $L_{D,\Om}$, based on domain monotonicity discussed, for instance, in 
\cite{AtE97}, \cite{Da87}, \cite[Ch.\ 3]{Da89}, \cite{vdB90} (see also 
\cite[Sect.\ 1.2]{Ke94}, \cite{vdB90}), and the references therein: For all 
$(x,y) \in \Om \times \Om$, and all $\lambda > 0$, 
\begin{align}
& G^{(0)}_{D,\Omega}(-\lambda,x,y) = (- \Delta_{D,\Om} + \lambda I_{\Om})^{-1}(x,y)   \no \\
& \quad \leq G_0 (-\lambda, x,y) = (H_0 + \lambda I_{\bbR^n})^{-1}(x,y)     \no \\
& \quad = \f{1}{2 \pi} \bigg(\f{2\pi |x - y|}{\lambda^{1/2}}\bigg)^{(2-n)/2} 
K_{(n-2)/2} (\lambda^{1/2} |x - y|)     \no \\
& \quad \leq \begin{cases} C_{\lambda,\Omega,n} |x - y|^{2-n}, & n \geq 3, \\[1mm]
C_{\lambda,\Omega} \big|\ln\big(1 + |x-y|^{-1}\big)\big|, & n=2,  \end{cases} 
\quad  x \neq y,     \lb{3.17}
\end{align}
with $K_{\nu}(\cdot)$ the modified irregular Bessel function of order $\nu$ 
(cf.\ \cite[Sect.\ 9.6]{AS72}). Similarly, in connection with $L_{D,\Om}$ 
one obtains, for all $(x,y) \in \Om \times \Om$, and $\lambda > 0$, 
\begin{equation}
G_{D,\Omega}(-\lambda,x,y) \leq 
\begin{cases} C_{\lambda,\Omega,n} |x - y|^{2-n}, & n \geq 3, \\[1mm]
C_{\lambda,\Omega} \big|\ln\big(1 + |x-y|^{-1}\big)\big|, & n=2,  \end{cases} 
\quad  x \neq y.      \lb{3.17a}
\end{equation}
The estimates \eqref{3.17} and \eqref{3.17a} ignore all effects of 
the boundary $\partial\Omega$ of $\Omega$, but they suffice for the purpose at hand.  

Finally, we explicitly derive the analog of \eqref{3.17a} in the case of Neumann boundary 
conditions as this is not so simple to find in the literature (particularly for $n=2$).

%%%%%%%%%%%%
\begin{lemma} \lb{C.1} 
Assume Hypothesis \ref{h4.1}. Then  
\begin{align} 
\begin{split}
& G_{N,\Om} (- \lambda,x,y) \leq \begin{cases} 
C_{a_0,a_1,\lambda,\Omega,n} |x - y|^{2-n}, & n \geq 3, \\[1mm]
C_{a_0,a_1,\lambda,\Omega} \big|\ln\big(1 + |x - y|^{-1}\big)\big|, & n=2,  \end{cases}   \\[1mm] 
& \hspace*{4.95cm}  \lambda > 0, \; x, y \in \Omega, \; x \neq y.       \lb{C.6}
\end{split}
\end{align}
More generally, let $\alpha \in (0, (n/2)]$, then
\begin{align} 
\begin{split}
& (H_{N,\Om} + \lambda I_{\Om})^{-\alpha} (x,y) \leq \begin{cases} 
C_{a_0,a_1,\alpha,\lambda,\Omega,n} |x - y|^{2\alpha - n}, & \alpha \in (0, (n/2)), \\[1mm]
C_{a_0,a_1,\alpha,\lambda,\Omega} \big|\ln\big(1 + |x - y|^{-1}\big)\big|, & \alpha=n/2,  \end{cases}   \\[1mm] 
& \hspace*{7.75cm}  \lambda > 0, \; x, y \in \Omega, \; x \neq y.       \lb{C.7}
\end{split}
\end{align}
\end{lemma}
%%%%%%%%%%%
\begin{proof}
The main ingredients in deriving the bounds \eqref{C.6} and \eqref{C.7} are the identity 
\begin{equation}
(H + \lambda I_{\cH})^{-\alpha} = \Gamma(\alpha)^{-1} \int_0^{\infty} dt \, t^{\alpha -1} e^{-t (H + \lambda)}, 
\quad \alpha > 0, \; \lambda > 0,    \lb{C.8} 
\end{equation}
for a given self-adjoint operator $H \geq 0$, and the integral representation (see, e.g., \cite[\# 84327, p.\ 959]{GR80}) 
\begin{equation}
\int_0^{\infty} dt \, t^{- \nu - 1} e^{- at - b t^{-1}} = 2 (b/a)^{-\nu/2} K_{\nu}\big(2(ab)^{1/2}\big), \quad 
\nu \in \bbR, \; a>0, \; b>0.    \lb{C.9} 
\end{equation}
The asymptotic behavior of $K_{\nu}(\cdot)$ (cf., \cite[p.\ 375, 378]{AS72})
\begin{align}
& K_{\nu}(x) \underset{x\downarrow 0}{=} \begin{cases} 2^{\nu-1} 
\Gamma(\nu) x^{-\nu}\big[1+\Oh(\big(x^2\big)\big], & \nu > 0, \\[1mm] 
- [\ln(x/2) + C_E]\big[1+\Oh(\big(x^2\big)\big], & \nu = 0, \end{cases}     \lb{C.10} \\[2mm]
& K_{\nu}(x) \underset{x\uparrow \infty}{=} (\pi/2)^{1/2} x^{-1/2} e^{-x} \big[1+\Oh(\big(x^{-1}\big)\big], 
\quad \nu \geq 0,    \lb{C.11} 
\end{align}
with $C_E$ denoting Euler's constant (see, \cite[p.\ 255]{AS72}), then implies the bounds 
\begin{equation}
K_\nu(x) \leq \begin{cases} C \big[1 + \Gamma(\nu) 2^{\nu -1} x^{-\nu}\big] 
\big[1 + (2x/\pi)^{1/2}\big]^{-1} e^{-x}, &\nu > 0, \\[1mm]
C \big[\ln\big(1 + x^{-1}\big)\big] \big[1 + (2x/\pi)^{1/2}\big]^{-1} e^{-x}, & \nu =0, 
\end{cases} \quad x>0,    \lb{C.12} 
\end{equation} 
for some universal constant $C>0$ in either case. Hence, combining \eqref{C.8} and \eqref{C.9} 
with the heat kernel bound \eqref{3.16a}, and using the crude estimate,
\begin{equation}
\max (t^{- n/2}, 1) \leq c_{\varepsilon} t^{-n/2} e^{\varepsilon t}, \quad \varepsilon >0, \;t>0, 
\end{equation}
for some constant $c_{\varepsilon} > 0$, then yields 
\begin{align}
& (H_{N,\Om} + \lambda I_{\Om})^{-\alpha} (x,y)   \no \\
& \quad \leq \Gamma(\alpha)^{-1} 
C_{\gamma, a_0,\Om,\varepsilon} \int_0^{\infty} dt \, t^{\alpha - 1 - (n/2)} 
e^{\varepsilon t} 
e^{-\lambda t} e^{- |x - y|^2 [4 (1 + \gamma) a_1 t]^{-1}}    \no \\
& \quad \leq C_{\gamma, a_0,a_1,\Om, \alpha, \varepsilon} 
\bigg[\f{|x - y|^2}{\lambda - \varepsilon}\bigg]^{(2\alpha -n)/4} 
K_{(n-2\alpha)/2} \big(((1+\gamma) a_1)^{-1/2}(\lambda - \varepsilon)^{1/2} |x - y| \big),  \no \\
& \quad \leq \begin{cases} 
C_{a_0,a_1,\alpha,\lambda,\Omega,\varepsilon,n} |x - y|^{2\alpha - n}, & \alpha \in (0, (n/2)), \\[1mm]
C_{a_0,a_1,\alpha,\lambda,\Omega,\varepsilon} \big|\ln\big(1 + |x - y|^{-1}\big)\big|, & \alpha=n/2,  
\end{cases}   \no \\[1mm] 
& \hspace*{2.7cm} \lambda > \varepsilon, \; \alpha \in (0, (n/2)], \; x, y \in \Om, \; x \neq y. 
\end{align}
Here the last inequality follows from \eqref{C.12} and the fact that $\Om$ is bounded. 
Since $\varepsilon > 0$ can be chosen arbitrarily small, this proves \eqref{C.6} and \eqref{C.7}.
\end{proof}
%%%%%%%%%%%

We note again that the bound \eqref{C.6} for $n\geq 3$ can be found, for instance, in \cite[Lemma\ 3.2]{DM96} (see also \cite{AKK12}, \cite{CK11}, \cite{TKB12}). 

We emphasize that the basic ingredients we used to derive \eqref{C.6} and \eqref{C.7} appear, for 
instance, in \cite[Sect.\ 3.4]{Da89}. Moreover, the identical proof also yields the crude Dirichlet bound 
\eqref{3.17a} (taking $\varepsilon =0$ throughout). In our opinion, the value of presenting Lemma \ref{C.1} (besides establishing 
the bound for $n=2$) lies in providing an explicit reference for the result and the ease with which it 
is derived. 

\medskip

%%%%%%%%%%%%%%%%%%%%%%%%%%%%%%%%%%%%%
{\bf Acknowledgments.} We are indebted to Zhen-Qing Chen, E.\ Brian Davies, Joe Diestel, 
Chris Evans, Harald Hanche-Olsen, Steve Hofmann, Seick Kim, Vladimir Maz'ya, El Maati Ouhabaz, 
Michael Pang, Derek Robinson, Barry Simon, and Qi Zhang for very helpful correspondence and 
discussions. Special thanks are due to Seick Kim for making available unpublished work 
to us, to El Maati Ouhabaz for his detailed comments, especially, on Section \ref{s4}, to Michael Pang for his help with the material in Appendix \ref{sC}, and to Derek Robinson for supplying us with very helpful comments and corrections 
for our manuscript. 
%%%%%%%%%%%%%%%%%%%%%%%%%%%%%%%%%%%%%

%%%%%%%%%%%%%%%%%%%%%%%%%%%%%%%%
%%%%%%%%%%%%%%%%%%%%%%%%%%%%%%%%

\end{document}